\documentclass[11pt,a4paper]{article}
\pdfoutput=1
\usepackage{etoolbox}
\newbool{production}
\setbool{production}{true}
\usepackage{preamble} 
\title{Grounded persistent path homology: a stable, topological descriptor for weighted digraphs}
\author[1,*]{Thomas Chaplin}
\author[1]{Heather A. Harrington}
\author[1]{Ulrike Tillmann}
\affil[1]{Mathematical Institute, University of Oxford}
\affil[*]{%
  Corresponding author: {%
  \normalfont\texttt{%
  \href{mailto:thomas.chaplin@maths.ox.ac.uk}{thomas.chaplin@maths.ox.ac.uk}%
}}}
\begin{document}
\maketitle
\begin{abstract}
Weighted digraphs are  used to model a variety of natural systems and can exhibit interesting structure across a range of scales.
In order to understand and compare these systems, we require stable, interpretable, multiscale descriptors.
To this end, we propose grounded persistent path homology (\grpph) -- a new, functorial, topological descriptor that describes the structure of an edge-weighted digraph via a persistence barcode.
We show there is a choice of circuit basis for the graph which yields geometrically interpretable representatives for the features in the barcode.
Moreover, we show the barcode is stable, in bottleneck distance, to both numerical and structural perturbations.
\grpph\ arises from a flexible framework, parametrised by a choice of digraph chain complex and a choice of filtration; for completeness, we also investigate replacing the path homology complex, used in \grpph, by the directed flag complex.
\end{abstract}

\section{Introduction}\label{sec:intro}

Directed graphs with positive edge-weights arise both as natural objects of mathematical study and as useful models of real-world systems (e.g.~\cite{bittner2018comparing,Brown2022,Medaglia2017,Sweeney2019}).
A common task is to distinguish between weighted digraphs.
Frequently, this is achieved by defining an invariant, i.e. a map $\invariant:\WDgr\to X$ from weighted digraphs into some set $X$, together with a metric $d$ on $X$.
The metric allows us to quantitatively measure to what extent a pair of weighted digraphs differ.
When $\invariant$ is well understood we may be able to explain \emph{why} the weighted digraphs differ.
In order to determine desirable characteristics of such an $\invariant$, consider the examples shown in Figure~\ref{fig:intro_wdgrs}, in which edge weights correspond to length as drawn.

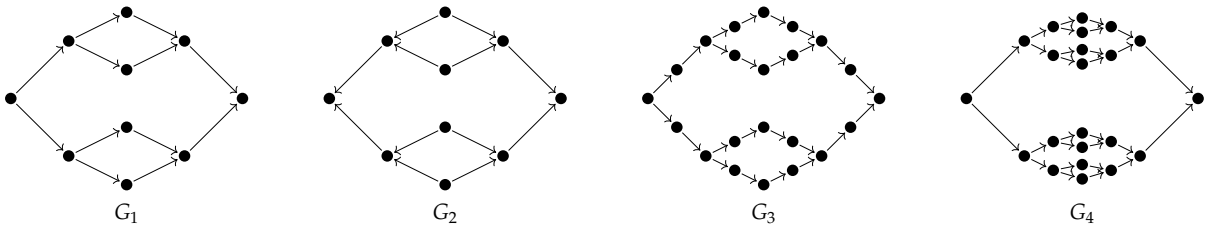
\begin{figure}[htbp]
  \centering
  \resizebox{\textwidth}{!}{%
    \begin{tikzpicture}[
  roundnode/.style={circle, fill=black, minimum size=4pt},
  bluenode/.style={circle, fill=blue, minimum size=1pt},
	squarenode/.style={fill=black, minimum size=4pt},
	inner sep=2pt,
	outer sep=1pt
  ]
  \node (a) at (0, 0) [roundnode] {};
  \node (b) at (1, 1) [roundnode] {};
  \node (c) at (1, -1) [roundnode] {};
  \node (d) at (2, 1.5) [roundnode] {};
  \node (e) at (2, 0.5) [roundnode] {};
  \node (f) at (2, -0.5) [roundnode] {};
  \node (g) at (2, -1.5) [roundnode] {};
  \node (h) at (3, 1) [roundnode] {};
  \node (i) at (3, -1) [roundnode] {};
  \node (j) at (4, 0) [roundnode] {};
  \draw[->] (a)--(b);
  \draw[->] (a)--(c);
  \draw[->] (b)--(d);
  \draw[->] (b)--(e);
  \draw[->] (c)--(f);
  \draw[->] (c)--(g);
  \draw[->] (d)--(h);
  \draw[->] (e)--(h);
  \draw[->] (f)--(i);
  \draw[->] (g)--(i);
  \draw[->] (h)--(j);
  \draw[->] (i)--(j);

  \node[] at (2, -2) {$G_1$};

  \node (a2) at (5.5, 0) [roundnode] {};
  \node (b2) at (6.5, 1) [roundnode] {};
  \node (c2) at (6.5, -1) [roundnode] {};
  \node (d2) at (7.5, 1.5) [roundnode] {};
  \node (e2) at (7.5, 0.5) [roundnode] {};
  \node (f2) at (7.5, -0.5) [roundnode] {};
  \node (g2) at (7.5, -1.5) [roundnode] {};
  \node (h2) at (8.5, 1) [roundnode] {};
  \node (i2) at (8.5, -1) [roundnode] {};
  \node (j2) at (9.5, 0) [roundnode] {};
  \draw[->] (b2)--(a2);
  \draw[->] (c2)--(a2);

  \draw[->] (d2)--(b2);
  \draw[->] (e2)--(b2);
  \draw[->] (f2)--(c2);
  \draw[->] (g2)--(c2);
  \draw[->] (d2)--(h2);
  \draw[->] (e2)--(h2);
  \draw[->] (f2)--(i2);
  \draw[->] (g2)--(i2);

  \draw[->] (h2)--(j2);
  \draw[->] (i2)--(j2);

  \node[] at (7.5, -2) {$G_2$};

  \node (a3) at (11, 0) [roundnode] {};
  \node (abmid3) at (11.5, 0.5) [roundnode] {};
  \node (acmid3) at (11.5, -0.5) [roundnode] {};
  \node (b3) at (12, 1) [roundnode] {};
  \node (bdmid3) at (12.5, 1.25) [roundnode] {};
  \node (bemid3) at (12.5, 0.75) [roundnode] {};
  \node (c3) at (12, -1) [roundnode] {};
  \node (cfmid3) at (12.5, -0.75) [roundnode] {};
  \node (cgmid3) at (12.5, -1.25) [roundnode] {};
  \node (d3) at (13, 1.5) [roundnode] {};
  \node (dhmid3) at (13.5, 1.25) [roundnode] {};
  \node (e3) at (13, 0.5) [roundnode] {};
  \node (ehmid3) at (13.5, 0.75) [roundnode] {};
  \node (f3) at (13, -0.5) [roundnode] {};
  \node (fimid3) at (13.5, -0.75) [roundnode] {};
  \node (g3) at (13, -1.5) [roundnode] {};
  \node (gimid3) at (13.5, -1.25) [roundnode] {};
  \node (h3) at (14, 1) [roundnode] {};
  \node (hjmid3) at (14.5, 0.5) [roundnode] {};
  \node (i3) at (14, -1) [roundnode] {};
  \node (ijmid3) at (14.5, -0.5) [roundnode] {};
  \node (j3) at (15, 0) [roundnode] {};
  \draw[->] (a3)--(abmid3);
  \draw[->] (abmid3)--(b3);
  \draw[->] (a3)--(acmid3);
  \draw[->] (acmid3)--(c3);
  \draw[->] (b3)--(bdmid3);
  \draw[->] (bdmid3)--(d3);
  \draw[->] (b3)--(bemid3);
  \draw[->] (bemid3)--(e3);
  \draw[->] (c3)--(cfmid3);
  \draw[->] (cfmid3)--(f3);
  \draw[->] (c3)--(cgmid3);
  \draw[->] (cgmid3)--(g3);
  \draw[->] (d3)--(dhmid3);
  \draw[->] (dhmid3)--(h3);
  \draw[->] (e3)--(ehmid3);
  \draw[->] (ehmid3)--(h3);
  \draw[->] (f3)--(fimid3);
  \draw[->] (fimid3)--(i3);
  \draw[->] (g3)--(gimid3);
  \draw[->] (gimid3)--(i3);
  \draw[->] (h3)--(hjmid3);
  \draw[->] (hjmid3)--(j3);
  \draw[->] (i3)--(ijmid3);
  \draw[->] (ijmid3)--(j3);

  \node[] at (13, -2) {$G_3$};

  \node (a4) at (16.5, 0) [roundnode] {};
  \node (b4) at (17.5, 1) [roundnode] {};
  \node (bdmid4) at (18, 1.25) [roundnode] {};
  \node (bemid4) at (18, 0.75) [roundnode] {};
  \node (c4) at (17.5, -1) [roundnode] {};
  \node (cfmid4) at (18, -0.75) [roundnode] {};
  \node (cgmid4) at (18, -1.25) [roundnode] {};
  \node (d4) at (18.5, 1.4) [roundnode] {};
  \node (dhmid4) at (19, 1.25) [roundnode] {};
  \node (e4) at (18.5, 0.6) [roundnode] {};
  \node (ehmid4) at (19, 0.75) [roundnode] {};
  \node (f4) at (18.5, -0.6) [roundnode] {};
  \node (fimid4) at (19, -0.75) [roundnode] {};
  \node (g4) at (18.5, -1.4) [roundnode] {};
  \node (gimid4) at (19, -1.25) [roundnode] {};
  \node (h4) at (19.5, 1) [roundnode] {};
  \node (i4) at (19.5, -1) [roundnode] {};
  \node (j4) at (20.5, 0) [roundnode] {};
  \node(new14) at (18.5, 1.15) [roundnode] {};
  \node(new24) at (18.5, 0.85) [roundnode] {};
  \node(new34) at (18.5, -0.85) [roundnode] {};
  \node(new44) at (18.5, -1.15) [roundnode] {};
  \draw[->] (a4)--(b4);
  \draw[->] (a4)--(c4);
  \draw[->] (b4)--(bdmid4);
  \draw[->] (bdmid4)--(d4);
  \draw[->] (b4)--(bemid4);
  \draw[->] (bemid4)--(e4);
  \draw[->] (c4)--(cfmid4);
  \draw[->] (cfmid4)--(f4);
  \draw[->] (c4)--(cgmid4);
  \draw[->] (cgmid4)--(g4);
  \draw[->] (d4)--(dhmid4);
  \draw[->] (dhmid4)--(h4);
  \draw[->] (e4)--(ehmid4);
  \draw[->] (ehmid4)--(h4);
  \draw[->] (f4)--(fimid4);
  \draw[->] (fimid4)--(i4);
  \draw[->] (g4)--(gimid4);
  \draw[->] (gimid4)--(i4);
  \draw[->] (h4)--(j4);
  \draw[->] (i4)--(j4);
  \draw[->] (bdmid4)--(new14);
  \draw[->] (bemid4)--(new24);
  \draw[->] (new14)--(dhmid4);
  \draw[->] (new24)--(ehmid4);
  \draw[->] (cfmid4)--(new34);
  \draw[->] (cgmid4)--(new44);
  \draw[->] (new34)--(fimid4);
  \draw[->] (new44)--(gimid4);

  \node[] at (18.5, -2) {$G_4$};

\end{tikzpicture}
  }
  \caption{Four example weighted digraphs; weights correspond to length as drawn.}%
  \label{fig:intro_wdgrs}
\end{figure}

Consider each $G_i$ from the perspective of a particle flowing through the digraph, such that the particle may only traverse an edge in the direction specified and the time it takes corresponds to the weight.
To the particle, loops (or circuits) in the graph are significant features.
However, loops can vary greatly based on the orientation and weight of constituent edges.

Despite sharing the same underlying undirected graph, $G_1$ and $G_2$ support very different flows since $G_1$ has a single source and a single sink whereas $G_2$ has 4 sources and 2 sinks.
To reflect this, $d(\invariant(G_1), \invariant(G_2))$ should be large.
In contrast, $G_3$ has a different undirected graph but can be obtained from $G_1$ by simply subdividing each edge.
In applications, this may arise from a finer resolution image of the same system.
A suitable invariant should be relatively stable to such subdivisions, ideally converging to a limiting value upon iterated subdivision.
Finally, $G_4$ has a higher circuit rank but the new loops are on a small scale, whilst the large scale organisation is mostly similar to $G_1$.
Therefore, $d(\invariant(G_1), \invariant(G_4))$ should be small and the difference between $\invariant(G_1)$ and $\invariant(G_4)$ should reflect this multiscale comparison.

For successful application, any invariant should be stable to a reasonable noise model.
A typical requirement is that $\invariant$ is continuous (or better yet Lipschitz), with respect to a choice of metric on $\WDgr$.
Designing metrics for graphs is an active area of research but a common choice is the graph edit distance~\cite{Gao2010}. 
For this metric, costs are assigned to operations such as deleting an edge or modifying a weight, then the distance between two graphs is the minimal cumulative cost of modifying one into the other.
Since assigning costs to graph operations is somewhat arbitrary, it is reasonable instead to require a bound on $d(\invariant(G), (\invariant(\wdop{}{}{G}))$, over a range of graph operations, $\wdop{}{}{}:\WDgr\to\WDgr$.

Finally, in many applications (particularly in biology), it is important that any invariant $\invariant$ is interpretable.
That is, one must be able to explain \emph{why} the invariant has the value it does.
Typically this is achieved through the identification of key contributing subgraphs.

In summary, we seek an invariant for weighted digraphs which
\begin{enumerate}[label=(\alph*)]
  \item distinguishes graphs with different flow profiles due to directionality;
  \item can detect and describe features (i.e. loops) across a range of scales;
  \item is stable to reasonable perturbations and converges under iterated subdivision; and
  \item is interpretable, e.g.\ through the identification of important subgraphs.
\end{enumerate}

\vskip .2in

Pursuant to these goals, we employ two tools from topological data analysis (TDA) -- path homology and persistent homology.
A number of homology theories for digraphs have been developed (see summary in Section 2 of~\cite{caputi2022hochschild}).
Path homology is one such theory~\cite{Grigoryan2012}, which is sensitive to directionality and has useful functorial properties~\cite{ChowdhurySIAM, Grigoryan2014}.
Persistent homology is a tool for developing stable descriptors~\cite{Bauer2013a}
that extract relevant information in multiscale scenarios.
As such, persistent homology has seen successful applications to fields including
neuroscience~\cite{Caputi2021, Goodbrake2022, Govc2021, Xing2022},
vasculature~\cite{Nardini2021, Stolz2020}
and financial networks~\cite{Ismail2022},
to name but a few~\cite{giuntuTDA-app}.
A theory of persistent path homology (PPH) was proposed by \citeauthor{ChowdhurySIAM}~\cite{ChowdhurySIAM} and is a stable descriptor for directed networks.
In a search to develop an interpretable invariant for weighted digraphs, which respects the inert topology of the underlying digraph, we are lead to an alteration of PPH which we prove meets goals (a)-(d).

\subsection{Contributions and outline}

In Section~\ref{sec:background} we give an overview of path homology and persistent homology, and set up the categorical framework for the rest of the paper.
In particular, we define a category of weighted digraphs $\Cont\WDgr$ in which morphisms are digraph maps of the underlying digraphs as well as contractions of the natural, shortest-path quasimetric.

In Section~\ref{sec:intro_subdivision} we review a standard pipeline for extracting a topological invariant of a weighted digraph, via PPH.
Evaluating this invariant against our stated goals, motivates an alteration of this pipeline, which we call the `grounded pipeline'.
We define this new pipeline and describe categories upon which the resulting invariant is functorial in Section~\ref{sec:descriptor_defin}.

Arising from the grounded pipeline, our main contribution is Definition/Theorem~\ref{defin:grpph}, wherein we define grounded persistent path homology (\grpph).
This new invariant is a functor
\begin{equation}
\zbpershommap : \Cont\WDgr \to \PersVec
\end{equation}
where $\PersVec$ is the category of persistent vector spaces.
Couching this definition in category theory yields a strong framework for comparing weighted digraphs through the invariant.
Indeed, we use this functoriality later in the paper to aid the proof of decomposition and stability results.

Section~\ref{sec:properties} is devoted to developing an interpretation of \grpph.
The early subsections are dedicated to understanding the features detected by $\zbpershom{G}$; the following theorem summarises our findings.
\begin{theorem}\label{thm:intro_representatives}
Given a weighted digraph $G\in\WDgr$, denote the underlying undirected graph by $\underlying{G}$.
\begin{enumerate}[label=(\alph*)]
  \item All features in $\zbpershom{G}$ are born at $t=0$;
  \item $\zbpershom{G}$ at time $t=0$ is the (real) cycle space of $\underlying{G}$; and moreover
  \item there is a choice of circuits in $\underlying{G}$ whose homology classes generate $\zbpershom{G}$.
\end{enumerate}
\end{theorem}
These results demonstrate how \grpph\ is sensitive to circuits in the digraph at all scales, meeting goal (b), and can be interpreted through a persistence basis of such circuits, meeting goal (d).
We also discuss how to use $\zbpershom{G}$ to assign a `scale' to any circuit in $\underlying{G}$ in Section~\ref{sec:comparison}.
In Section~\ref{sec:decomposition}, we prove that if $G$ can be decomposed into smaller parts, then \grpph\ also decomposes.
\begin{theorem}\label{thm:intro_decomposition}
Given a weighted digraph $G\in\WDgr$, if $G$ decomposes as a wedge decomposition $G=G_1 \vee_{\hat{v}} G_2$ or a disjoint union $G = G_1 \sqcup G_2$ then
\begin{equation}
\zbpershom{G} \cong \zbpershom{G_1} \oplus \zbpershom{G_2}. 
\end{equation}
\end{theorem}

In Section~\ref{sec:stability} we investigate the stability of $\zbpershommap$.
We prove a number of bounds on the bottleneck distance of the barcode upon perturbing the input weighted digraph.
We find local stability to operations such as weight perturbation, edge subdivision and certain classes of edge collapses and edge deletions.
In particular, edge subdivision stability automatically implies that our invariant converges under iterated subdivision (Corollary~\ref{cor:gen_ims_convergence}).
In contrast, the descriptor is unstable to generic edge collapses and edge deletions, which we demonstrate through a number of counter-examples.
We argue that these stability properties suffice to meet goal (c) and indeed stability to larger classes of edge collapses and edge deletion would be undesirable.
For a summary of all stability results obtained, please consult Table~\ref{tbl:stability_result_summary}.

In order to build intuition for what is measured by \grpph, we compute a number of illustrative examples in Section~\ref{sec:examples}.
In particular, in Section~\ref{sec:asymptotic}, we consider a simple, cycle graph and determine the limiting value of \grpph\ under iterated edge subdivision.
In Section~\ref{sec:square_motifs}, we compute the invariant for a number of small square digraphs with varying edge orientations, illustrating sensitivity to directionality, as required by goal (a).

\grpph\ arises from the grounded pipeline after fixing a number of choices, including path homology as a functor from digraphs to chain complexes.
Another popular method for producing chain complexes from digraphs is the directed flag complex~\cite{Masulli2016}.
Replacing path homology with directed flag complex homology yields an alternative descriptor, called grounded persistent directed flag complex homology (\grpdflh).
For completeness, in Appendix~\ref{appdx:complex}, we revisit each of the results obtained in the main paper, replacing \grpph\ with \grpdflh.
Of particular note is Example~\ref{ex:dflag_appendage_instability}, in which we show that \grpdflh\ is sensitive to a particularly simple class of edge deletions.
This stands in contrast to \grpph, which is unaffected by these deletions.

\subsection{Computations}
The software package \texttt{Flagser} allows the user to flexibly define filtrations of directed flag complexes and subsequently compute persistent homology~\cite{Luetgehetmann2020}.
We use an alteration of \texttt{Flagser} to compute grounded persistent directed flag complex homology (available at~\cite{flagser-fork}).
An algorithm for computing PPH in arbitrary degrees was proposed by \citeauthor{ChowdhurySIAM}~\cite{ChowdhurySIAM};
a more efficient algorithm for computing PPH in degree $1$ was later proposed by \citeauthor{Dey2020a}~\cite{Dey2020a}.
A slight modification of the latter algorithm can be used to compute \grpph.
A software package for computing both grounded homologies, as well as persistence bases, in collaboration with G. Henselman-Petrusek, is forthcoming.

\subsection{Acknowledgements}
The first author would like to thank H.~Byrne, A.~Goriely, A.~\'O~hEachteirn and T.~Thompson for valuable discussions which motivated and aided the early stages of this work. HAH gratefully acknowledges funding from a Royal Society University Research Fellowship.
The authors are members of the Centre for Topological Data Analysis, which is funded by the EPSRC grant `New Approaches to Data Science: Application Driven Topological Data Analysis' \href{https://gow.epsrc.ukri.org/NGBOViewGrant.aspx?GrantRef=EP/R018472/1}{\texttt{EP/R018472/1}}.
For the purpose of Open Access, the authors have applied a CC BY public copyright licence to any  Author Accepted Manuscript (AAM) version arising from this submission. 

\section{Background}\label{sec:background}
\subsection{Basic notation and category theory}

We introduce some basic language for graphs and categories, and then recall the definitions of  path homology and persistent homology, the two theories we combine later  in a new way to define our  invariant.

\begin{notation}
  For $d\in \N$, the \mdf{standard $d$-simplex} is
    \begin{equation}
      \mdf{\Delta^d}\defeq \left\{ (x_0, \dots x_d) \in \R^{d+1} \rmv \sum x_i=1, \text{ and }x_i\geq 0 \; \forall i\right\}.
    \end{equation}
\end{notation}

\begin{notation}
  Given a category $\C$, we denote the collection of objects \mdf{$\Obj(\C)$} and the collection of morphisms \mdf{$\Mor(\C)$}.
  For two objects $X, Y\in\Obj(\C)$, we denote the collection of morphisms $X \to Y$ by \mdf{$\MorXY{\C}{X}{Y}$}.
  Where it is clear from object whether $\alpha$ is an object or a morphism, we simply write $\alpha \in \C$.
\end{notation}

\begin{defin}
  Given two categories $\C$ and $\D$ we denote the category of functors $\C\to\D$, where morphisms are natural transformations, by \mdf{$\Funct{\C}{\D}$}.
  For a morphism $f\in\MorXY{\Funct{\C}{\D}}{M}{N}$, we denote the components of the natural transformation by $f_x: M(x) \to N(x)$ for each $x\in\C$.
\end{defin}

\begin{lemma}
Given categories $\C$, $\D$ and $\D'$ and a functor $\mu: \D \to \D'$, there is a functor $\Funct{\C}{\mu}:\Funct{\C}{\D}\to\Funct{\C}{\D'}$.
\end{lemma}
\begin{proof}
Given $F\in\Obj(\Funct{\C}{\D})$, we map $\Funct{\C}{\mu}(F)\defeq \mu \circ F$.
Given a natural transformation $\nu: F \Rightarrow F'$ between functors $F, F'\in\Funct{C}{D}$, $\mu \circ \nu$ is a natural transformation $\mu \circ F \Rightarrow \mu \circ F'$.
This construction satisfies the usual functorial axioms.
\end{proof}

\begin{notation}
  \begin{enumerate}[label=(\alph*)]
    \item We let \mdf{$\Rposet$} denote the poset of $\R$ equipped with the $\leq$ relation, viewed as a category. 
    \item We let \mdf{$\VecCat$} denote the category of $\R$-vector spaces and $\vecCat$ denote the subcategory of finite-dimensional $\R$-vector spaces.
    \item We let \mdf{$\Ch$} denote the category of chain complexes over $\R$.
  \end{enumerate}
\end{notation}

\begin{defin}
Given a chain complex $C_\bullet\in \Ch$, we denote the $k^{th}$ chain group by $C_k$ and the boundary map by $\bd_k : C_k \to C_{k-1}$.
For each $k\in \N$, the \mdf{$k^{th}$ homology group}, \mdf{$H_k(C)$} is the quotient
\begin{equation}
  H_k(C) \defeq 
  \frac{\ker\bd_k}{\im\bd_{k+1}}.
\end{equation}
We can view $H_k$ as a functor $\Ch\to\VecCat$.
\end{defin}

\subsection{Directed graphs}

\begin{defin}
  \begin{enumerate}[label=(\alph*)]
    \item A \mdf{(simple) digraph} is a tuple $G=(V, E)$ where $V$, the set of \mdf{vertices}, is a finite set and $E$, the set of \mdf{edges}, is a subset of $V\times V \setminus \Delta_V$ where
\begin{equation}
  \Delta_V \defeq \left\{ (i, i) \in V\times V \rmv i \in V \right\}.
\end{equation}
We call elements of $\Delta_V$ \mdf{self-loops}.
  \item A \mdf{directed acyclic graph (DAG)} is a simple digraph $G=(V, E)$ such that there is a linear ordering on the nodes $V=\{ v_1 , \dots , v_n \}$ such that $(v_i, v_j)\in E \implies i < j$. 
  \item An \mdf{oriented graph} is a simple digraph $G=(V, E)$ with no double edges, i.e. $(i, j) \in E \implies (j, i)\not \in E$.
  \item A \mdf{weighted (digraph/DAG/oriented graph)} is a triple $G=(V, E, w)$ such that $(V, E)$ is a (simple digraph/DAG/oriented graph) and $w: E\to \Rpos$ is a positively-valued function on the edges.
  \item An \mdf{undirected graph} is a tuple $G=(V, E)$ where $E$ is a multiset of 2-element subsets of $V$.
  \item Given a digraph $G=(V, E)$, the \mdf{underlying undirected graph} is $\mdf{\underlying{G}} \defeq (V, \underlying{E})$ where 
    \begin{align}
      (i, j)\in E ,(j, i)\not\in E &\implies \{i, j\} \in \underlying{E} \text{ with multiplicity }1,\\
      (i, j), (j, i)\in E &\implies \{i, j\} \in \underlying{E} \text{ with multiplicity }2.
    \end{align}
  \item Given two digraphs $G_1=(V_1, E_1)$, $G_2=(V_2, E_2)$ with $V_1, V_2 \subseteq V$, their union is $\mdf{G_1 \cup G_2} \defeq (V_1 \cup V_2, E_1 \cup E_2)$.
    If $V_1$, $V_2$ are disjoint, then we denote this $\mdf{G_1\sqcup G_2}$.
  \end{enumerate}
\end{defin}

\begin{notation}
Fix a weighted digraph $G=(V, E, w)$.
  \begin{enumerate}[label=(\alph*)]
  \item We denote $\mdf{V(G)}\defeq V$, $\mdf{E(G)}\defeq E$, $\mdf{w(G)}\defeq w$.
  \item For an edge $e=(i, j)\in E$, we write $\mdf{\st(e)}\defeq i$ and $\mdf{\fn(e)}\defeq j$, and say that $i$ and $j$ are \mdf{incident} to $e$.
  \item For an edge $e=(i, j)\in E$, we write $\mdf{w(i, j)}\defeq w(e)$.
  \item We write $\mdf{i\to j}$ to mean there is an edge $(i, j) \in E$.
  \end{enumerate}
\end{notation}

\begin{defin}
Fix a weighted digraph $G=(V, E, w)$.
  \begin{enumerate}[label=(\alph*)]
  \item Given $V'\subseteq V$, the \mdf{induced subgraph} on $V'$ is $(V', E', w')$, where $E' = E \cap (V' \times V')$ and $w'$ is $w$ restricted to $E'$.
  \item Given $E'\subseteq E$, the \mdf{induced subgraph} on $E'$ is $(V', E', w')$, where $V'$ is the set of all vertices incident to some edges in $E'$ and $w'$ is $w$ restricted to $E'$.
  \item For $v\in V$,
    \begin{align}
      \mdf{\nbhd_{in}(v;G)} &\defeq \left\{ a\in V \rmv a \to v \right\},\\
      \mdf{\nbhd_{out}(v;G)} &\defeq \left\{ b\in V \rmv v\to b \right\},\\
      \mdf{\nbhd(v;G)} & \defeq \nbhd_{in}(v) \cup \nbhd_{out}(v)
    \end{align}
    and the \mdf{neighbourhood graph, $\nbhdgraph(v;G)$,} is the induced subgraph on $\nbhd(v)$.
    Where $G$ is clear from context, we omit it from notation.
  \item For $F\subseteq E$, 
    \begin{equation}
\mdf{\nbhd(F; G)}\defeq\left\{\tau \in E \rmv \exists\, e\in F \text{ such that } \tau, e \text{ are incident to a common vertex} \right\}
    \end{equation}
    and the \mdf{neighbourhood graph, $\nbhdgraph(F; G)$,} is the induced subgraph on $\nbhd(F; G)$.
    Where $G$ is clear from context, we omit it from notation.
  \item For two vertices $a,b \in V$, a \mdf{(directed) trail} from $a$ to $b$ is an alternating sequence of vertices, $v_i \in V$, and forward edges, $e_i \in E$,
    \begin{equation}
      p = (v_0 = a , e_1, v_1, e_2, \dots, e_{k}, v_k=b)
    \end{equation}
    such that $e_i = (v_{i-1}, v_i)$.
    We write that \mdf{$p$ is a trail $a\leadsto b$}.
    In a simple digraph, the $e_i$ uniquely determine the $v_i$ (and vice versa) so we occasionally omit one from the notation.
  \item A trail which does not repeat vertices is a \mdf{path}.
  \item An \mdf{undirected circuit} is an alternating sequence of vertices, $v_i \in V$, and edges, $e_i \in E$,
    \begin{equation}
      p = (v_0 = a , e_1, v_1, e_2, \dots, e_{k}, v_k=a)
    \end{equation}
    such that $\{\st(e_i), \fn(e_i)\} = \{v_{i-1}, v_i\}$ and $v_0 = v_k$.
  \item Given an undirected circuit $p$, as above, if $v_0, \dots, v_{k-1}$ are all distinct then we say $p$ is \mdf{simple}.
  \end{enumerate}
\end{defin}

\begin{notation}
Fix a weighted digraph $G=(V, E, w)$.
  \begin{enumerate}[label=(\alph*)]
    \item Given a trail $p$, the \mdf{length of $p$} is defined as
      \begin{equation}
        \pathlen(p)\defeq\sum_{i=1}^{k}w(e_i). 
      \end{equation}
    \item We denote the set of all paths $i\leadsto j$ by $\mdf{\mathcal{P}(i, j)}$.
  \end{enumerate}
\end{notation}

\begin{defin}
For a weighted digraph $G$, the \mdf{shortest-path quasimetric} $d:V(G)\times V(G) \to \R\sqcup\{\infty\}$ is defined by 
\begin{equation}
  d(i, j) \defeq 
  \begin{cases}
    \min_{p \in \mathcal{P}(i, j)} \pathlen(p) & \text{if }\mathcal{P}(i, j)\neq\emptyset, \\
    \infty & \text{otherwise}.
  \end{cases}
\end{equation}
\end{defin}

There is a notion of morphisms between (weighted) digraphs.

\begin{defin}
  \begin{enumerate}[label=(\alph*)]
    \item Given two simple digraphs $G, H$, a \mdf{digraph map} (or simply \mdf{map}), $f:G \to H$, is a map on vertices,
      $f:V(G) \to V(H)$,
    such that
    \begin{equation}
      i\to j \implies f(i) \to f(j)\text{ or }f(i) = f(j).\label{eq:dig_map_cond}
    \end{equation}
    Given a vertex map $f: V(G) \to V(H)$ satisfying condition~(\ref{eq:dig_map_cond}), we say $f$ \mdf{induces} a digraph map $G \to H$.
  \item A digraph map $f:G \to H$ is called an \mdf{inclusion} if $V(G) \subseteq V(H)$ and $f$ is induced by the inclusion vertex map.
  \item Given two weighted digraphs $G, H$, a digraph map $f: G \to H$ is called a \mdf{contraction} if for all nodes $i, j \in V(G)$, we have
  \begin{equation}
 d_H(f(i), f(j)) \leq d_G(i, j)
  \end{equation}
  where $d_G$ and $d_H$ are the shortest-path quasimetrics on $G$ and $H$ respectively.
  \end{enumerate}
\end{defin}

\begin{notation}
\begin{enumerate}[label=(\alph*)]
  \item For $e=(i, j)\in E(G)$, we denote $\mdf{f(e)}\defeq (f(i), f(j))$.
    Note that $f(e) \in E(H)\sqcup \Delta_{V(H)}$.
\item Given a path $p=(v_0, e_1, \dots, e_k, v_k)$ and a digraph map $f:G \to H$, the \mdf{image of $p$} is the path, $\mdf{f(p)}$, obtained from
  \begin{equation}
     (f(v_0), f(e_1), f(v_1), \dots , f(e_k), f(v_k))
  \end{equation}
  by removing $f(e_i)$ and $f(v_i)$ from the sequence if $f(e_i)$ is a self-loop.
  \end{enumerate}
\end{notation}

\begin{rem}
  Suppose $f:G \to H$ is a weighted digraph such that $w(H)(f(e)) \leq w(G)(e)$ for every edge $e\in E(G)$ with $f(e) \not\in\Delta_{V(H)}$.
  Then for any path $p:i \leadsto j$ in $G$, $f(p)$ is a path $f(i)\leadsto f(j)$ in $H$ and $\pathlen(f(p))\leq\pathlen(p)$.
  Hence, $f$ is certainly a contraction.
  However, this is not a necessary contraction since the shortest path joining $f(i)\leadsto f(j)$ need not be the image of the shortest path joining $i \leadsto j$.
\end{rem}

Given these ways of mapping between (weighted) digraphs, a number of categories naturally arise.

\begin{defin}\label{defin:dgr_cats}
\begin{enumerate}[label=(\alph*)]
  \item We denote the category of simple digraphs, directed acyclic graphs and oriented graphs, where the morphisms are all digraph maps, by $\mdf{\Dgr}$, $\mdf{\Dag}$ and $\mdf{\Dor}$ respectively.
  \item We use the prefix $\mdf{\bm{\mathrm{W}}}$ to denote the corresponding categories of \emph{weighted} digraphs where a morphism is any digraph map of the underlying, unweighted digraphs.
    For example, $\WDgr$ is a category of weighted simple digraphs.
  \item For a category of weighted digraphs, we use the prefix $\mdf{\Cont}$ to denote the subcategory, containing all objects, with the additional restriction that morphisms must be contractions.
  \item For a category of weighted or unweighted digraphs, we use the prefix $\mdf{\Incl}$ to denote the wide subcategory, containing all objects, with the additional restriction that morphisms must be inclusions.
\end{enumerate}
\end{defin}

\subsection{Path homology}
Path homology is a homology theory for directed graph, which was first introduced by \citeauthor{Grigoryan2012}~\cite{Grigoryan2012}.
Subsequent papers prove K\"unneth theorems for Cartesian products and joins~\cite{Grigoryan2020}, and invariance under an appropriate notation of digraph homotopy~\cite{Grigoryan2014}
A directed network gives rise to a natural filtration of digraphs which leads to a stable theory of persistent path homology~\cite{ChowdhurySIAM}.
(Persistent) path homology has also been extended to vertex-weighted digraphs~\cite{Lin2019}.
Path homology can be defined for an arbitrary \mdf{path complex}; here we present the definition for a digraph.

Fix a ring $\Ring$ and a simple directed graph $G=(V, E)$.
\begin{defin}
  The following definitions classify sequences of vertices in $V$:
  \begin{enumerate}[label=(\alph*)]
    \item An \mdf{elementary $p$-path} is any sequence $v_0 \dots v_p$ of $(p+1)$ vertices, $v_i \in V$.
    \item An elementary $p$-path, $v_0 \dots v_p$, is \mdf{regular} if $v_i \neq v_{i+1}$ for every $i$.
      Otherwise, we say it is \mdf{non-regular}.
    \item An elementary $p$-path, $v_0 \dots v_p$, is \mdf{allowed} if $(v_i, v_{i+1})\in E$ for every $i$.
  \end{enumerate}
\end{defin}

\begin{defin}
  We freely generate $\Ring$-modules from these sequences of vertices, for each $p\geq 0$.
  \begin{align}
      \Lambda_p\defeq\Lambda_p(G; \Ring) 
      &\defeq
      \Ring\big\langle\left\{ 
      v_0 \dots v_p 
      \text{ elementary }p\text{-path on }V
      \right\}\big\rangle \\
      \mathcal{R}_p\defeq\mathcal{R}_p(G; \Ring) 
      &\defeq
      \Ring\big\langle\left\{ 
      v_0 \dots v_p 
      \text{ regular }p\text{-path on }V
      \right\}\big\rangle \\
      \mathcal{A}_p\defeq\mathcal{A}_p(G; \Ring) 
      &\defeq
      \Ring\big\langle\left\{ 
      v_0 \dots v_p 
      \text{ allowed }p\text{-path in }G
      \right\}\big\rangle 
  \end{align}
  For $p=-1$, we let $\Lambda_{-1}\defeq\mathcal{R}_{-1}\defeq\mathcal{A}_{-1}\defeq\Ring$.
\end{defin}

\begin{defin}
  Given $p\geq 0$, the \mdf{non-regular boundary map} $\bd[nreg]_p:\Lambda_p\to\Lambda_{p-1}$ is given on the standard basis by
  \begin{equation}\label{eq:bd_formula}
    \bd[nreg]_p(v_0 \dots v_p) \defeq \sum_{i=0}^p (-1)^i v_0 \dots \hat{v_i}\dots v_p
  \end{equation}
  where $v_0 \dots \hat{v_i} \dots v_p$ is the $(p-1)$-path obtained by removing $v_i$ from $v_0 \dots v_p$.
\end{defin}

\begin{defin}
  Since $\mathcal{R}_p \subseteq \Lambda_p$, let $\pi:\Lambda_p \to \mathcal{R}_p$ denote the projection map onto $\mathcal{R}_p$.
  The \mdf{regular boundary map} $\bd[reg]_p : \mathcal{R}_p \to \mathcal{R}_{p-1}$ is given by
  \begin{equation}
    \bd[reg]_p \defeq \pi \circ \bd[nreg]_p.
  \end{equation}
\end{defin}

This boundary operator does not pass down to a boundary operator between the $\left\{ \mathcal{A}_p\right\}$ so we must define the following sub-modules.

\begin{defin}
  The \mdf{space of $\bd[reg]$-invariant $p$-paths} is
  \begin{equation}
    \biv[reg]_p\defeq\biv[reg]_p(G; \Ring)\defeq
    \left\{ v \in \mathcal{A}_p \rmv \bd[reg]_pv \in \mathcal{A}_{p-1} \right\}.
  \end{equation}
\end{defin}

\begin{rem}
  Note that $\bd[reg]_p$ restricts to a homomorphism $\biv[reg]_p \to \biv[reg]_{p-1}$ and
  a standard check confirms that $\bd[reg]_{p-1} \circ \bd[reg]_p = 0$ \cite[Lemma~2.9]{Grigoryan2012} 
\end{rem}

\begin{defin}
  The \mdf{regular path (chain) complex} is
  \begin{equation}
    \begin{tikzcd}
      \dots \arrow[r, "\partial_3"] &
      \biv[reg]_2 \arrow[r, "\partial_2"] &
      \biv[reg]_1 \arrow[r, "\partial_1"] &
      \biv[reg]_0 \arrow[r, "\partial_0"] &
      \Ring \arrow[r, "\partial_{-1}"] &
      0
    \end{tikzcd}
  \end{equation}
  The homology of the regular path complex is the \mdf{regular path homology} of $G$, the $k^{th}$ homology group is
  \begin{equation}
    \homgr[reg]_k \defeq \homgr[reg]_k(G; \Ring) \defeq \frac{\ker \bd[reg]_k}{\im \bd[reg]_{k+1}}.
  \end{equation}
  The \mdf{$k^{th}$ Betti number} is $\betti_k \defeq \rank \homgr[reg]_k$.
\end{defin}

\begin{defin}
  Given a digraph map $f: G \to H$, the \mdf{induced map} $\inducedch{f}:\mathcal{R}_p(G) \to \mathcal{R}_p(H)$ is given on the standard basis by
  \begin{equation}
    \inducedch{f}(v_0 \dots v_p) \defeq
    \begin{cases}\label{eq:induced_ch}
      f(v_0) \dots f(v_p) & \text{if }f(v_0)\dots f(v_p)\text{ is regular} \\
      0 & \text{otherwise}
    \end{cases}
  \end{equation}
\end{defin}

\begin{lemma}[{\cite[Theorem~2.10]{Grigoryan2014},~\cite[Proposition~A.2]{ChowdhurySIAM}}]
The induced maps restrict to maps $\inducedch{f}:\biv[reg]_p(G) \to \biv[reg]_p(H)$ which commute with $\partial_p$ and hence form chain maps between the regular path complexes.
Moreover these chain maps are functorial, i.e. $\inducedch{(f\circ g)} = \inducedch{f} \circ \inducedch{g}$ and $\inducedch{(\id)}=\id$.
Hence $\biv[reg]$ is a functor $\Dgr\to\Ch$.
\end{lemma}

We will primarily be interested in the $1^{st}$ homology group $\homgr[reg]_1$.
Hence the following characterisation of the low-dimensional chain groups will be of use.

\begin{prop}[{\cite[\S~3.3]{Grigoryan2012}}]
For any simple digraph $G=(V, E)$, 
$\biv[reg]_0(G)$ is isomorphic to the $\Ring$-module freely generated by the vertices and
$\biv[reg]_1(G)$ is isomorphic to the $\Ring$-module freely generated by the edges, i.e.
\begin{equation}
  \biv[reg]_0(G) \cong \Rspan{V}
  \quad\text{and}\quad
  \biv[reg]_1(G) \cong \Rspan{E}.
\end{equation}
\end{prop}

\begin{notation}
Note that an edge $e=(a, b)\in E$ gives rise to an allowed 1-path $ab\in\mathcal{A}_1(G)$.
Moreover, $ab\in\biv[reg]_1(G)$.
For ease of notation, given an edge $e\in E$, we will also use $e=ab\in\biv[reg]_1(G)$ to refer to the generator in $\biv[reg]_1(G)$.
\end{notation}

Since $\biv_1$ is generated by edges in $G$, any trail $p$ has a representative.

\begin{notation}
  Given a directed trail $p=(e_1, \dots, e_k)$ in a digraph $G$, the \mdf{representative} of $p$ is 
  \begin{equation}
    \repres{p}\defeq \sum_{i=1}^k e_i \in \biv_1.
  \end{equation}
\end{notation}

Likewise, there is a representative for any undirected circuit.

\begin{notation}
  Given an undirected circuit $p=(v_0=a, e_1, v_1, \dots, e_k, v_k=a)$ in a digraph $G$, the \mdf{representative} of $p$ is 
  \begin{equation}
    \repres{p}\defeq \sum_{i=1}^k \alpha_i e_i \in \biv_1.
  \end{equation}
  where $\alpha_i = 1$ if $e_i = (v_{i-1}, v_i)$, else $\alpha_i = -1$.
\end{notation}

\begin{rem}
The representative of a circuit $p$, does not depend on the starting point, but if $p'$ traverses the circuit in the opposite direction then $\repres{p'} = - \repres{p}$.
\end{rem}

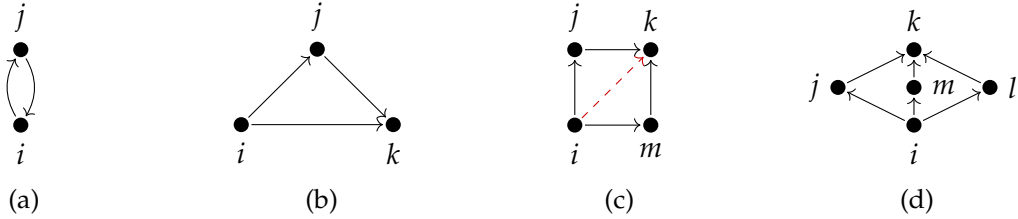
\begin{figure}[hbt]
	\centering
  \begin{subfigure}[t]{0.24\textwidth}
   \centering
   \begin{tikzpicture}[
	roundnode/.style={circle, fill=black, minimum size=4pt},
	inner sep=2pt,
	outer sep=1pt
	]
	\node (i) at (0, 0) [roundnode, label=below:$i$] {};
	\node (j) at (0, 1) [roundnode, label=above:$j$] {};

  \draw [->, bend left] (i) to (j);
  \draw [->, bend left] (j) to (i);
\end{tikzpicture}
   \caption{}
  \end{subfigure}
  \begin{subfigure}[t]{0.24\textwidth}
   \centering
   \begin{tikzpicture}[
	roundnode/.style={circle, fill=black, minimum size=4pt},
	inner sep=2pt,
	outer sep=1pt
	]
	\node (i) at (0, 0) [roundnode, label=below:$i$] {};
	\node (j) at (1, 1) [roundnode, label=above:$j$] {};
	\node (k) at (2, 0) [roundnode, label=below:$k$] {};

  \draw [->] (i) to (j);
  \draw [->] (j) to (k);
  \draw [->] (i) to (k);
\end{tikzpicture}
   \caption{}
  \end{subfigure}
  \begin{subfigure}[t]{0.24\textwidth}
   \centering
   \begin{tikzpicture}[
	roundnode/.style={circle, fill=black, minimum size=4pt},
	inner sep=2pt,
	outer sep=1pt
	]
	\node (i) at (0, 0) [roundnode, label=below:$i$] {};
	\node (j) at (0, 1) [roundnode, label=above:$j$] {};
	\node (k) at (1, 1) [roundnode, label=above:$k$] {};
	\node (m) at (1, 0) [roundnode, label=below:$m$] {};

  \draw [->] (i) to (j);
  \draw [->] (j) to (k);
  \draw [->] (i) to (m);
  \draw [->] (m) to (k);
  \draw [->, dashed, red!80!black] (i) to (k);

\end{tikzpicture}
   \caption{}
  \end{subfigure}
  \begin{subfigure}[t]{0.24\textwidth}
   \centering
   \begin{tikzpicture}[
	roundnode/.style={circle, fill=black, minimum size=4pt},
	inner sep=2pt,
	outer sep=1pt
	]
	\node (i) at (0, 0) [roundnode, label=below:$i$] {};
	\node (j) at (-1, 0.5) [roundnode, label=left:$j$] {};
	\node (m) at (0, 0.5) [roundnode, label=right:$m$] {};
	\node (l) at (1, 0.5) [roundnode, label=right:$l$] {};
	\node (k) at (0, 1) [roundnode, label=above:$k$] {};

	\draw [->] (i)--(j);
	\draw [->] (i)--(m);
	\draw [->] (i)--(l);
	\draw [->] (j)--(k);
	\draw [->] (m)--(k);
	\draw [->] (l)--(k);
\end{tikzpicture}
   \caption{}\label{fig:ld_squares}
  \end{subfigure}
	\caption{The three types of generators for $\biv[reg]_2(G; \Z)$ and $\biv[reg]_2(G; \R)$:
    (a) A double edge.
    (b) A directed triangle.
    (c) A long square (the dashed red edge must not be present).
    Finally (d) shows a linear dependency between long squares.
  }\label{fig:omega2_structure}
\end{figure}

\begin{prop}[{\cite[Proposition~2.9]{Grigoryan2014}}, {\cite[Theorem~3]{Dey2020a}}]\label{prop:biv2_gens}
Let G be a finite, simple digraph and $\Ring=\R\text{ or }\Z$.
Any $\omega\in\biv[reg]_2(G; \Ring)$ can be written as a linear combination of $\bd$-invariant $2$-paths of the following three types:
\begin{enumerate}[label=(\alph*)]
  \item $(iji)$ where $i\to j \to i$ (double edge);
  \item $(ijk)$ where $i \to j \to k$, $i\to k$  and $i\neq k$ (directed triangle); and
  \item $(ijk - imk)$ where $i\to j \to k$, $i \to m \to k$, $i\not\to k$ and $i\neq k$ (long square).
\end{enumerate}
\end{prop}

First note that all of the elements identified in Proposition~\ref{prop:biv2_gens} are elements of $\biv[reg]_2(G; \Ring)$ and hence they form a generating set.
However, the generators corresponding to long squares are not necessarily linearly independent.
For example, in Figure~\ref{fig:ld_squares}, we see
\begin{equation}
  (ijk - ilk) = (ijk - imk) + (imk -ilk).
\end{equation}
Removing some long squares to account for these linear relations, we can obtain a basis of $\biv_2(G; \R)$.

\subsection{Persistent homology}

Topological data analysis (TDA) is a field of applied mathematics which employs the powerful, discriminative tools of algebraic topology to study complex datasets .
The cornerstone of the field is persistent homology (PH) which yields a stable, discrete, topological invariant, called a barcode (see~\cite{Bubenik2012, Chazal2021, Otter2015} for an overview).
The barcode summarises topological features in the data (e.g.\ connected components and loops) and measures the range of scales across which they persist.

\begin{defin}\label{defin:persvec}
  \begin{enumerate}[label=(\alph*)]
    \item A \mdf{persistent chain complex} is a functor $\Rposet\to\Ch$.
    \item A \mdf{persistent vector space} is a functor $\Rposet\to\VecCat$.
      We denote the category of such functors $\mdf{\PersVec}\defeq\Funct{\Rposet}{\VecCat}$.
    \item A persistent vector space $M$ is \mdf{pointwise finite-dimensional (p.f.d)} if $M(t)$ is finite dimensional for all $t\in\R$, that is $M\in\Funct{\Rposet}{\vecCat}$.
      We denote the category of p.f.d persistent vector spaces $\mdf{\Persvec}\defeq\Funct{\Rposet}{\vecCat}$.
    \item A persistent vector space $M:\Rposet\to\VecCat$ is \mdf{tame}~\cite{nanda2021computational} if 
      \begin{enumerate}[label=(\roman*)]
        \item $M(t)$ is finite dimensional for all $t \in \R$, and
        \item there are finitely many $t\in \R$ such that there is no $\epsilon > 0$ such that $M(t- \epsilon \leq t+ \epsilon)$ is an isomorphism.
      \end{enumerate}
    \item For an interval $I\subseteq R$, we define the \mdf{corresponding interval}, $\mdf{P(I)}\in\PersVec$, in which the vector spaces are
      \begin{equation}
        P(I)(t) \defeq
        \begin{cases}
          \R &\text{if } t \in I,\\
          0 &\text{otherwise},
        \end{cases}
      \end{equation}
      and $P(I)(s\leq t)$ is the identity if $s, t\in I$ and the trivial map otherwise.
    \item Given $M, N\in\PersVec$, their direct sum $\mdf{M\oplus N}\in\PersVec$ is given pointwise by
      \begin{align}
        (M\oplus N)(t) &\defeq M(t) \oplus N(t),  \\
        (M\oplus N)(s \leq t) &\defeq M(s\leq t) \oplus N(s\leq t).
      \end{align}
    \item A \mdf{morphism of persistent vector spaces} is a morphism in the category $\Funct{\Rposet}{\VecCat}$. That is, for $M, N \in \Funct{\Rposet}{\VecCat}$ a morphism $\phi: M \to N$ is a family of linear maps $\left\{ \phi_t : M(t) \to N(t) \right\}$ such that
      \begin{equation}
        N(s\leq t) \circ \phi_s = \phi_t \circ M(s \leq t)
      \end{equation}
      whenever $s\leq t$.
      We say $\phi$ is an \mdf{isomorphism} if each $\phi_t$ is an isomorphism of vector spaces.
      If an isomorphism $M \to N$ exists, we write $M\cong N$.
  \end{enumerate}
\end{defin}

A (p.f.d) persistent vector space can be decomposed as a direct sum of interval modules, the \emph{indecomposable} persistent vector spaces. 
Moreover, this decomposition is unique, discrete and finite in all practical applications.

\begin{theorem}[Structure Theorem for p.f.d persistent vector spaces,%
                  {\cite[Theorem~1.1]{CrawleyBoevey2012},%
                  \cite[Theorem~2.8]{chazal2016structure}}]%
                  \label{thm:pvs_decompose}
Given $M\in\Persvec$, there is a multiset $\Barcode{M}$ of intervals of $\R$ such that
\begin{equation}
  M \cong \bigoplus_{I\in\Barcode{M}} P(I)
\end{equation}
and any such decomposition is unique, up to reordering.
We call $\Barcode{M}$ the \mdf{barcode} of $M$.
\end{theorem}

\begin{defin}
  \begin{enumerate}[label=(\alph*)]
    \item A multiset of intervals of $\R$ is called a \mdf{barcode}.
    \item We call an interval $I$ in a barcode a \mdf{feature}.
      If $I$ starts at $a$ and ends at $b$, we say the feature is \mdf{born} at time $a$ and \mdf{dies} at time $b$.
    \item Given a barcode $\mathcal{B}$, the \mdf{diagram} of $\mathcal{B}$ is the multiset of endpoints
  \begin{equation}
    \mdf{\Dgm(B)} \defeq
    \ms{
      (a_k, b_k) 
      \rmv
      I_k \in \mathcal{B}\text{ has endpoints }a_k \leq b_k
    }.
  \end{equation}
  \item Given $M\in\Persvec$, the \mdf{persistence diagram} of $M$ is $\mdf{\Dgm(M)}\defeq\Dgm(\Barcode{M})$.
  \end{enumerate}
\end{defin}

The barcode can be used as a summary of the persistent vector space.
When arising as the homology of a filtration of topological spaces, this summary captures how topological features are born and killed throughout the filtration.
In order to use this summary for further statistics, it is desirable that this summary is \emph{stable} to noise and perturbations in the input data.
To quantify this stability, we require metrics on persistent vector spaces and the resulting barcodes.

\begin{defin}[{\cite{edelsbrunner2008persistent}}]
  Given two multisets $D_1, D_2 \subseteq \R^2$, the \mdf{bottleneck distance} is
  \begin{equation}
    \mdf{d_B(D_1, D_2)} \defeq \inf_\gamma \sup_{x \in D_1 \cup \Delta_\R} \norm{x-\gamma(x)}_\infty
  \end{equation}
  where $\gamma$ is over all multi-bijections $D_1 \cup \Delta_\R \to D_2 \cup \Delta_\R$ and $\Delta_\R=\left\{(x, x) \rmv x \in \R\right\}$ is the diagonal with multiplicity $1$.
  Given barcodes $\mathcal{B}_1$, $\mathcal{B}_2$, we define the bottleneck distance between them to be the bottleneck distance between their diagrams
  \begin{equation}
    \mdf{d_B(\mathcal{B}_1, \mathcal{B}_2)} \defeq d_B(\Dgm(\mathcal{B}_1), \Dgm(\mathcal{B}_2)).
  \end{equation}
\end{defin}

\begin{defin}[{\cite{cohen2010lipschitz}}]
  Given $p\geq 1$ and two multisets $D_1, D_2 \subseteq \R^2$, the \mdf{$p$-Wasserstein distance}~\cite{cohen2010lipschitz} is
  \begin{equation}
    \mdf{d_{W_p}(D_1, D_2)} \defeq \inf_\gamma \left(\sum_{x \in D_1 \cup \Delta_\R} \norm{x-\gamma(x)}_\infty^p\right)^{1/p}
  \end{equation}
  where $\gamma$ is over all multi-bijections $D_1 \cup \Delta_\R \to D_2 \cup \Delta_\R$ and $\Delta_\R=\left\{(x, x) \rmv x \in \R\right\}$ is the diagonal with multiplicity $1$.
  The $p$-Wasserstein between two barcodes $\mathcal{B}_1$, $\mathcal{B}_2$, is
  \begin{equation}
   \mdf{d_{W_p}(\mathcal{B}_1, \mathcal{B}_2)} \defeq d_B(\Dgm(\mathcal{B}_1), \Dgm(\mathcal{B}_2)).
  \end{equation}
\end{defin}

\begin{defin}[{\cite{Bauer2013a}}]
    Given a category $\C$, fix $M, N\in\Funct{\Rposet}{\C}$ and $\delta \geq 0$.
  \begin{enumerate}[label=(\alph*)]
    \item The \mdf{$\delta$-shift} of $M$ is $\shifted{M}{\delta}\in\Funct{\Rposet}{\C}$ where
      \begin{equation}
        \shifted{M}{\delta}(t) \defeq M(t + \delta)
        \quad\text{ and }\quad
        \shifted{M}{\delta}(s \leq t) \defeq M(s + \delta \leq t + \delta).
      \end{equation}
    \item Given a morphism $f\in\MorXY{\Funct{\Rposet}{\C}}{M}{N}$ the \mdf{$\delta$-shift} of $f$ is $\shifted{f}{\delta}:\shifted{M}{\delta}\to\shifted{N}{\delta}$ in which $f_t \defeq f_{t+\delta}$.
      When clear from context we often denote $\shifted{f}{\delta}=f$.
    \item The \mdf{$\delta$-transition morphism} is a morphism $\mdf{\transmorph{M}{\delta}}:M\to\shifted{M}{\delta}$ which at $t\geq 0$ is given by $M(t \leq t+\delta)$.
    \item A \mdf{$\delta$-interleaving} is a pair of morphisms $\phi: M \to \shifted{N}{\delta}$ and $\psi: N \to \shifted{M}{\delta}$ such that
      \begin{equation}
        \shifted{\psi}{\delta} \circ \phi = \transmorph{M}{2\delta}
        \quad\text{ and }\quad
        \shifted{\phi}{\delta} \circ \psi = \transmorph{N}{2\delta}.
      \end{equation}
      \item The \mdf{interleaving distance} of $M$ and $N$ is
    \begin{equation}
    \mdf{d_I(M, N)} \defeq \inf \left\{ \delta \geq 0 \rmv \exists\;\delta\text{-interleaving} \right\}.
    \end{equation}
  \end{enumerate}
\end{defin}

\begin{rem}
Recall that, given a $\delta$-interleaving $\phi$ and $\psi$, in order to constitute morphisms $M\to \shifted{N}{\delta}$ and $N\to\shifted{M}{\delta}$, they must satisfy relations
\begin{equation}
 \phi_t \circ M(s \leq t) = N(s+\delta \leq t+ \delta) \circ \psi_s
 \quad\text{ and }\quad
 \psi_t \circ N(s \leq t) = M(s+\delta \leq t+ \delta) \circ \phi_s
\end{equation}
for each $s\leq t$.
\end{rem}

Now that we have metrics on persistent vector spaces and their barcode summaries, we can state the isometry theorem.
This guarantees that the barcode is a stable summary of the input persistent vector space.

\begin{theorem}[{Isometry~Theorem~\cite[Theorem~3.5]{Bauer2013a}}]\label{thm:isometry}
  Given p.f.d persistent vector spaces $M,N\in\Persvec$, 
  \begin{equation}
    d_B(\Barcode{M}, \Barcode{N}) = d_I(M, N).
  \end{equation}
\end{theorem}

Finally, when a persistent vector space is tame, there are finitely many critical values $t_0 < \dots < t_k$ such that if $t_{i-1} < s \leq t < t_i$ then $M(s \leq t)$ is an isomorphism~\cite{chazal2016structure}.
Hence, all information of the persistent vector space is contained within the maps $M(t_{i-1} \leq t_i)$ for $i=1, \dots, k$.
In particular, any interval in the barcode must have its endpoints at one of the critical values (or $\pm\infty$).
In these scenarios, it suffices to consider $M$ as a functor $\textbf{[}\bm{k}\textbf{]} \to \VecCat$, where $\textbf{[}\bm{k}\textbf{]}$ is the sub-poset of $\Rposet$ consisting of the integers $0, \dots, k$~\cite{nanda2021computational}.

\section{Motivation and definition of GrPPH}\label{sec:new_descriptor}
Firstly, in Section~\ref{sec:intro_subdivision}, we describe a standard pipeline for extracting a topological summary from a weighted digraph, and illustrate a number of issues that naturally arise.
Motivated by this in Section~\ref{sec:descriptor_defin}, we alter the standard pipeline in order to define a `grounded pipeline'.
We prove that this new pipeline is functorial in an appropriate sense, which we will later exploit for stability results.
The pipeline is parameterised by two choices; in Section~\ref{sec:grpph_defin} we fix these choices in order to define our proposed descriptor.

\subsection{Standard pipeline}\label{sec:intro_subdivision}
A typical TDA pipeline for weighted digraphs consists of three ingredients:
\begin{enumerate}
  \item a map $F:\Obj(\WDgr)\to\Obj(\Funct{\Rposet}{\Dgr})$, which assigns a filtration of digraphs to every weighted digraph;
  \item a chain complex functor $C:\Dgr \to \Ch$ which maps each digraph to a chain complex and induces chain map for every digraph map; and finally
  \item a choice of homology functor $H_k: \Ch \to \VecCat$ in some degree $k$.
\end{enumerate}
These components can then be combined into the following pipeline.
\begin{figure}[H]
  \centering
  \begin{tikzcd}[column sep=large]
\mathcal{H}_k : \WDgr \arrow[r, "F"]
& \Funct{\Rposet}{\Dgr} \arrow[r, "\Funct{\Rposet}{C}"]
& \Funct{\Rposet}{\Ch} \arrow[r, "\Funct{\Rposet}{H_k}"]
& \Funct{\Rposet}{\VecCat}
  \end{tikzcd}
\end{figure}

We obtain a map $\mathcal{H}_k : \Obj(\WDgr) \to \Obj(\Funct{\Rposet}{\VecCat})$ given by
$\mathcal{H}_k \defeq \Funct{\Rposet}{H_k} \circ \Funct{\Rposet}{C} \circ F$.
Since $F$ is not a priori functorial, neither is $\mathcal{H}_k$.
Under mild assumptions on $F$ and $C$, it is possible to define a subcategory of $\WDgr$ which makes this pipeline functorial.

\begin{notation}
Given $F:\Obj(\WDgr)\to\Obj(\Funct{\Rposet}{\Dgr})$, $G \in \WDgr$ and $s\leq t$ we write
\begin{enumerate}[label=(\alph*)]
  \item $\mdf{F^t G}\defeq F(G)(t)$, the image of $t$ under the functor $F(G)$; and
  \item $\mdf{\filtincl{s}{t}}\defeq F(G)(s\leq t)$, the image of $s\leq t$ under the functor $F(G)$.
\end{enumerate}
\end{notation}

\begin{defin}
  \begin{enumerate}[label=(\alph*)]
    \item A \mdf{filtration map} is any map $F:\Obj(\WDgr) \to \Obj(\Funct{\Rposet}{\Incl\Dgr})$
    such that $V(F^t G) \subseteq V(G)$ for all $t\in \R$.
    In particular, $\filtincl{s}{t}$ must always be an inclusion.

    \item Given a filtration map $F$, 
    a morphism of weighted digraphs $f\in\MorXY{\WDgr}{G}{H}$ is called \mdf{$F$-compatible} if for every $t\in \R$
    the underlying vertex map $f: V(G) \to V(H)$ restricts to a vertex map $V(F^t G) \to V(F^t H)$ which in turn yields a digraph map $F^t G \to F^t H$.

    \item Given a filtration map $F$, the \mdf{$F$-compatible category of weighted digraphs}, \mdf{$\WDgrF$}, is the subcategory of $\WDgr$ such that $\Obj(\WDgrF)=\Obj(\WDgr)$ and
    \begin{equation}
      \Mor(\WDgrF) = \left\{ f \in \Mor(\WDgr) \rmv f\text{ is }F\text{-compatible} \right\}. 
    \end{equation}
  \end{enumerate}
\end{defin}

\begin{lemma}\label{lem:filtration_functor}
Any filtration map $F: \Obj(\WDgr)\to\Obj(\Funct{\Rposet}{\Incl\Dgr})$ induces a functor $F: \WDgrF \to \Funct{\Rposet}{\Dgr}$, which we call a \mdf{filtration functor}.
\end{lemma}
\begin{proof}
  Given $f\in\MorXY{\WDgrF}{G}{H}$, since $f$ is $F$-compatible, the underlying vertex map induces digraph maps $f:F^t G \to F^t H$ for every $t \in \R$.
  Given $s\leq t$, both $F(G)(s\leq t)$ and $F(H)(s\leq t)$ are digraph maps induced by the inclusion vertex map.
  Hence the following square  of morphisms in $\Dgr$ commutes.
  \begin{figure}[H]
    \centering
    \begin{tikzcd}
      F^s G \arrow[d, "F(G)(s\leq t)"'] \arrow[r, "f"] & F^s H \arrow[d, "F(H)(s\leq t)"] \\
      F^t G \arrow[r, "f"'] & F^t H
    \end{tikzcd}
  \end{figure}\noindent
  Hence $f$ induces a natural transformation between $F(G)$ and $F(H)$.
  Moreover, since each $f:F^t G \to F^t H$ is fully determined by the underlying vertex map $V(G) \to V(H)$, this construction is certainly functorial.
\end{proof}

\begin{rem}
Since any filtration map induces a filtration functor, it suffices to define a filtration functor only as a map on objects. 
\end{rem}

When $F$ is a filtration map, it induces a functor $\WDgrF\to\Funct{\Rposet}{\Dgr}$ and hence $\mathcal{H}_k$ is a functor $\WDgrF \to \Funct{\Rposet}{\VecCat}$, as desired.
We will now consider an illustrative example of this pipeline.
Assuming the weight of an edge corresponds to a distance between its endpoints (e.g. the time it takes for a particle to flow down the edge), a natural choice of filtration functor is the following.

\begin{defin}\label{defin:shortest_path}
The \mdf{shortest-path filtration} is a map
$F_d:\WDgr\to\Funct{\Rposet}{\Dgr}$.
For $G=(V, E, w)\in\WDgr$ and $t\in \R$, we define
\begin{equation}
F_d(G)(t) \defeq G^t \defeq (V, E^t)
\quad\text{ where }\quad
E^t \defeq \left\{(i, j)  \rmv d(i, j) \leq t \right\}
\end{equation}
and $d$ is the shortest-path quasimetric on $G$.
For $s\leq t$, the digraph map $G^s \to G^t$ is induced by the identity vertex map $\id_V$.
\end{defin}

\begin{figure}[htbp]
  \centering
  \resizebox{0.9\textwidth}{!}{%
    \begin{tikzpicture}[
  roundnode/.style={circle, fill=black, minimum size=4pt},
  bluenode/.style={circle, fill=blue, minimum size=1pt},
	squarenode/.style={fill=black, minimum size=4pt},
	inner sep=2pt,
	outer sep=1pt
  ]
  \node (a) at (0, 0) [roundnode] {};
  \node (b) at (1, 1) [roundnode] {};
  \node (c) at (1, -1) [roundnode] {};
  \node (d) at (2, 1.5) [roundnode] {};
  \node (e) at (2, 0.5) [roundnode] {};
  \node (f) at (2, -0.5) [roundnode] {};
  \node (g) at (2, -1.5) [roundnode] {};
  \node (h) at (3, 1) [roundnode] {};
  \node (i) at (3, -1) [roundnode] {};
  \node (j) at (4, 0) [roundnode] {};
  \draw[->] (a)--(b);
  \draw[->] (a)--(c);
  \draw[->] (b)--(d);
  \draw[->] (b)--(e);
  \draw[->] (c)--(f);
  \draw[->] (c)--(g);
  \draw[->] (d)--(h);
  \draw[->] (e)--(h);
  \draw[->] (f)--(i);
  \draw[->] (g)--(i);
  \draw[->] (h)--(j);
  \draw[->] (i)--(j);

  \node[] at (2, -2) {$G_1$};

  \draw[->] (5.5, -1.5) -- (5.5, 1.5); 
  \draw[->] (5.5, -1.5) -- (8.5, -1.5);
  \node[rotate=90] at (5, 0) {\scriptsize Death};
  \node at (7, -2) {\scriptsize Birth};
  \draw[dashed, gray] (5.5, -1.5) -- (8.5, 1.5);
  \draw[densely dotted, blue] (6.7, -1.5) -- (6.7, 0.9) -- (5.5, 0.9);
  \node[] at (6.7, -1.65) {\scriptsize $1$};
  \node[] at (5.4, 0.9) {\scriptsize $2$};
  \node at (6.7, 0.9) [bluenode, label={[right]{\scriptsize$\times 1$}}] {};

  \node (a2) at (9.5, 0) [roundnode] {};
  \node (abmid2) at (10, 0.5) [roundnode] {};
  \node (acmid2) at (10, -0.5) [roundnode] {};
  \node (b2) at (10.5, 1) [roundnode] {};
  \node (bdmid2) at (11, 1.25) [roundnode] {};
  \node (bemid2) at (11, 0.75) [roundnode] {};
  \node (c2) at (10.5, -1) [roundnode] {};
  \node (cfmid2) at (11, -0.75) [roundnode] {};
  \node (cgmid2) at (11, -1.25) [roundnode] {};
  \node (d2) at (11.5, 1.5) [roundnode] {};
  \node (dhmid2) at (12, 1.25) [roundnode] {};
  \node (e2) at (11.5, 0.5) [roundnode] {};
  \node (ehmid2) at (12, 0.75) [roundnode] {};
  \node (f2) at (11.5, -0.5) [roundnode] {};
  \node (fimid2) at (12, -0.75) [roundnode] {};
  \node (g2) at (11.5, -1.5) [roundnode] {};
  \node (gimid2) at (12, -1.25) [roundnode] {};
  \node (h2) at (12.5, 1) [roundnode] {};
  \node (hjmid2) at (13, 0.5) [roundnode] {};
  \node (i2) at (12.5, -1) [roundnode] {};
  \node (ijmid2) at (13, -0.5) [roundnode] {};
  \node (j2) at (13.5, 0) [roundnode] {};
  \draw[->] (a2)--(abmid2);
  \draw[->] (abmid2)--(b2);
  \draw[->] (a2)--(acmid2);
  \draw[->] (acmid2)--(c2);
  \draw[->] (b2)--(bdmid2);
  \draw[->] (bdmid2)--(d2);
  \draw[->] (b2)--(bemid2);
  \draw[->] (bemid2)--(e2);
  \draw[->] (c2)--(cfmid2);
  \draw[->] (cfmid2)--(f2);
  \draw[->] (c2)--(cgmid2);
  \draw[->] (cgmid2)--(g2);
  \draw[->] (d2)--(dhmid2);
  \draw[->] (dhmid2)--(h2);
  \draw[->] (e2)--(ehmid2);
  \draw[->] (ehmid2)--(h2);
  \draw[->] (f2)--(fimid2);
  \draw[->] (fimid2)--(i2);
  \draw[->] (g2)--(gimid2);
  \draw[->] (gimid2)--(i2);
  \draw[->] (h2)--(hjmid2);
  \draw[->] (hjmid2)--(j2);
  \draw[->] (i2)--(ijmid2);
  \draw[->] (ijmid2)--(j2);

  \node[] at (11.5, -2) {$G_2$};

  \draw[->] (15, -1.5) -- (15, 1.5);
  \draw[->] (15, -1.5) -- (18, -1.5);
  \node[rotate=90] at (14.5, 0) {\scriptsize Death};
  \node at (16.5, -2) {\scriptsize Birth};
  \draw[dashed, gray] (15, -1.5) -- (18, 1.5);
  \draw[densely dotted, blue] (15.6, -1.5) -- (15.6, 0.9) -- (15, 0.9);
  \draw[densely dotted, blue] (15.6, -0.3) -- (15, -0.3);
  \node[] at (15.6, -1.65) {\scriptsize $0.5$};
  \node[] at (14.9, -0.3) {\scriptsize $1$};
  \node[] at (14.9, 0.9) {\scriptsize $2$};
  \node at (15.6, 0.9) [bluenode, label={[right]{\scriptsize$\times 1$}}] {};
  \node at (15.6, -0.3) [bluenode, label={[right]{\scriptsize$\times 2$}}] {};


\end{tikzpicture}
  }
  \caption{Persistent path homology of the shortest-path filtration of flow through a bifurcation network, before and after edge subdivision. In $G_1$ all edges have unit weight, in $G_2$ all edges have weight $0.5$.}
  \label{fig:edge_subdivision}
\end{figure}
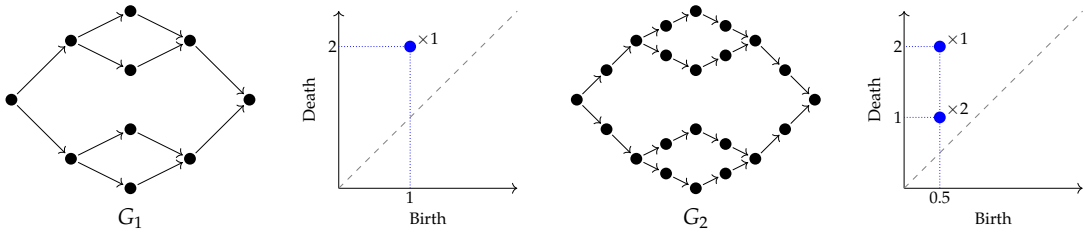

\begin{example}\label{ex:std_pipe}
Choosing $F=F_d$ as above and $C$ to be the regular path complex, we obtain a functor $\mathcal{H}_k:\WDgrF\to\Funct{\Rposet}{\VecCat}$.
The shortest-path quasimetric of a weighted digraph is a directed network and this pipeline measures the persistent path homology of that network.
This pipeline was first considered in~\cite{ChowdhurySIAM} for cycle networks, alongside a stability analysis of persistent path homology for arbitrary directed networks.

In Figure~\ref{fig:edge_subdivision} we apply this pipeline (with $k=1$) to a small bifurcating network $G_1$, a toy model for vasculature networks, in which each edge is given unit weight.
The resulting barcode has a single feature with lifetime $[1, 2)$.
The second network, $G_2$, is obtained by subdividing each edge in $G_1$,  giving all edges weight $0.5$.
The resulting barcode is $\ms{ [0.5, 1), [0.5, 1), [0.5, 2) }$ which has three features.

This example highlights three key issues with this pipeline:
\begin{enumerate}[label=(\alph*)]
  \item the number of features in the barcode changes upon subdivision;
  \item loops bounded by triangles or long squares are `killed' as soon as they are born; and
  \item the birth-time of each feature is an artefact of the `resolution' of the weighted digraph.
\end{enumerate}
A subtler issue arises when we attempt to interpret the diagram.
A feature born at time $t$ is supported on edges of the digraph $G^t$, which may not be edges in the original weighted digraph $G$.
This makes interpretation of features more challenging.
\end{example}

\subsection{Grounded pipeline}\label{sec:descriptor_defin}
We now describe an alteration to the standard pipeline which alleviates these issues by including the underlying digraph $G$ in degree 1 for all $t\in\R$.
The main distinction is that we do not factor through a filtration functor $F$.
Instead, we use $F$ and $C\in\Funct{\Dgr}{\Ch}$ to construct a new functor $\zb{C}_F:\WDgrF\to\Funct{\Rposet}{\Ch}$.
\begin{figure}[H]
  \centering
  \begin{tikzcd}[column sep=large]
\WDgr \arrow[r, "F"]
\arrow[rr, rounded corners, dashed,
       to path={(\tikztostart.north)
                -| ([yshift=4ex]\tikztostart.north)
                -- node[above]{\scriptsize$\zb{C}_F$} ([yshift=4ex]\tikztotarget.north)
                |- (\tikztotarget.north)}]
& \Funct{\Rposet}{\Dgr} \arrow[r, "\Funct{\Rposet}{C}"]
& \Funct{\Rposet}{\Ch} \arrow[r, "\Funct{\Rposet}{H_1}"]
& \Funct{\Rposet}{\VecCat}
  \end{tikzcd}
\end{figure}

In order to define the map on objects, we only need a weaker condition on $C$.

\begin{defin}\label{defin:zb_chain_complex}
  Given a filtration functor $F:\WDgrF\to\Funct{\Rposet}{\Dgr}$, a functor $C:\Incl\Dgr\to\Ch$, $G\in\WDgr$ and $t\in\R$, the chain complex $\zb{C}_\bullet(G, t; F)$ is the top row of the following diagram.
\begin{figure}[H]
    \centering
    \begin{tikzcd}[row sep=small, column sep=small]
        \cdots C_3(F^t G) \arrow[r, "\partial_3"] &
        C_2(F^t G) \arrow[rr, "\inducedch{\iota}\circ\bd_2"] \arrow[rd, "\bd_2"', dotted] & &
        C_1(G\cup F^t G) \arrow[rr, "\bd_1"] & &
        C_0(G\cup F^t G) \cdots \\
       & & C_1(F^t G)\arrow[ru, "\inducedch{\iota}"', dotted, hook]
       \arrow[rr, dotted, "\bd_1"']
       & & C_0(F^t G)\arrow[ru, dotted, "\inducedch{\iota}"', hook] &
    \end{tikzcd}
\end{figure}\noindent
In the above, $\iota:F^t G \hookrightarrow G\cup F^t G$ is the inclusion digraph map, induced by the inclusion vertex map.
Then $\inducedch{\iota}=C(\iota)$ is the image of this map under the functor $C$.
The boundary maps $\bd_k$ are derived either from the chain complex $C_\bullet(F^t G)$ or $C_\bullet(G \cup F^t G)$.

We denote the chain groups as \mdf{$\zb{C}_k(G, t; F)$} and the boundary maps as \mdf{$\zb{\bd}_k^t$}.
When $F$ and $t$ are clear from context, we omit them from notation
\end{defin}

We use the prescript $\zb{\square}$ to denote that this chain complex is \emph{\zbadj};
as we will show in Lemma~\ref{prop:cycles_born_at_0}, after appropriate choices of $F$ and $C$, all degree $1$ homology classes have representatives in the underlying digraph.

\begin{lemma}
For each $t\in \R$ and $G\in\WDgr$, $(\zb{C}_\bullet(G, t; F), \zb{\bd}_\bullet^t)$ is a chain complex.
\end{lemma}
\begin{proof}
Since $\inducedch{\iota}$ is a chain map $C(F^t G)\to C(G\cup F^t G)$ we have
\begin{equation}
    \bd_1 \circ (\inducedch{\iota} \circ \bd_2) = \inducedch{\iota} \circ \bd_1 \circ \bd_2 = 0
\end{equation}
and hence $C_\bullet(G,t;F)$ defines a chain complex.
\end{proof}

\begin{lemma}\label{lem:zb_pers_complex}
Given a filtration functor $F:\WDgrF\to\Funct{\Rposet}{\Dgr}$ and a functor $C:\Incl\Dgr \to \Ch$, the chain complex $\zb{C}_\bullet(G, t; F)$ is functorial in $t$.
\end{lemma}
\begin{proof}
For each $s\leq t$ we require chain maps $\zbinclinduc{ch}{s}{t}:\zb{C}_\bullet(G, s) \to \zb{C}_\bullet(G, t)$
which satisfy the usual functorial axioms in $t$.
First note that $\filtincl{s}{t}: F^s G \to F^t G$ is induced by the inclusion vertex map, which is a restriction of the identity vertex map $V(G) \to V(G)$.
Hence this identity vertex map also yields a digraph map $G \cup F^s G \to G \cup F^t G$, which we also denote $\filtincl{s}{t}$.
Applying the functor $C$ to these digraph maps, we obtain two chain maps, which can be joined as shown in the following diagram.
\begin{figure}[H]
    \centering
    \begin{tikzcd}[row sep=small]
        \cdots C_3(F^s G) \arrow[r, "\bd_3"] \arrow[ddd, hook, "\inducedch{\filtincl{s}{t}}"]
        \arrow[rddd, phantom, "\color{blue}\boxed{A}", description]&
        C_2(F^s G) \arrow[rr, "\inducedch{\iota}\circ\bd_2"]
        \arrow[rd, dotted, "\bd_2"'] \arrow[ddd, hook, "\inducedch{\filtincl{s}{t}}"]
        \arrow[rddd, phantom, "\color{blue}\boxed{B}", description] & &
        C_1(G\cup F^s G) \arrow[r, "\bd_1"] \arrow[ddd, hook, "\inducedch{\filtincl{s}{t}}"]
        \arrow[lddd, phantom, "\color{blue}\boxed{C}", description]
        \arrow[rddd, phantom, "\color{blue}\boxed{D}", description]  &
        C_0(G\cup F^s G) \arrow[ddd, hook, "\inducedch{\filtincl{s}{t}}"]\cdots \\
& & C_1(F^s G) 
        \arrow["\inducedch{\iota}"', ru, dotted, hook] \arrow[d, dotted, hook, "\inducedch{\filtincl{s}{t}}"] & &
         \\
& & C_1(F^t G) 
        \arrow["\inducedch{\iota}", rd, dotted, hook] &
         & \\
        \cdots C_3(F^t G) \arrow[r, "\bd_3"'] &
        C_2(F^t G) \arrow[rr, "\inducedch{\iota}\circ\bd_2"']
        \arrow[ru, dotted, "\bd_2"] & {\phantom{C_1(G\cup F^t G)}} &
        C_1(G\cup F^t G) \arrow[r, "\bd_1"']  &
        C_0(G\cup F^t G)\cdots \\
    \end{tikzcd}
\end{figure}\noindent
Squares $\color{blue}\boxed{A}$ and $\color{blue}\boxed{B}$ and $\color{blue}\boxed{D}$ commute because all vertical maps are components of the same chain map -- either $\inducedch{\filtincl{s}{t}} : C(F^s G) \to C(F^t G)$ or $\inducedch{\filtincl{s}{t}}:C(G \cup F^s G) \to C(G \cup F^t G)$.
Then, in the following diagram, all vertex maps are induced by an inclusion vertex map so the square of morphisms commute.
\begin{figure}[H]
  \centering
  \begin{tikzcd}
    F^s G \arrow[r, hook, "\iota"] \arrow[d, hook, "\filtincl{s}{t}"'] &
    G \cup F^s G \arrow[d, hook, "\filtincl{s}{t}"] \\
    F^t G \arrow[r, hook, "\iota"'] & G \cup F^t G
  \end{tikzcd}
\end{figure}\noindent
Applying the functor $C$ to this diagram and restricting to degree $1$, we obtain square $\color{blue}\boxed{C}$ which must, therefore, commute.
Hence, we obtain a chain map, as required.
Functoriality follows because the components of this new chain map are themselves components of chain maps induced by functors $F(G)\in\Funct{\Rposet}{\Dgr}$ and  $C\in\Funct{\Incl\Dgr}{\Ch}$.
\end{proof}

\begin{defin}\label{defin:gcf}
  Given a filtration functor $F :\WDgrF\to\Funct{\Rposet}{\Dgr}$ and a functor $C:\Incl\Dgr\to\Ch$,
  the map $\mdf{\zb{C}_F}: \Obj(\WDgr) \to\Obj(\Funct{\Rposet}{\Ch})$ is given on objects by 
  \begin{equation}
  \zb{C}_F(G)(t) \defeq (\zb{C}_\bullet(G, t; F), \zb{\bd}^t_\bullet)
  \end{equation}
  and the morphism $\zb{C}_F(G)(s\leq t)$ is as constructed in the proof of Lemma~\ref{lem:zb_pers_complex}.
\end{defin}

\begin{notation}
  \begin{enumerate}[label=(\alph*)]
    \item We denote the induced map on chain complexes by
      $\mdf{\zbinclinduc{ch}{s}{t}}\defeq \zb{C}_F(G)(s \leq t)$.
    \item In each homology degree $k$, we denote the induced map on homology by
      $\mdf{\zbinclinduc{hom}{s}{t}}\defeq \Funct{\Rposet}{H_k}(\zbinclinduc{ch}{s}{t})$.
  \end{enumerate}
\end{notation}

Thanks to Lemma~\ref{lem:zb_pers_complex}, $\zb{C}_F$ gives us a map $\Obj(\WDgr)\to\Obj(\Funct{\Rposet}{\Ch})$, which we can compose with homology to obtain a persistent vector space.
Under additional functorial assumptions on $C$, when we restrict to the appropriate category $\zb{C}_F$ becomes a functor.

\begin{theorem}\label{thm:grd_functoriality}
Given a filtration functor $F:\WDgrF\to\Funct{\Rposet}{\Dgr}$ and a functor $C:\Dgr\to\Ch$, $\zb{C}_F$ is a functor $\zb{C}_F:\WDgrF \to \Funct{\Rposet}{\Ch}$.
\end{theorem}
\begin{proof}
Given $f\in\MorXY{\WDgrF}{G}{H}$ and $t\in \R$, we need a chain map $\zbdiginduc{ch}{f}:\zb{C}_\bullet(G, t)\to\zb{C}_\bullet(H, t)$.
Moreover, $f$ must satisfy the usual functorial axioms, as well as 
commute with $\zbinclinduc{ch}{s}{t}$, i.e.\ for any $s\leq t$ we need
\begin{equation}
  \zbdiginduc{ch}{f} \circ \zbinclinduc{ch}{s}{t} = \zbinclinduc{ch}{s}{t} \circ \zbdiginduc{ch}{f}.
 \label{eq:func_commute}
\end{equation}
Note that $f$ is given by a vertex map $f: V(G) \to V(H)$ which induces a digraph map $G \to H$.
Moreover, since $f$ is $F$-compatible, it induces digraph maps $F^t G \to F^t H$.
Therefore $f$ must also induce digraph maps $G \cup F^t G \to H \cup F^t H$.
Applying the functor $C$ to these digraph maps, we obtain two chain maps, which can be joined as shown in the following diagram.
\begin{figure}[H]
    \centering
    \begin{tikzcd}[row sep=small]
        \cdots C_3(F^t G) \arrow[r, "\bd_3"] \arrow[ddd, "\inducedch{f}"]
        \arrow[rddd, phantom, "\color{blue}\boxed{A}", description]&
        C_2(F^t G) \arrow[rr, "\inducedch{\iota}\circ\bd_2"]
        \arrow[rddd, phantom, "\color{blue}\boxed{B}", description]
        \arrow[rd, dotted, "\bd_2"'] \arrow[ddd, "\inducedch{f}"] & &
        C_1(G\cup F^t G) \arrow[r, "\bd_1"] \arrow[ddd, "\inducedch{f}"] 
        \arrow[lddd, phantom, "\color{blue}\boxed{C}", description]
        \arrow[rddd, phantom, "\color{blue}\boxed{D}", description]  &
        C_0(G\cup F^t G) \arrow[ddd, "\inducedch{f}"]\cdots \\
& & C_1(F^t G) 
        \arrow["\inducedch{\iota}"', ru, dotted, hook] \arrow[d, dotted, "\inducedch{f}"] &
          & \\
& & C_1(F^t H) 
        \arrow["\inducedch{\iota}", rd, dotted, hook] &
         & \\
        \cdots C_3(F^t H) \arrow[r, "\bd_3"'] &
        C_2(F^t H) \arrow[rr, "\inducedch{\iota}\circ\bd_2"']
        \arrow[ru, dotted, "\bd_2"] & {\phantom{C_1(G\cup F^t G)}} &
        C_1(H\cup F^t H) \arrow[r, "\bd_1"'] &
        C_0(H\cup F^t H)\cdots \\
    \end{tikzcd}
\end{figure}\noindent
Squares $\color{blue}\boxed{A}$, $\color{blue}\boxed{B}$ and $\color{blue}\boxed{D}$ all commute because all vertical maps are parts of the same chain maps -- either $\inducedch{f}:C(F^t G) \to C(F^t H)$ or $\inducedch{f}:C(G \cup F^t G) \to C(H \cup F^t H)$.
Since the inclusion digraph map $\iota: G \to G \cup F^t G$ is induced by an inclusion vertex map, it certainly commutes with $f$ and hence square $\color{blue}\boxed{C}$ commutes.
Hence, the whole diagram commutes and yields a chain map.

This construction is functorial thanks to the functoriality of the underlying induced chain maps.
Finally, since $\filtincl{s}{t}$ is always an inclusion vertex map, $f$ must commutes with each $\filtincl{s}{t}$.
Hence, the two morphisms commute past each other as required by equation (\ref{eq:func_commute}).
\end{proof}

\begin{notation}
Given a morphism $f\in\Mor(\WDgrF)$, we denote the induced map on chain complexes, constructed above, by $\mdf{\zbdiginduc{ch}{f}}\defeq \zb{C}_F(f)$ and the induced map on homology in degree $k$ by $\mdf{\zbdiginduc{hom}{f}}\defeq\Funct{\Rposet}{H_k}(\zbdiginduc{ch}{f})$. 
\end{notation}

\subsection{Definition of \grpph}\label{sec:grpph_defin}

In order to investigate properties and stability of the grounded pipeline, we make choice for both $F$ and $C$.
As we have already discussed, since we interpret edge-weights as a measure of distance, a natural choice for $F$ is the shortest-path filtration.
Choices for $C$ include the regular path complex, non-regular path complex and the directed flag complex.
However, the latter two constructions are \emph{not} functors $\Dgr\to\Ch$.
\textbf{Henceforth, for the rest of the paper, we fix $F$ to be the shortest-path filtration and $C$ to be the regular path complex},
\begin{equation}
 F = F_d
 \quad\text{ and }\quad
 C=\Omega .
\end{equation}
Since $F$ is fixed, we will largely remove it from notation.
We also use $C$ instead of $\biv[reg]$.

\begin{lemma}\label{lem:wdgrfd_is_contwdgr}
The $F_d$-compatible category of weighted digraphs is the contraction category of weighted digraphs (see Definition~\ref{defin:dgr_cats}), $\WDgrFd=\Cont\WDgr$.
\end{lemma}
\begin{proof}
First note $f\in\Mor(\WDgr)$ is precisely a vertex map $f:V(G) \to V(H)$ which induces a digraph map $G\to H$.
A vertex map $f:V(G) \to V(H)$ induces a digraph map $G^t \to H^t$ for every $t\in\R$ if and only if
\begin{equation}
  d(f(i), f(j)) \leq d(i, j)
\end{equation}
for every $i, j \in V(G)$.
Hence, $f\in\Mor(\WDgr)$ is $F_d$-compatible if and only if it is a contraction map.
So a morphism $f\in\Mor(\WDgrFd)$ is precisely a contraction digraph map $G \to H$.
\end{proof}

Now that we understand the category $\WDgrFd$, applying Theorem~\ref{thm:grd_functoriality} yields a functor $\zb{C}:\Cont\WDgr\to\Funct{\Rposet}{\Ch}$.
Taking the first homology yields a persistent vector space in a functorial way; this functor is our proposed invariant for weighted digraphs.

\begin{definthm}\label{defin:grpph}
  \mdf{Grounded persistent path homology (\grpph)} is the functor
  \begin{equation}
 \mdf{\zb{\mathcal{H}}_1} \defeq \Funct{\Rposet}{H_1} \circ \zb{C} : \Cont\WDgr \to \PersVec
    \end{equation}
    from the contraction category of weighted digraphs (see Definition~\ref{defin:dgr_cats}) to the category of persistent vector spaces (see Definition~\ref{defin:persvec}).
\end{definthm}

\begin{notation}\label{notation:grd_notation}
  Given $G\in\WDgr$, in degree $k$ at filtration step $t$, we denote 
  \begin{align}
    &\mdf{\text{the space of \zbadj\ }k\text{-cycles}}\text{ by}   &\mdf{\zb{Z}(G, t)} &\defeq \ker \zb{\bd}_{k}^t;\\
    &\mdf{\text{the space of \zbadj\ }k\text{-boundaries}}\text{ by}   &\mdf{\zb{B}(G, t)} &\defeq \im \zb{\bd}_{k+1}^t \\
    &\mdf{\text{the (degree }k\text{) \zbadj\ homology}}\text{ by}  &\mdf{\zb{H}_k(G, t) } &\defeq \frac{\zb{Z}_k(G, t)}{\zb{B}_k(G, t)}.
  \end{align}
\end{notation}

\begin{rem}
Note that for any $t\in \R$, 
\begin{align*}
  k > 1 &\implies \zb{H}_k(G, t) \cong H_k(G^t), \\
  k < 1 &\implies \zb{H}_k(G, t) \cong H_k(G \cup G^t).
\end{align*}
Therefore, the only new homology occurs in degree $k=1$, since it compares $1$-cycles in $G\cup F^t G$ with $1$-boundaries from $G^t$.
This justifies our focus on degree $1$ homology in Definition/Theorem~\ref{defin:grpph}.
\end{rem}

\begin{figure}[htbp]
  \centering
  \resizebox{0.5\textwidth}{!}{%
    \begin{tikzpicture}[
  roundnode/.style={circle, fill=black, minimum size=4pt},
  bluenode/.style={circle, fill=blue, minimum size=1pt},
	squarenode/.style={fill=black, minimum size=4pt},
	inner sep=2pt,
	outer sep=1pt
  ]
  \node (a) at (0, 0) [roundnode] {};
  \node (b) at (1, 1) [roundnode] {};
  \node (c) at (1, -1) [roundnode] {};
  \node (d) at (2, 1.5) [roundnode] {};
  \node (e) at (2, 0.5) [roundnode] {};
  \node (f) at (2, -0.5) [roundnode] {};
  \node (g) at (2, -1.5) [roundnode] {};
  \node (h) at (3, 1) [roundnode] {};
  \node (i) at (3, -1) [roundnode] {};
  \node (j) at (4, 0) [roundnode] {};
  \draw[->] (a)--(b);
  \draw[->] (a)--(c);
  \draw[->] (b)--(d);
  \draw[->] (b)--(e);
  \draw[->] (c)--(f);
  \draw[->] (c)--(g);
  \draw[->] (d)--(h);
  \draw[->] (e)--(h);
  \draw[->] (f)--(i);
  \draw[->] (g)--(i);
  \draw[->] (h)--(j);
  \draw[->] (i)--(j);

  \draw[blue!50, rounded corners] (1.3, 1) -- (2, 1.3) -- (2.7, 1) -- (2, 0.7) -- cycle;
  \draw[blue!50, rounded corners] (1.3, -1) -- (2, -1.3) -- (2.7, -1) -- (2, -0.7) -- cycle;
  \draw[green!50, rounded corners] (0.3, 0) -- (1, 0.8) -- (2, 0.3) -- (3, 0.8) -- (3.7, 0) -- (3, -0.8) -- (2, -0.3) -- (1, -0.8) -- cycle;
  \draw[red!50, rounded corners] (-0.3, 0) -- (1, 1.2) -- (2, 1.7) -- (3, 1.2) -- (4.3, 0) -- (3, -1.2) -- (2, -1.7) -- (1, -1.2) -- cycle;

  \node[] at (2, -2) {$G_1$};

  \node (a2) at (5.5, 0) [roundnode] {};
  \node (abmid2) at (6, 0.5) [roundnode] {};
  \node (acmid2) at (6, -0.5) [roundnode] {};
  \node (b2) at (6.5, 1) [roundnode] {};
  \node (bdmid2) at (7, 1.25) [roundnode] {};
  \node (bemid2) at (7, 0.75) [roundnode] {};
  \node (c2) at (6.5, -1) [roundnode] {};
  \node (cfmid2) at (7, -0.75) [roundnode] {};
  \node (cgmid2) at (7, -1.25) [roundnode] {};
  \node (d2) at (7.5, 1.5) [roundnode] {};
  \node (dhmid2) at (8, 1.25) [roundnode] {};
  \node (e2) at (7.5, 0.5) [roundnode] {};
  \node (ehmid2) at (8, 0.75) [roundnode] {};
  \node (f2) at (7.5, -0.5) [roundnode] {};
  \node (fimid2) at (8, -0.75) [roundnode] {};
  \node (g2) at (7.5, -1.5) [roundnode] {};
  \node (gimid2) at (8, -1.25) [roundnode] {};
  \node (h2) at (8.5, 1) [roundnode] {};
  \node (hjmid2) at (9, 0.5) [roundnode] {};
  \node (i2) at (8.5, -1) [roundnode] {};
  \node (ijmid2) at (9, -0.5) [roundnode] {};
  \node (j2) at (9.5, 0) [roundnode] {};
  \draw[->] (a2)--(abmid2);
  \draw[->] (abmid2)--(b2);
  \draw[->] (a2)--(acmid2);
  \draw[->] (acmid2)--(c2);
  \draw[->] (b2)--(bdmid2);
  \draw[->] (bdmid2)--(d2);
  \draw[->] (b2)--(bemid2);
  \draw[->] (bemid2)--(e2);
  \draw[->] (c2)--(cfmid2);
  \draw[->] (cfmid2)--(f2);
  \draw[->] (c2)--(cgmid2);
  \draw[->] (cgmid2)--(g2);
  \draw[->] (d2)--(dhmid2);
  \draw[->] (dhmid2)--(h2);
  \draw[->] (e2)--(ehmid2);
  \draw[->] (ehmid2)--(h2);
  \draw[->] (f2)--(fimid2);
  \draw[->] (fimid2)--(i2);
  \draw[->] (g2)--(gimid2);
  \draw[->] (gimid2)--(i2);
  \draw[->] (h2)--(hjmid2);
  \draw[->] (hjmid2)--(j2);
  \draw[->] (i2)--(ijmid2);
  \draw[->] (ijmid2)--(j2);

  \node[] at (7.5, -2) {$G_2$};

  \draw[blue!50, rounded corners] (6.8, 1) -- (7.5, 1.3) -- (8.2, 1) -- (7.5, 0.7) -- cycle;
  \draw[blue!50, rounded corners] (6.8, -1) -- (7.5, -1.3) -- (8.2, -1) -- (7.5, -0.7) -- cycle;
  \draw[green!50, rounded corners] (5.8, 0) -- (6.5, 0.8) -- (7.5, 0.3) -- (8.5, 0.8) -- (9.2, 0) -- (8.5, -0.8) -- (7.5, -0.3) -- (6.5, -0.8) -- cycle;
  \draw[red!50, rounded corners] (5.2, 0) -- (6.5, 1.2) -- (7.5, 1.7) -- (8.5, 1.2) -- (9.8, 0) -- (8.5, -1.2) -- (7.5, -1.7) -- (6.5, -1.2) -- cycle;

\end{tikzpicture}
  }
  \caption{A bifurcation network, before and after subdivision; all edges in $G_1$ have unit weight, all edges in $G_2$ have weight $0.5$ (as in Figure~\ref{fig:edge_subdivision}).
    Highlighted in red, green and blue are circuits whose representatives generate $\zbpershom{G_i}$.}
  \label{fig:edge_subdivision_new}
\end{figure}
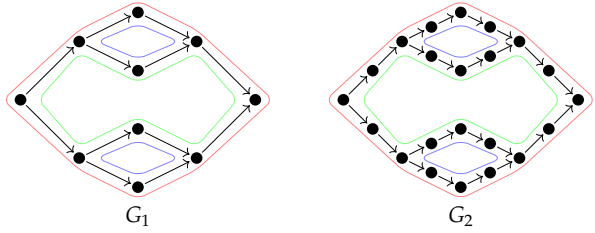

\begin{example}\label{ex:grd_pipe}
  In Figure~\ref{fig:edge_subdivision_new}, we consider again the bifurcating network example of Figure~\ref{fig:edge_subdivision}, in which all edges have weight $1$.
  We see the barcodes of the grounded persistent homology are both
  \begin{equation}
    \thedesc{G_1}=\thedesc{G_2} =
    \ms{
      [0, 1), [0, 1), [0, 2)
    }.
  \end{equation}
  In $G_1$ the two $[0, 1)$ features correspond to the smaller $4$-node circuits in the centre of the network, coloured in blue.
  These circuits birth homological cycles in $G\cup G^t$ at $t=0$, which are then killed by long squares when the edges appear in $G^t$ at $t=1$.

  The $[0, 2)$ feature corresponds to the large inner circuit (coloured in green).
  Again this circuit births a homological cycle at $t=0$ which then becomes null-homologous at $t=2$ when shortcut edges give rise to a new long square.

  The outer red cycle is a linear combination of the inner green and blue cycles, hence it does not give rise to a fourth feature in the barcode.
  Moreover, at $t=1$ the red and green cycles becomes homologous.

  In $G_2$ the features correspond to the same circuits (once subdivided).
\end{example}

\section{Interpretation of GrPPH}\label{sec:properties}
\subsection{Decreasing Betti curves}\label{sec:decreasing_curves}

In the first example we saw (Example~\ref{ex:grd_pipe}) we saw that all features were born at $t=0$.
Indeed, this is always the case and $\zbhom{G}{0}$ is in fact the real cycle space of $\underlying{G}$.

\begin{lemma}\label{lem:paths_homologous}
  Given a digraph $G=(V, E, w)$, two distinct nodes $a, b \in V$ and a trail $p:a\leadsto b$, then for all $t \geq \pathlen(p)$
  \begin{equation}
\repres{p}=\sum_{\tau\in E(p)}\tau = ab \pmod{\zbbdrs{G}{t}}.
  \end{equation}
\end{lemma}
\begin{proof}
Fix arbitrary $t\geq \pathlen(p)$ and denote the vertices of the path as $a=v_0, \dots, v_m = b$.
Whenever $i<j$ we can truncate $p$ to obtain a path $v_i\leadsto v_j$ of length at most $t$ and so $(v_i , v_j) \in E(G^t)$.
Hence, whenever $i< j < k$, there is a directed triangle $v_i v_j v_k \in C_2(G^t)$ and hence
$v_i v_k = v_i v_j + v_j v_k \pmod{\zbbdrs{G}{t}}$.
Therefore, inductively we can write
\begin{equation}
  ab = v_0 v_m  = v_0 v_1 + v_1 v_m = v_0 v_1 + v_1 v_2 + v_2 v_m = \dots = \sum_{i=1}^m v_{i-1} v_i \pmod{\zbbdrs{G}{t}}
\end{equation}
as required.
\end{proof}

\begin{prop}\label{prop:cycles_born_at_0}
Fix a weighted digraph $(G, w)$ and $t\geq 0$.
For any cycle $v\in\zbcycles{G}{t}$, there is an initial cycle $v'\in\zbcycles{G}{0}$, supported on the edge in $G$, such that $v$ is homologous to $\zbinclinduc{ch}{0}{t} v'$.
\end{prop}
\begin{proof}
Given any edge $\tau=(a, b)\in E(G^t)$, there is a path $p:a \leadsto b$ in $G$ of length at most $t$.
Denoting the edges of $p$ by $(\tau_1, \dots, \tau_m)$,
Lemma~\ref{lem:paths_homologous} tells us, $\tau = \sum_{i=1}^m \tau_i \pmod{\zbbdrs{G}{t}}$.
Now, since $\tau_i \in E(G)$, we see $\zbinclinduc{ch}{0}{t}\tau_i = \tau_i$ for each edge in $p$.
\end{proof}

\begin{cor}\label{cor:decreasing_curves}
  Given $G\in\WDgr$,
  \begin{enumerate}[label=(\alph*)]
    \item any interval in $\thedesc{G}$ has birth time $0$; 
    \item $\zbhom{G}{0}$ has a basis of simple undirected circuits; and
    \item $\card{\Dgm(\zb{\mathcal{H}}_1(G))}$ coincides with the circuit rank of the underlying undirected graph.
  \end{enumerate}
\end{cor}
\begin{proof}
The first point follows immediately from Lemma~\ref{prop:cycles_born_at_0}. 
To see the final two points, consider the chain complex, $\zb{C}_\bullet(G, 0)$ at the start of the filtration.
Since $G^0$ has no edges, the chain complex is simply
\begin{figure}[H]
  \centering
  \begin{tikzcd}
  \cdots 0 \arrow[r] & 0 \arrow[r] & C_1(G) \arrow[r] & C_0(G) \cdots
  \end{tikzcd}
\end{figure}\noindent
Since $G$ is an orientation of $\underlying{G}$, the first homology of this chain complex is precisely the real cycle space of $\underlying{G}$.
Fix an arbitrary spanning forest $T$ of $\underlying{G}$, and label the remaining edges $e_1, \dots, e_k$.
Note $k$ is the circuit rank of $\underlying{G}$.
Let $p_i$ denote the simple undirected circuit in $G$ which traverses $e_i$ and then returns to $\st(e_i)$ through the unique path in $T$.
Then $\{ \repres{p_1}, \dots, \repres{p_k} \}$ is a basis for $\zbhom{G}{0}$.
\end{proof}

\subsection{Circuit lifetimes}\label{sec:comparison}

While all features in the barcode (and hence all cycles) are born at time $t=0$, their death times generally differ.
We can assign a death-time to any cycle $v\in\zbcycles{G}{0}$, as the first time $v$ becomes null-homologous.

\begin{defin}
  Given $v\in\zbcycles{G}{0}$, the \mdf{death-time} of $v$ is
  \begin{equation}
    \mdf{\death(v)} \defeq \inf
    \left\{ 
      t \geq 0 \rmv
      \zbinclinduc{hom}{0}{t} [v] = 0
    \right\}
  \end{equation}
  where we let $\death(v)\defeq\infty$ if there is no such $t$.
  The \mdf{lifetime} of $v$ is the interval $\mdf{\lifetime(v)} \defeq [0, \death(v))$.
\end{defin}

\begin{rem}
  Let $p$ be an undirected circuit in $G$ and $p'$ the same circuit, traversed in the opposite direction so that $\repres{p'} = - \repres{p}$.
  Since $\zbinclinduc{hom}{0}{t}$ is linear, $\death(\repres{p})=\death(\repres{p'})$.
  Also, since $\repres{p}$ does not depend on the starting vertex of $p$, neither does $\death(\repres{p})$.
\end{rem}

This pipeline gives us a method for associating a lifetime to an undirected circuit in $G$ which is `geometric' in the sense that it does not depend on the starting vertex or direction.
The length of this lifetime gives us a `size' to the circuit from the perspective of the filtration.
Since we use the shortest-path filtration, we interpret this size as the time it takes for the flow to `fill in' the circuit.

\begin{lemma}\label{lem:death_bound}
Given $G=(V, E, w)\in\WDgr$ and two directed paths $p_1, p_2: a \leadsto b$ between distinct vertices $a, b \in V$, let $p_c$ denote the undirected circuit which traverses $p_1$ forwards and then $p_2$ in reverse.
For $i= 1, 2$, define
\begin{equation}
 h_i \defeq \min
 \left\{ 
   t \geq 0 \rmv
   \exists v_i \in V\text{ along } p_i \text{ such that } d(a, v_i) \leq t \text{ and } d(v_i, b) \leq t
 \right\}.
\end{equation}
Then $\death(p_c) \leq \max( h_1, h_2 )$.
\end{lemma}
\begin{proof}
Denote $T\defeq \max(h_1, h_2)$.
First assume that there are at least 2 edges in each $p_i$.
Then, by the definition of $h_i$, there exists $v_i\in V\setminus\left\{ a, b \right\}$ along each $p_i$ such that $d(a, v_i) \leq T$ and $d(v_i, b) \leq T$.
Hence there is a long square $a v_1 b - a v_2 b \in \zb{C}_2(G, T)$.
By Lemma~\ref{lem:paths_homologous},
\begin{equation}
  0=\zb{\bd}_2 (a v_1 b - a v_2 b) = \repres{p_c}
  \pmod{\zbbdrs{G}{T}}.
\end{equation}

Finally, if $p_1$ contains one edge then $h_1 = \pathlen(p_1)$ so $d(a, b) \leq h_1 \leq T$.
The definition of $h_2$ ensures there is $v_2\in V$ along $p_2$ such that $d(a, v_2), d(v_2, b) \leq T$.
Hence there is a directed triangle $a v_2 b \in \zb{C}_2(G, T)$.
By Lemma~\ref{lem:paths_homologous},
\begin{equation}
  0=\zb{\bd}_2 (a v_2 b) = \repres{p_c}
  \pmod{\zbbdrs{G}{T}}
\end{equation}
which concludes the proof.
\end{proof}

\begin{figure}[hptb]
  \centering
  \begin{tikzpicture}[
  roundnode/.style={circle, fill=black, minimum size=4pt},
	squarenode/.style={fill=black, minimum size=4pt},
	inner sep=2pt,
	outer sep=1pt
  ]

  \node (a) at (0, 0) [roundnode, label=left:$a$] {};
  \node (b) at (2, 1) [roundnode, label=above:$b$] {};
  \node (c) at (2, -1) [roundnode, label=below:$c$] {};
  \node (d) at (4, 0) [roundnode, label=right:$d$] {};
  \node (e) at (1.2, 0) [roundnode, label=right:$e$] {};
  \node (f) at (2.8, 0) [roundnode, label=left:$f$] {};
  \draw[->, red] (a) -- (b) node[midway, above, sloped] {\tiny $10$};
  \draw[->, red] (a) -- (c) node[midway, below, sloped] {\tiny $10$};
  \draw[->, red] (b) -- (d) node[midway, above, sloped] {\tiny $10$};
  \draw[->, red] (c) -- (d) node[midway, below, sloped] {\tiny $10$};
  \draw[->] (a) -- (e) node[midway, above, sloped, xshift=5pt] {\tiny $1$};
  \draw[->] (e) -- (b) node[midway, below, sloped] {\tiny $1$};
  \draw[->] (e) -- (c) node[midway, above, sloped] {\tiny $1$};
  \draw[->] (b) -- (f) node[midway, below, sloped] {\tiny $1$};
  \draw[->] (c) -- (f) node[midway, above, sloped] {\tiny $1$};
  \draw[->] (f) -- (d) node[midway, above, sloped, xshift=-5pt] {\tiny $1$};

\end{tikzpicture}
  \caption{A simple example of a weighted digraph for which the bound of Lemma~\ref{lem:death_bound} fails to be sharp.}\label{fig:death_bound}
\end{figure}
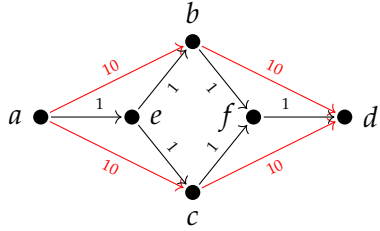

\begin{example}
  Note that the bound of Lemma~\ref{lem:death_bound} is by no means sharp.
  Consider for example Figure~\ref{fig:death_bound}.
  Let $p_1$ be the outer red path $(a, b, d)$, $p_2$ the lower red path $(a, c, d)$ and $p_c$ the undirected circuit which traverse $p_1$ forward then $p_2$ in reverse.
  Then $h_1 = h_2 = 10$ but $\death(\repres{p_c}) = 2$.

  To see this is the correct death time, first note that $\zbhom{G}{t}$ can only change at integer values.
  At $t=1$, the only edges present in $G^t$ are the black ones drawn in in Figure~\ref{fig:death_bound}.
  Hence, $C_2(G, 1)$ is generated by the long square $ebf - ecf$, whose boundary is not $\repres{p_c}$.

  However, at $t=2$ the edges $(a, b), (b, d), (a, c)$ and $(c, d)$ also appear in $G^t$.
  These generate additional long squares and directed triangles.
  In particular $abd - acd \in C_2(G, 2)$ and
  \begin{equation}
 \zb{\bd}_2(abd - acd) = \repres{p_c}.
  \end{equation}
\end{example}

\subsection{Representatives}\label{sec:representatives}

After computing persistent homology, it is common to compute cycles with represent the features of the barcode.
In general, representatives are not unique and may be quite complicated.
In practice, one can often compute representatives with integer (and even unit) coefficients~\cite{li2021minimal}.
These representatives are frequently used for interpreting features (e.g. \cite{Benjamin2022, Sizemore2018}).

\begin{defin}
A \mdf{persistence basis} for $\zbpershom{G}$ is a choice of initial cycles 
$B=\left\{ 
  b_i \in \zbcycles{G}{0}
\right\}$
such that for each $t\geq 0$, the set
$\left\{ 
  \zbinclinduc{hom}{0}{t}[b_i] 
\right\}\setminus\left\{ 0 \right\}$
yields a basis for $\zbhom{G}{t}$.
We call elements of a persistence basis \mdf{representatives}.
\end{defin}

\begin{lemma}
Given any $G\in\WDgr$, a persistence basis for $\zbpershom{G}$ always exists.
\end{lemma}
\begin{proof}
This follows from the structure theorem (Theorem~\ref{thm:pvs_decompose}) and Corollary~\ref{cor:decreasing_curves}.
\end{proof}

Obtaining a persistence basis $B\subseteq\zbcycles{G}{0}$ is desirable because the constituent cycles \emph{represent} the features of the barcode, in the following sense.
If the barcode is $\thedesc{G} = \ms{ I_1, \dots, I_m }$ 
then there is an ordering on the cycles $B = \left\{  b_1, \dots, b_m \right\}$ such that $\lifetime(b_i) = I_i$ and
\begin{equation}
\bigoplus_{i=1}^m P(I_i)\cong 
 \zbpershom{G} .
\end{equation}
Moreover, the isomorphism $\phi : \oplus_{i=1}^m P(I_i) \to \zbpershom{G}$ is given by mapping 
\begin{equation}
1\in I_i(t) \mapsto \zbinclinduc{hom}{0}{t}[b_i] \text{ whenever } t \in I_i.
\end{equation}
This mapping gives an isomorphism because $\{\zbinclinduc{hom}{0}{t}[b_i]\}\setminus\{0\}$ is always a basis for $\zbhom{G}{t}$.
In this sense, the representatives in $B$ generate $\zbpershom{G}$.

Representatives live in $\zbcycles{G}{0}$ so they are just $\R$-linear combinations of edges in $G$.
However, a priori, the coefficients of these linear combinations may be arbitrarily complicated.
The goal of this section is to show that grounded persistent homology always admits a \emph{geometrically interpretable} persistence basis, in the following sense.

\begin{theorem}\label{thm:rep_of_cycles}
Given $G\in \WDgr$ with circuit rank $m$ there exist undirected circuits $p_1, \dots, p_m$ in $G$, such that $\{\repres{p_1},\dots, \repres{p_m}\}$ is a persistence basis for $\zbpershom{G}$.
\end{theorem}

\begin{figure}[hptb]
  \centering
  \begin{tikzpicture}[
  roundnode/.style={circle, fill=black, minimum size=4pt},
	squarenode/.style={fill=black, minimum size=4pt},
	inner sep=2pt,
	outer sep=1pt
  ]

  \node (a) at (1, 2) [roundnode, label=above:$a$] {};
  \node (b) at (0, 1) [roundnode, label=left:$b$] {};
  \node (c) at (2, 1) [roundnode, label=right:$c$] {};
  \node (d) at (1, 0) [roundnode, label=below:$d$] {};
  \draw[->] (a) -- (b) node[midway, above, sloped] {\tiny $1$};
  \draw[->] (a) -- (c) node[midway, above, sloped] {\tiny $1$};
  \draw[->, red] (b) -- (d) node[midway, below, sloped] {\tiny $2$};
  \draw[->, red] (c) -- (d) node[midway, below, sloped] {\tiny $2$};
  \draw[->] (b) -- (c) node[midway, above, sloped] {\tiny $1$};

  \node (a2) at (5, 2) [roundnode, label=above:$a$] {};
  \node (b2) at (4, 1) [roundnode, label=left:$b$] {};
  \node (c2) at (6, 1) [roundnode, label=right:$c$] {};
  \node (d2) at (5, 0) [roundnode, label=below:$d$] {};
  \draw[->] (a2) -- (b2) node[midway, above, sloped] {\tiny $1$};
  \draw[->] (a2) -- (c2) node[midway, above, sloped] {\tiny $1$};
  \draw[->] (b2) -- (d2) node[midway, below, sloped] {\tiny $1$};
  \draw[->] (c2) -- (d2) node[midway, below, sloped] {\tiny $1$};
  \draw[->, red] (b2) -- (c2)  node[midway, above, sloped] {\tiny $2$};

  \node at (1, -1) {$G_1$};
  \node at (5, -1) {$G_2$};
\end{tikzpicture}
  \caption{
    Two weighted digraphs with the same underlying digraph but different persistence bases, illustrating that not every circuit basis of $\zbhom{G}{0}$ yields a persistence basis for $\zbpershom{G}$.}\label{fig:representatives_example}
\end{figure}
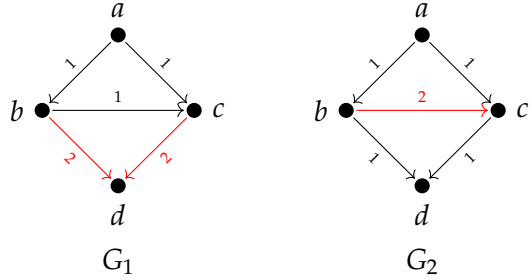

\begin{example}\label{ex:not_any_reps}
First, we note that it does not suffice to chose \emph{any} basis of undirected circuits for $\zbcycles{G}{0}$.
For example, consider the two weighted digraphs pictured in Figure~\ref{fig:representatives_example}.
In both digraphs, ignoring choice of direction, there are three undirected simple circuits, whose representatives we denote
\begin{align}
  \gamma_1^i & \defeq ab + bc - ac, \\
  \gamma_2^i & \defeq bc + cd - bd, \\
  \gamma_3^i & \defeq ab + bd - cd - ac.
\end{align}
where $\gamma_j^i \in \zbcycles{G_i}{0}$.
Note $\gamma_3^i = \gamma_1^i - \gamma_2^i$.
The \grpph\ of these two weighted digraphs is
\begin{equation}
\thedesc{G_1} = \ms{ [0, 2), [0, 1) }\quad\text{and}\quad\thedesc{G_2} =  \ms{ [0, 2), [0, 1) }.
\end{equation}

A persistence basis for $\zbpershom{G_1}$ is $\{\gamma_1^1, \gamma_2^1\}$ with 
$\lifetime(\gamma_1^1) = [0, 1)$ and $\lifetime(\gamma_2^1) = [0, 2)$.
Note that $\{\gamma_2^1, \gamma_3^1\}$ is \emph{not} a persistence basis for $\zbpershom{G_1}$ because $\lifetime(\gamma_3^1) = [0, 2)$.

However $\lifetime(\gamma_1^2) = \lifetime(\gamma_2^2) = [0, 2)$, hence $\{\gamma_1^2, \gamma_2^2\}$ does \emph{not} form a persistence basis for $\zbpershom{G_2}$.
At $t=2$, we see $\zbinclinduc{hom}{0}{2}\gamma_3^2 = 0$ and hence $\zbinclinduc{hom}{0}{2}\gamma_1^2 = \zbinclinduc{hom}{0}{2}\gamma_2^2$.
Instead, a persistence basis for $\zbpershom{G_2}$ is $\{\gamma_2^2, \gamma_3^2\}$.

This illustrates that an arbitrary choice of undirected circuit basis for $\zbhom{G}{0}$ may not yield a persistence basis of $\zbpershom{G}$.
Moreover, a correct choice of basis does not depend only on $\underlying{G}$; we must incorporate information about how cycles in $\zbcycles{G}{0}$ die, in order to choose a persistence basis.
\end{example}

To begin tackling Theorem~\ref{thm:rep_of_cycles}, since $G$ is finite, we note there are finitely critical values $t_1=0, \dots, t_m$ where the chain complex $\zb{C}_\bullet(G, t)$ changes.
Therefore, it suffices to study the following finite persistent chain complex instead.
\begin{figure}[H]
  \centering
  \begin{tikzcd}
\cdots C_2(G^{t_1}) \arrow[r] \arrow[d]
& C_1(G \cup G^{t_1}) \arrow[r] \arrow[d]
& C_0(G \cup G^{t_1}) \cdots \arrow[d]
& & \zbhom{G}{t_1} \arrow[d, two heads, "\zbinclinduc{hom}{t_1}{t_2}"] \\
\cdots C_2(G^{t_2}) \arrow[r] \arrow[d]
& C_1(G \cup G^{t_2}) \arrow[r] \arrow[d]
& C_0(G \cup G^{t_2}) \cdots \arrow[d]
& & \zbhom{G}{t_2} \arrow[d, two heads, "\zbinclinduc{hom}{t_2}{t_3}"] \\
\vdots \arrow[d]
& \vdots \arrow[d]
& \vdots \arrow[d]
& & \vdots \arrow[d, two heads, "\zbinclinduc{hom}{t_{m-1}}{t_m}"] \\
\cdots C_2(G^{t_m}) \arrow[r]
& C_1(G \cup G^{t_m}) \arrow[r]
& C_0(G \cup G^{t_m}) \cdots
& & \zbhom{G}{t_m} \\
  \end{tikzcd}
\end{figure}

To the right of the chain complex we show the induced maps on homology $\zbinclinduc{hom}{t_{i-1}}{t_i}$.
By Lemma~\ref{prop:cycles_born_at_0}, these maps on homology are always surjective.
Our strategy is to find undirected circuit bases for the kernel of each of these maps; Lemma~\ref{lem:submap_circuit_basis} achieves this and is the key result.
We then collect these elements into a basis for $\ker\zbinclinduc{hom}{0}{t_m}$.
Together, these elements form representatives for the homology classes with finite lifetime.
To obtain representatives for the infinite feature, we extend this to a basis for all of $\zbhom{G}{0}$ and show that we obtain a persistence basis.

\begin{lemma}\label{lem:submap_circuit_basis}
For each $i=2, \dots, m$, there is a basis $\{b_1, \dots, b_{k_i}\}$ of $\ker\zbinclinduc{hom}{t_{i-1}}{t_i}$ such that $b_j = \zbinclinduc{hom}{0}{t_{i-1}} [ \repres{p_{i, j}} ]$ for some  undirected circuit $p_{i, j}$ in $G$.
\end{lemma}
\begin{proof}
For notational convenience, we define $r\defeq t_{i-1}$ and $s\defeq t_i$.
When the filtration increases from $t=r$ to $t=s$, some number of edges are added to $G^t$ which yield new generators for both $C_1(G \cup G^t)$ and $C_2(G^t)$.
Our approach is to decompose $\zbinclinduc{ch}{r}{s}$ into a sequence of chain maps.
In the first, all the new generators of $C_1(G \cup C^{s})$ are added, along with sufficient new generators in $C_2(G^{s})$ to make the new edges homologous to a sum of edges already present in $C_1(G\cup G^{r})$.
Therefore, on homology this first map is an isomorphism.
We then add the remaining generators of $C_2(G^{s})$ one at a time in order to find a basis for $\ker\zbinclinduc{hom}{r}{s}$.
Since we are only interested in degree $1$ homology, it suffices to restrict our attention to degrees $0, 1$ and $2$.

Denote the set of new edges $E_{new} \defeq E(G\cup G^{s}) \setminus E(G \cup  G^{r})$.
Given an edge $e=(a, b)\in E_{new}$, there is some directed path $p:a \leadsto b$ in $G$, of length at most $s$.
Moreover this path must have at least one vertex distinct from the endpoints of $e$, otherwise $e \in G$.
Choose arbitrary such $v_e \in V(p)$.
Then $(a, v_e, b)$ is a directed triangle in $G^{s}$ so $w_e \defeq av_e b$ is a new generator of $C_2(G^{s})$.
Repeating this for all new edges we obtain a set of generators
\begin{equation}
  U_{new} \defeq \left\{w_e \rmv e \in E_{new} \right\}
\end{equation}
which were not present in $C_2(G^{r})$ and are linearly independent.
Define $W_0 \defeq \left\langle U_{new} \right\rangle$ and let $Q_0$ denote the degree 1 homology of the chain complex
\begin{figure}[H]
  \centering
  \begin{tikzcd}
     C_2(G^r) \oplus W_0 \arrow[r, "\zb{\bd}_2"]
    & C_1( G \cup G^s ) \arrow[r, "\zb{\bd}_1"]
    & C_0( G \cup G^s ) 
  \end{tikzcd}
\end{figure}\noindent
which is a subcomplex of $\zb{C}_\bullet(G, s)$.

\begin{claim}
  The inclusion chain map
\begin{figure}[H]
  \centering
  \begin{tikzcd}
 C_2(G^{r}) \arrow[r] \arrow[d, "q^0_2"]
& C_1(G \cup G^{r}) \arrow[r] \arrow[d, "q^0_1"]
& C_0(G \cup G^{r})  \arrow[d, "q^0_0"]
& \zbhom{G}{{r}}  \arrow[d, "q^0"', "\cong"] \\
 C_2(G^{r}) \oplus W_0 \arrow[r] 
& C_1(G \cup G^{s}) \arrow[r] 
& C_0(G \cup G^{s}) 
& Q_0 
  \end{tikzcd}
\end{figure}\noindent
induces an isomorphism on degree $1$ homology, $q^0 : \zbhom{G}{r} \to Q_0$.
\end{claim}
\begin{poc}
We define a chain map $\inducedch{c}$ in the opposite direction and a homotopy $P:C_1(G\cup G^s) \to C_2(G^r)\oplus W_0$ such that $c_1 \circ q^0_1 = \id$ while $q^0_1 \circ c_1 - \id = \zb{\bd}_2 \circ P$.
Hence, on degree 1 homology, $\inducedch{c}$ induces an inverse to $q^0$.
The chain map in degrees $k\neq 1$ is given by
\begin{equation}
  c_k(v) \defeq
  \begin{cases}
    v &\text{if }v \in \zb{C}_k(G, r), \\
    0 &\text{otherwise},
  \end{cases}
\end{equation}
and $c_1$ is defined on the basis of $C_1(G \cup G^s)$ by
\begin{equation}
  c_1(ab) \defeq
  \begin{cases}
    av_{(a, b)} + v_{(a, b)}b & \text{if } (a, b) \in E_{new}, \\
    ab &\text{otherwise}.
  \end{cases}
\end{equation}
The homotopy is given on the basis of $C_1(G\cup G^s)$ by
\begin{equation}
  P(ab) \defeq
  \begin{cases}
    w_{(a, b)} & \text{if } (a, b) \in E_{new}, \\
    0 &\text{otherwise}.
  \end{cases}
\end{equation}
where $w_{(a, b)}\in U_{new}$.
A standard check of the two cases verifies that $\inducedch{c}$ is a chain map and the relations $c_1 \circ q_1^0 = \id$ and $q_1^0 \circ c_1 - \id = \zb{\bd}_2 \circ P$ hold.

Intuitively, $\inducedch{c}$ collapses a new edge $e=(a, b)\in E_{new}$ onto the sum of edges $av_e + v_eb$ in $G \cup G^r$.
The homotopy $P$ shows the two elements are homologous thanks to the presence of the directed triangle $w_e\in U_{new}$.
\end{poc}

By Proposition~\ref{prop:biv2_gens}, there exists $u_1 , \dots, u_n \in C_2( G^{s} )$ such that
\begin{equation}
  C_2 (G^{s}) = C_2(G^{r}) \oplus W_0 \oplus \left\langle u_1 , \dots, u_n\right\rangle
\end{equation}
where each $u_i$ is amongst the generators identified in Proposition~\ref{prop:biv2_gens}.
Define $W_i \defeq W_0 \oplus \left\langle u_1 , \dots, u_i \right\rangle$.
This gives a sequence of inclusion chain maps
\begin{figure}[H]
  \centering
  \begin{tikzcd}
C_2(G^{r}) \arrow[r] \arrow[d, "q^i_2"] \oplus W_{i-1} 
& C_1(G \cup G^{s}) \arrow[r] \arrow[d, "q^i_1"]
& C_0(G \cup G^{s}) \arrow[d, "q^i_0"]
& Q_{i-1} \arrow[d, two heads, "q^i"] \\
C_2(G^{r}) \oplus W_i \arrow[r] 
& C_1(G \cup G^{s}) \arrow[r] 
& C_0(G \cup G^{s}) 
& Q_i
  \end{tikzcd}
\end{figure}\noindent
where each rows is a subcomplex of $\zb{C}_\bullet(G, s)$.
We denote the degree $1$ homology groups by $Q_i$, with $Q_n \defeq \zbhom{G}{S}$, and the induced homology maps by $q^i : Q_{i-1} \to Q_i$.
Together these chain maps decompose $\zbinclinduc{ch}{r}{s}$ and hence the $q^i$ decompose $\zbinclinduc{hom}{r}{s}$, as show in the following diagram.
\begin{figure}[H]
  \centering
  \begin{tikzcd}
C_2(G^{r}) \arrow[r] \arrow[d]
& C_1(G \cup G^{r}) \arrow[r] \arrow[d]
& C_0(G \cup G^{r}) \arrow[d]
& 
& \zbhom{G}{{r}} \arrow[ddddd, two heads, "\zbinclinduc{hom}{r}{s}"] \arrow[ld, "q^0", sloped,"\cong"'] \\
C_2(G^{r}) \oplus W_0 \arrow[r] \arrow[d]
& C_1(G \cup G^{s}) \arrow[r] \arrow[d]
& C_0(G \cup G^{s}) \arrow[d]
& Q_0 \arrow[d, two heads, "q^1"'] \arrow[ddddr, dashed, two heads, sloped, "q^n \circ \dots \circ q^1"]
& \\
C_2(G^{r}) \oplus W_1 \arrow[r] \arrow[d]
& C_1(G \cup G^{s}) \arrow[r] \arrow[d]
& C_0(G \cup G^{s}) \arrow[d]
& Q_1 \arrow[d, two heads, "q^2"']
& \\
\vdots \arrow[d]
& \vdots \arrow[d]
& \vdots \arrow[d]
& \vdots \arrow[d, two heads, "q^{n-1}"']
& \\
C_2(G^{n}) \oplus W_{n-1} \arrow[r] \arrow[d]
& C_1(G \cup G^{s}) \arrow[r] \arrow[d]
& C_0(G \cup G^{s}) \arrow[d]
& Q_{n-1} \arrow[rd, two heads,sloped, "q^n"']
& \\
C_2(G^{s}) \arrow[r]
& C_1(G \cup G^{s}) \arrow[r]
& C_0(G \cup G^{s})
&
& \zbhom{G}{{s}}
  \end{tikzcd}
\end{figure}\noindent

Note that $\dim Q_i$ drops by at most $1$ at each step, since $\dim W_{i+1} = \dim W_i + 1$.
Let $J$ denote the subset of indices where the dimension drops, i.e.
\begin{equation}
  J \defeq \left\{ i \in \N \rmv \dim Q_i = \dim Q_{i-1} - 1 \right\}.
\end{equation}
Then for each $i\in J$, the map on homology $q^i$ has nullity $1$ and a basis for $\ker q^i$ is $\{[ b_i ]\}$ where $b_i \defeq \zb{\bd}_2 u_i \in C_1(G \cup G^{s})$.

Note that $[b_i] \in Q_0$ and $(q^{i-1} \circ \dots \circ q^1)[b_i] = [b_i]$ in $Q_{i-1}$.
Since each $[b_i]$ for $i \in J$ dies in a different $Q_i$, they must be linearly independent in $Q_0$.
Moreover, the nullity of $q^n \circ \dots \circ q^1$ is $\card{J}$
so $\left\{ [b_i] \rmv i \in J \right\}$ gives a basis for $\ker(q^n \circ \dots \circ q^1)$.
Hence $\left\{ (q^0)^{-1}[b_i] \rmv i \in J \right\}$ forms a basis for $\ker\zbinclinduc{hom}{r}{s}$.
It now remains to prove that each $(q^0)^{-1}[b_i]$ has a undirected circuit representative.

\begin{claim}
  For each $i\in J$, there exists an undirected circuit $p_i$ in $G$ such that $\zbinclinduc{hom}{0}{r}[\repres{p_i}] = (q^0)^{-1}[b_i]$.
\end{claim}
\begin{poc}
First we recall that $b_i = \zb{\bd}_2 u_i$ where $u_i$ is either a double edge, a directed triangle or a long square in $G^s$.
Some of the boundary edges may be edges in $G^r$ but at least one boundary edge is new in $G^s$.
By the previous claim, a representative for $(q^0)^{-1}[b_i]$ is $c_1(b_i)$.

We can write $b_i = \repres{p}$ where $p$ is the undirected circuit which traces the outline of the generator $u_i$.
Now $c_1$ maps edges in $G^r$ to themselves and edges not in $G^r$ to a sum of two edges in $G^r$ with the same boundary.
Therefore, no matter which type of generator $u_i$ is, we can write
$c_1(b_i) = \repres{\widetilde{p}}$
for some undirected circuit $\widetilde{p}$ through $G^r$.
Using Lemma~\ref{lem:paths_homologous}, for each edge $\tau \in E(\widetilde{p})$ in this circuit there is a directed path $T_\tau$ through $G$ of length at most $r$ such that $\tau = \repres{T_\tau} \pmod{\zbbdrs{G}{r}}$.
Concatenating the $T_\tau$ we obtain an undirected circuit $p_{i}$ through $G$ such that $\repres{p} = \repres{p_i}\pmod{\zbbdrs{G}{r}}$.
\end{poc}

This claim concludes the proof.
\end{proof}

\begin{rem}
Note that the undirected circuits $p_{i, j}$ may not be simple.
We conjecture that it should be possible to choose every $p_{i, j}$ to be simple but do not, as yet, have a proof.
\end{rem}

\begin{figure}[hptb]
  \centering
  \begin{tikzpicture}[
  roundnode/.style={circle, fill=black, minimum size=4pt},
	squarenode/.style={fill=black, minimum size=4pt},
	inner sep=2pt,
	outer sep=1pt
  ]

  \node (a) at (-1, 0) [roundnode, label=left:$a$] {};
  \node (b) at (1, 0) [roundnode, label=right:$b$] {};
  \node (v1) at (-1, 1) [roundnode, label=above:$v_1$] {};
  \node (v2) at (0, 1) [roundnode, label=above:$v_2$] {};
  \node (v3) at (1, 1) [roundnode, label=above:$v_3$] {};
  \node (v4) at (-1, -1) [roundnode, label=below:$v_4$] {};
  \node (v5) at (0, -1) [roundnode, label=below:$v_5$] {};
  \node (v6) at (1, -1) [roundnode, label=below:$v_6$] {};
  \draw[->] (a) -- (v1) node[midway, above, sloped] {\tiny $1$};
  \draw[->] (a) -- (v4) node[midway, below, sloped] {\tiny $1$};
  \draw[->, red] (v1) -- (v2) node[midway, above, sloped] {\tiny $2$};
  \draw[->, red] (v2) -- (v3) node[midway, above, sloped] {\tiny $2$};
  \draw[->, red] (v4) -- (v5) node[midway, below, sloped] {\tiny $2$};
  \draw[->, red] (v5) -- (v6) node[midway, below, sloped] {\tiny $2$};
  \draw[->] (v3) -- (b) node[midway, above, sloped] {\tiny $1$};
  \draw[->] (v6) -- (b) node[midway, below, sloped] {\tiny $1$};

  \fill[blue!50, rounded corners] (3.1, 0.2) -- (3.1, 0.9) -- (3.8, 0.9) -- cycle;
  \fill[blue!50, rounded corners] (4.9, 0.2) -- (4.9, 0.9) -- (4.2, 0.9) -- cycle;
  \fill[blue!50, rounded corners] (3.1, -0.2) -- (3.1, -0.9) -- (3.8, -0.9) -- cycle;
  \fill[blue!50, rounded corners] (4.9, -0.2) -- (4.9, -0.9) -- (4.2, -0.9) -- cycle;
  \fill[green!50, rounded corners] (3.1, 0) -- (4, 0.9) -- (4.9, 0) -- (4, -0.9) -- cycle;

  \node (a_2) at (3, 0) [roundnode, label=left:$a$] {};
  \node (b_2) at (5, 0) [roundnode, label=right:$b$] {};
  \node (v1_2) at (3, 1) [roundnode, label=above:$v_1$] {};
  \node (v2_2) at (4, 1) [roundnode, label=above:$v_2$] {};
  \node (v3_2) at (5, 1) [roundnode, label=above:$v_3$] {};
  \node (v4_2) at (3, -1) [roundnode, label=below:$v_4$] {};
  \node (v5_2) at (4, -1) [roundnode, label=below:$v_5$] {};
  \node (v6_2) at (5, -1) [roundnode, label=below:$v_6$] {};
  \draw[->] (a_2) -- (v1_2);
  \draw[->] (a_2) -- (v4_2);
  \draw[->] (v1_2) -- (v2_2);
  \draw[->] (v2_2) -- (v3_2);
  \draw[->] (v4_2) -- (v5_2);
  \draw[->] (v5_2) -- (v6_2);
  \draw[->] (v3_2) -- (b_2);
  \draw[->] (v6_2) -- (b_2);
  \draw[->, blue] (a_2) -- (v2_2);
  \draw[->, blue] (v2_2) -- (b_2);
  \draw[->, blue] (a_2) -- (v5_2);
  \draw[->, blue] (v5_2) -- (b_2);

  \node[] at (0, -2) {$G$};
  \node[] at (4, -2) {$G^3$};

\end{tikzpicture}
  \caption{
    An example weighted digraph with a single feature which dies at $t=3$.
    To the right we show $G^3$, colouring the new edges in $G^3$ in blue and highlighting the new generators in $C_2(G^3)$ via blue and green polygons.
  }\label{fig:adding_generators}
\end{figure}
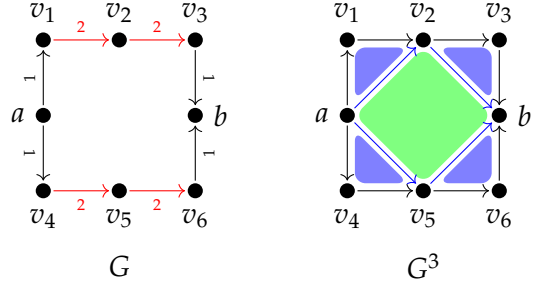

\begin{example}
To illustrate how the generators are added in the proof of Lemma~\ref{lem:submap_circuit_basis}, consider Figure~\ref{fig:adding_generators}.
First note that there is a single feature with representative
\begin{equation}
  \gamma \defeq (av_1 + v_1 v_2 + v_2 v_3 + v_3 b) - (a v_4 + v_4 v_5 + v_5 v_6 + v_6 b)
\end{equation}
which dies at $t=3$.
Further note that new edges appear at integer values in the shortest-path filtration and $\thedesc{G} = \ms{ [0, 3) }$.

The new edges which appear at $t=3$ are highlighted in blue.
The new generators of $C_2(G^3)$ are the four blue directed triangles and the central green long square.
The blue directed triangles form the elements of $U_{new}$ and the central long square is the sole remaining generator $u_1$.
So a basis for $\ker q^1$ is $\{[\zb{\bd}_2 u_1]\}$.

To find a basis for $\ker\zbinclinduc{hom}{2}{3}$ we must compute $c_1(\zb{\bd}_2 u_1)$.
Firstly $\zb{\bd}_2 u_1 = a v_2 + v_2 b - v_5 b - a v_5$.
Then the chain map $c_1$ maps each of the blue edges to a sum of edges in $G^2$.
In fact
$c_1(\zb{\bd}_2 u_1) = \gamma$ which is the representative of the sole simple undirected circuit in $G$.
\end{example}

The following Lemma follows by a standard linear algebra argument, since each $\zbinclinduc{hom}{t_{i-1}}{t_i}$ is surjective.

\begin{lemma}\label{lem:collection_of_bases}
Given $B_i\subseteq\zbhom{G}{0}$ such that $\zbinclinduc{hom}{0}{t_i}(B_i)$ is a basis for $\ker\zbinclinduc{hom}{t_i}{t_{i+1}}$, the union $\cup_{i=1}^{m-1} B_i$ is a basis for $\ker\zbinclinduc{hom}{0}{t_m}$.
\end{lemma}

Certainly a basis of undirected circuits for $\zbhom{G}{0}$ exists (by Corollary~\ref{cor:decreasing_curves}).
Therefore, we can always extend linearly independent undirected circuits to a basis of such circuits for $\zbhom{G}{0}$. 

\begin{lemma}\label{lem:basis_extension}
Given undirected circuits $p_1, \dots, p_k$ such that $[\repres{p_1}], \dots, [\repres{p_k}]$ are linearly independent in $\zbhom{G}{0}$,
there exists undirected circuits $p_{k+1}, \dots, p_N$ such that $\{[\repres{p_i}]\}_{i=1}^N$ is a basis for $\zbhom{G}{0}$.
\end{lemma}

We now have all the ingredients we need to prove the main theorem.

\begin{proof}[Proof of Theorem~\ref{thm:rep_of_cycles}]
Using Lemma~\ref{lem:submap_circuit_basis} and Lemma~\ref{lem:collection_of_bases}, we obtain undirected circuits $p_1, \dots, p_l$ such that $\left\{ \repres{p_1}, \dots, \repres{p_l} \right\}$ forms a basis for $\ker\zbinclinduc{hom}{0}{t_m}$.
Using Lemma~\ref{lem:basis_extension}, we extend this to a basis $\left\{ \repres{p_1}, \dots, \repres{p_l} , \repres{p_{l+1}}, \dots, \repres{p_N} \right\}$ for $\zbhom{G}{0}$.
This is a persistence basis for $\zbpershom{G}$.
\end{proof}

\begin{rem}\label{rem:pb_not_unique}
  While we are guaranteed a basis of undirected circuits, this choice of basis is by no means unique.
  As a simple example, consider again Example~\ref{ex:not_any_reps} and Figure~\ref{fig:representatives_example}.
  Two possible persistence bases for $\zbpershom{G_2}$ are 
  $\left\{ \gamma^2_1, \gamma^2_3 \right\}$ and 
  $\left\{ \gamma^2_2, \gamma^2_3 \right\}$.
  The non-uniqueness of the basis arises in the proof of Lemma~\ref{lem:submap_circuit_basis}.
  Namely, there is a choice of $v_e$ and $w_e$ for each $e\in E_{new}$, and a choice of order on the remaining $u_i$.
\end{rem}

\subsection{Decomposition}\label{sec:decomposition}

\begin{figure}[hptb]
  \centering
  \begin{tikzpicture}[
  roundnode/.style={circle, fill=black, minimum size=4pt},
	squarenode/.style={fill=black, minimum size=4pt},
	inner sep=2pt,
	outer sep=1pt
  ]

  \node (a) at (-3, 0) [roundnode] {};
  \node (b) at (-3, 1) [roundnode] {};
  \node (c) at (-3, -1) [roundnode] {};
  \node (d) at (-2, 0) [roundnode] {};
  \node (e) at (-1.5, 0) [roundnode] {};
  \node (f) at (-0.5, 1) [roundnode] {};
  \node (g) at (-0.5, -1) [roundnode] {};

  \draw[->] (a) -- (b);
  \draw[->] (a) -- (c);
  \draw[->] (a) -- (d);
  \draw[->] (b) -- (d);
  \draw[->] (c) -- (d);
  
  \draw[->] (e) -- (f);
  \draw[->] (f) -- (g);
  \draw[->] (e) -- (g);

  \node (a2) at (1.5, 0) [roundnode] {};
  \node (b2) at (1.5, 1) [roundnode] {};
  \node (c2) at (1.5, -1) [roundnode] {};
  \node (d2) at (2.5, 0) [roundnode, label=above:$\hat{v}$] {};
  \node (f2) at (3.5, 1) [roundnode] {};
  \node (g2) at (3.5, -1) [roundnode] {};

  \draw[->] (a2) -- (b2);
  \draw[->] (a2) -- (c2);
  \draw[->] (a2) -- (d2);
  \draw[->] (b2) -- (d2);
  \draw[->] (c2) -- (d2);
  
  \draw[->] (d2) -- (f2);
  \draw[->] (f2) -- (g2);
  \draw[->] (d2) -- (g2);

  \node[] at (-1.75, -2) {$G_1 \sqcup G_2$};
  \node[] at (2.5, -2) {$G_1 \vee_{\hat{v}} G_2$};

\end{tikzpicture}
  \caption{
    Illustrations of the disjoint union and wedge decomposition considered in Section~\ref{sec:decomposition}.
  }\label{fig:decompositions}
\end{figure}
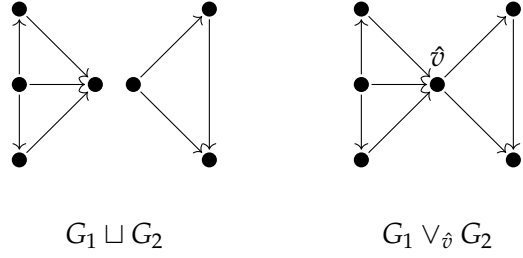

In order to more easily compute \grpph, it is desirable to understand how decompositions of the input weighted digraphs give rise to decompositions of the descriptor.
The simplest such decomposition is a disjoint union; as one might expect, the descriptor decomposes as a direct sum.

\begin{theorem}\label{thm:union_decomp}
Suppose $G\in\WDgr$ decomposes as a disjoint union, $G=G_1\sqcup G_2$, then
\begin{equation}
 \zbpershom{G} \cong \zbpershom{G_1} \oplus \zbpershom{G_2}.
\end{equation}
\end{theorem}
\begin{proof}
Note that for each $t\geq 0$, 
\begin{equation}
G^t = G_1^t \sqcup G_2^t
\quad\text{ and }\quad
G \cup G^t = (G_1 \cup G_1^t) \sqcup (G_2 \cup G_2^t).
\end{equation}
For each degree $k\geq 0$, if $H = H_1 \sqcup H_2$ then $C_k(H) = C_k(H_1)\oplus C_k(H_2)$.
Therefore, for each $k\geq 0$, $\zb{C}_k(G, t)$ splits as direct sum $\zb{C}_k(G_1, t)\oplus\zb{C}_k(G_2, t)$.
The boundary operator respects this split, mapping $\zb{C}_k(G_i, t) \to \zb{C}_{k-1}(G_i, t)$ and the maps $\zbinclinduc{ch}{s}{t}$ also respect this split, mapping $\zb{C}(G_i, s) \to \zb{C}(G_i, t)$.
Taking homology in degree $1$ maintains this direct sum decomposition.
\end{proof}

\begin{defin}
  \begin{enumerate}[label=(\alph*)]
    \item Given a weighted digraph $G=(V, E, w)$, a \mdf{wedge vertex} is a vertex $\hat{v}\in V$ such that there is a decomposition
\begin{equation}
 V = V_1 \cup V_2
\end{equation}
with $V_1 \cap V_2 = \{ \hat{v} \}$ such that $E \subseteq (V_1 \times V_1) \cup (V_2 \times V_2)$.
    \item Given a wedge vertex, $\hat{v}$ as above the corresponding \mdf{wedge decomposition} of $G$ is the pair $(G_1, G_2)$ where $G_1$ and $G_2$ are the induced subgraphs on $V_1$ and $V_2$ respectively.
      We write $\mdf{G=G_1\vee_{\hat{v}} G_2}$.
    \item Given a wedge decomposition as above, a pair of vertices $a, b \in V$ are called \mdf{separated} if they do not lie in a common $V_i$.
  \end{enumerate}
\end{defin}
\begin{rem}
Given a wedge decomposition $G = G_1\vee_{\hat{v}} G_2$ note that $G = G_1 \cup G_2$.
\end{rem}

In the case of a wedge decomposition $G=G_1 \vee_{\hat{v}} G_2$, since each simple circuit is contained either entirely in $G_1$ or entirely in $G_2$, one expects that \grpph\ also decomposes.
The proof is more complicated because, in general, $G^t \neq G_1^t \vee_{\hat{v}} G_2^t$, since there may be paths between separated vertices, through $\hat{v}$.
However, using a chain homotopy, we can show that these edges do not affect the homology.

\begin{theorem}\label{thm:wedge_decomp}
For a weighted digraph $G=(V, E, w)$ and a wedge decomposition $G=G_1 \vee_{\hat{v}} G_2$,
\begin{equation}
 \zbpershom{G} \cong \zbpershom{G_1} \oplus \zbpershom{G_2}.
\end{equation}
\end{theorem}
\begin{proof}
There are natural inclusion digraph maps $j_i : G_i \to G$ which are also contractions.
Less obviously, there are contraction digraph maps $f_i: G \to G_i$, where
\begin{equation}
 f_i(v) \defeq
 \begin{cases}
   v &\text{if }v\in V_i, \\
   \hat{v} &\text{otherwise}.
 \end{cases}
\end{equation}
Since these are all morphisms in $\Cont\WDgr$, we obtain induced morphisms $\zbdiginduc{ch}{j_i}$ and $\zbdiginduc{ch}{f_i}$.
We combine these morphisms to get two morphisms as follows
\begin{align}
 J : \zb{C}(G_1) \oplus \zb{C}(G_2) \to \zb{C}(G) , \quad
 & J(\gamma_1, \gamma_2) \defeq \zbdiginduc{ch}{j_1}\gamma_1 + \zbdiginduc{ch}{j_2}\gamma_2 ;
 \\
 F : \zb{C}(G) \to \zb{C}(G_1) \oplus \zb{C}(G_2) , \quad 
 & F(\gamma) \defeq (\zbdiginduc{ch}{f_1}\gamma , \zbdiginduc{ch}{f_2} \gamma).
\end{align}
Composing with homology in degree $1$, denote
$\inducedhom{J} \defeq \Funct{\Rposet}{H_1} \circ J$
and
$\inducedhom{F} \defeq \Funct{\Rposet}{H_1} \circ F$.
In the rest of the proof, we show that $\inducedhom{J}$ and $\inducedhom{F}$ are mutually inverse.

\begin{claim}
In degree $1$, $F\circ J = \id$ is the identity map on $\zb{C}(G_1)\oplus \zb{C}(G_2)$.
\end{claim}
\begin{poc}
First note that, $f_i \circ j_i$ is the identity digraph map $\id_i : G_i \to G_i$.
However, $f_{3-i}\circ j_i$ is the constant digraph map $c_{i}: G_i \to G_{3-i}$ which maps all of $G_i$ to the vertex $\hat{v}$.
Hence, in matrix form, we can write $F\circ J$ as
\begin{equation}
  F\circ J = 
  \begin{pmatrix}
    \zbdiginduc{ch}{\id_1} & \zbdiginduc{ch}{c_2} \\
    \zbdiginduc{ch}{c_1} & \zbdiginduc{ch}{\id_2}
  \end{pmatrix}.
\end{equation}
Since the constant maps $c_i$ send all vertices to a single vertex, $\zbdiginduc{ch}{c_i}$ maps every edge to $0$.
Hence, in degree $1$, $\zbdiginduc{ch}{c_i}$ is the zero map.
Whereas, in degree $1$, $\zbdiginduc{ch}{\id_i}$ is the identity map on $\zb{C}_1(G_i, t)$ at each $t$.
Therefore, on degree $1$, $F\circ J$ is the identity map on $\zb{C}(G_1) \oplus \zb{C}(G_2)$.
\end{poc}
Composing with homology in degree $1$, we see $\inducedhom{F}\circ\inducedhom{J}=\id$ is the identity map on $\zbpershom{G_1} \oplus \zbpershom{G_2}$.

\begin{claim}
In degree $1$, $\inducedhom{J} \circ \inducedhom{F} = \id$ is the identity map on $\zbpershom{G}$. 
\end{claim}
\begin{poc}
First, we compute $J\circ F$ in degree $1$.
Recall that $\zb{C}_1(G, t)$ is freely generated by the edges in $G \cup G^t$.
Given an edge $e=(a, b)\in E(G \cup G^t)$, if the vertices $a, b$ lie in a common $V_i$ then $(J \circ F)(e) = e$.
However, if $a, b$ are separated then $(J\circ F)(e) = a\hat{v} + \hat{v}b$.
So we see, at the level of chains, $J\circ F$ does not compose to the identity.

If $e=(a, b)\in E(G\cup G^t)$ but the endpoints are separated then we must have $e\in E(G^t)$.
Hence, there is a path $p: a\leadsto b$ in $G$ of length at most $t$.
Moreover, this path must traverse the vertex $\hat{v}$.
Hence, $p$ decomposes into two paths $a\leadsto \hat{v}$ and $\hat{v} \leadsto b$, each of length at most $t$.
Therefore, the directed triangle $a\hat{v} b$ is present in $G^t$ and is a generator of $\zb{C}_2(G, t)$.
Note that the boundary of $a\hat{v}b$ is
\begin{equation}
\bd_2(a\hat{v} b) = a\hat{v} + \hat{v} b - (ab) = (J\circ F)(e) - e.
\end{equation}

This discussion show that we can define a map $P: \zb{C}_1(G, t) \to \zb{C}_2(G, t)$ by
\begin{equation}
P(ab) \defeq 
\begin{cases}
  a\hat{v} b & \text{if }a, b\text{ are separated}, \\
  0 &\text{otherwise}.
\end{cases}
\end{equation}
Then, $J\circ F - \id = \bd_2 P$ as maps on $\zb{C}_1(G)$.
Composing with homology, we see $\inducedhom{J}\circ\inducedhom{F} = \id$ is the identity on $\zbpershom{G}$.
\end{poc}
Since $\inducedhom{J}$ and $\inducedhom{F}$ are mutually inverse, they induce isomorphisms of persistent vector spaces.
\end{proof}

\section{Stability analysis of GrPPH}\label{sec:stability}

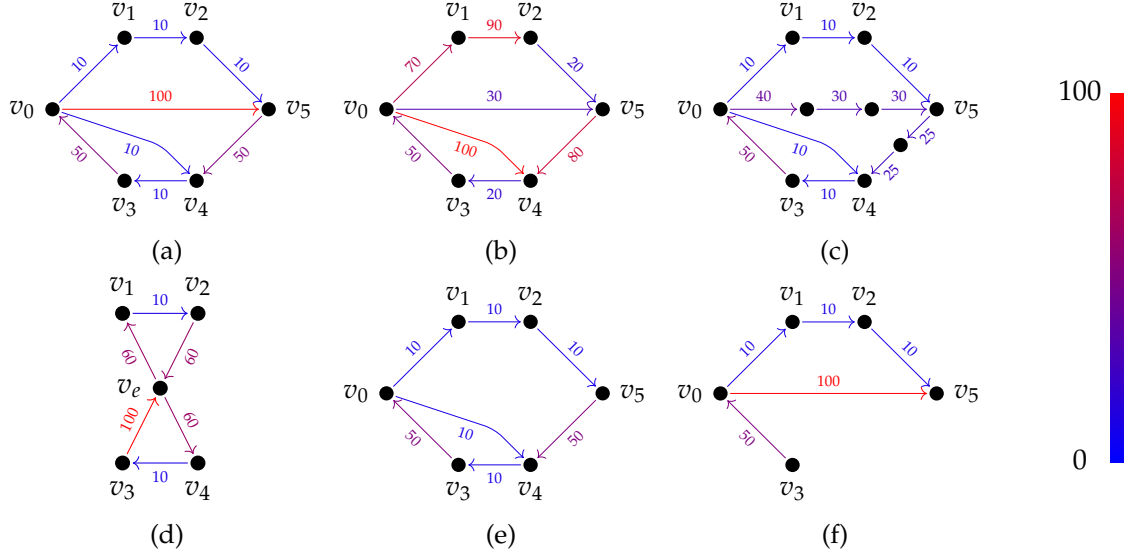
\begin{figure}[htbp]
  \begin{minipage}{.9\linewidth}
    \centering
    \begin{subfigure}[t]{0.3\textwidth}
      \centering
      \resizebox{\textwidth}{!}{%
        \begin{tikzpicture}[
  roundnode/.style={circle, fill=black, minimum size=4pt},
	squarenode/.style={fill=black, minimum size=4pt},
	inner sep=2pt,
	outer sep=1pt
  ]

  \node[] at (-0.5, 0) {};
  \node[] at (3.5, 0) {};

  \node (a) at (0, 0) [roundnode, label=left:$v_0$] {};
	\node (b) at (1, 1) [roundnode, label=above:$v_1$] {};
	\node (c) at (2, 1) [roundnode, label=above:$v_2$] {};
	\node (d) at (1, -1) [roundnode, label=below:$v_3$] {};
	\node (e) at (2, -1) [roundnode, label=below:$v_4$] {};
	\node (f) at (3, 0) [roundnode, label=right:$v_5$] {};

  \draw[red!10!blue, ->] (a)--(b) node[midway, above, sloped] {\tiny $10$};
  \draw[red!10!blue, ->] (b)--(c) node[midway, above, sloped] {\tiny $10$};
  \draw[red!10!blue, ->] (c)--(f) node[midway, above, sloped] {\tiny $10$};
  \draw[red!100!blue, ->] (a)--(f) node[midway, above, sloped] {\tiny $100$};
  \draw[red!50!blue, ->] (f)--(e) node[midway, below, sloped] {\tiny $50$};
  \draw[red!10!blue, ->] (e)--(d) node[midway, below, sloped] {\tiny $10$};
  \draw[red!50!blue, ->] (d)--(a) node[midway, below, sloped] {\tiny $50$};
  \draw[red!10!blue, ->, rounded corners] (a)--(1.5, -0.5)node[near end, below, sloped] {\tiny $10$}  -- (e) ;


\end{tikzpicture}
      }
      \caption{}
    \end{subfigure}
    \begin{subfigure}[t]{0.3\textwidth}
      \centering
      \resizebox{\textwidth}{!}{%
        \begin{tikzpicture}[
  roundnode/.style={circle, fill=black, minimum size=4pt},
	squarenode/.style={fill=black, minimum size=4pt},
	inner sep=2pt,
	outer sep=1pt
  ]
  \node[] at (-0.5, 0) {};
  \node[] at (3.5, 0) {};

  \node (a) at (0, 0) [roundnode, label=left:$v_0$] {};
	\node (b) at (1, 1) [roundnode, label=above:$v_1$] {};
	\node (c) at (2, 1) [roundnode, label=above:$v_2$] {};
	\node (d) at (1, -1) [roundnode, label=below:$v_3$] {};
	\node (e) at (2, -1) [roundnode, label=below:$v_4$] {};
	\node (f) at (3, 0) [roundnode, label=right:$v_5$] {};

  \draw[red!70!blue, ->] (a)--(b) node[midway, above, sloped] {\tiny $70$};
  \draw[red!90!blue, ->] (b)--(c) node[midway, above, sloped] {\tiny $90$};
  \draw[red!20!blue, ->] (c)--(f) node[midway, above, sloped] {\tiny $20$};
  \draw[red!30!blue, ->] (a)--(f) node[midway, above, sloped] {\tiny $30$};
  \draw[red!80!blue, ->] (f)--(e) node[midway, below, sloped] {\tiny $80$};
  \draw[red!20!blue, ->] (e)--(d) node[midway, below, sloped] {\tiny $20$};
  \draw[red!50!blue, ->] (d)--(a) node[midway, below, sloped] {\tiny $50$};
  \draw[red!100!blue, ->, rounded corners] (a)--(1.5, -0.5)node[near end, below, sloped] {\tiny $100$}  -- (e) ;
\end{tikzpicture}
      }
      \caption{}
    \end{subfigure}
    \begin{subfigure}[t]{0.3\textwidth}
      \centering
      \resizebox{\textwidth}{!}{%
        \begin{tikzpicture}[
  roundnode/.style={circle, fill=black, minimum size=4pt},
	squarenode/.style={fill=black, minimum size=4pt},
	inner sep=2pt,
	outer sep=1pt
  ]
  \node[] at (-0.5, 0) {};
  \node[] at (3.5, 0) {};

  \node (a) at (0, 0) [roundnode, label=left:$v_0$] {};
	\node (b) at (1, 1) [roundnode, label=above:$v_1$] {};
	\node (c) at (2, 1) [roundnode, label=above:$v_2$] {};
	\node (d) at (1, -1) [roundnode, label=below:$v_3$] {};
	\node (e) at (2, -1) [roundnode, label=below:$v_4$] {};
  \node (mid1) at (1.2, 0) [roundnode] {};
  \node (mid2) at (2.1, 0) [roundnode] {};
  \node (mid3) at (2.5, -0.5) [roundnode] {};
	\node (f) at (3, 0) [roundnode, label=right:$v_5$] {};

  \draw[red!10!blue, ->] (a)--(b) node[midway, above, sloped] {\tiny $10$};
  \draw[red!10!blue, ->] (b)--(c) node[midway, above, sloped] {\tiny $10$};
  \draw[red!10!blue, ->] (c)--(f) node[midway, above, sloped] {\tiny $10$};
  \draw[red!40!blue, ->] (a)--(mid1) node[midway, above, sloped] {\tiny $40$};
  \draw[red!30!blue, ->] (mid1)--(mid2) node[midway, above, sloped] {\tiny $30$};
  \draw[red!30!blue, ->] (mid2)--(f) node[midway, above, sloped, xshift=-2pt] {\tiny $30$};
  \draw[red!25!blue, ->] (f)--(mid3) node[midway, below, sloped] {\tiny $25$};
  \draw[red!25!blue, ->] (mid3)--(e) node[midway, below, sloped] {\tiny $25$};
  \draw[red!10!blue, ->] (e)--(d) node[midway, below, sloped] {\tiny $10$};
  \draw[red!50!blue, ->] (d)--(a) node[midway, below, sloped] {\tiny $50$};
  \draw[red!10!blue, ->, rounded corners] (a)--(1.5, -0.5) node[near end, below, sloped] {\tiny $10$} -- (e) ;

\end{tikzpicture}
      }
      \caption{}
    \end{subfigure}
    \begin{subfigure}[t]{0.3\textwidth}
      \centering
      \resizebox{\textwidth}{!}{%
        \begin{tikzpicture}[
  roundnode/.style={circle, fill=black, minimum size=4pt},
	squarenode/.style={fill=black, minimum size=4pt},
	inner sep=2pt,
	outer sep=1pt
  ]
  \node[] at (-0.5, 0) {};
  \node[] at (3.5, 0) {};

  \node (af) at (1.5, 0) [roundnode, label=left:$v_e$] {};
	\node (b) at (1, 1) [roundnode, label=above:$v_1$] {};
	\node (c) at (2, 1) [roundnode, label=above:$v_2$] {};
	\node (d) at (1, -1) [roundnode, label=below:$v_3$] {};
	\node (e) at (2, -1) [roundnode, label=below:$v_4$] {};

  \draw[red!60!blue, ->] (af)--(b) node[midway, below, sloped] {\tiny $60$};
  \draw[red!10!blue, ->] (b)--(c) node[midway, above, sloped] {\tiny $10$};
  \draw[red!60!blue, ->] (c)--(af) node[midway, below, sloped] {\tiny $60$};
  \draw[red!60!blue, ->] (af)--(e) node[midway, above, sloped] {\tiny $60$};
  \draw[red!10!blue, ->] (e)--(d) node[midway, below, sloped] {\tiny $10$};
  \draw[red!100!blue, ->] (d)--(af) node[midway, above, sloped] {\tiny $100$};


\end{tikzpicture}
      }
      \caption{}
    \end{subfigure}
    \begin{subfigure}[t]{0.3\textwidth}
      \centering
      \resizebox{\textwidth}{!}{%
        \begin{tikzpicture}[
  roundnode/.style={circle, fill=black, minimum size=4pt},
	squarenode/.style={fill=black, minimum size=4pt},
	inner sep=2pt,
	outer sep=1pt
  ]
  \node[] at (-0.5, 0) {};
  \node[] at (3.5, 0) {};

  \node (a) at (0, 0) [roundnode, label=left:$v_0$] {};
	\node (b) at (1, 1) [roundnode, label=above:$v_1$] {};
	\node (c) at (2, 1) [roundnode, label=above:$v_2$] {};
	\node (d) at (1, -1) [roundnode, label=below:$v_3$] {};
	\node (e) at (2, -1) [roundnode, label=below:$v_4$] {};
	\node (f) at (3, 0) [roundnode, label=right:$v_5$] {};

  \draw[red!10!blue, ->] (a)--(b) node[midway, above, sloped] {\tiny $10$};
  \draw[red!10!blue, ->] (b)--(c) node[midway, above, sloped] {\tiny $10$};
  \draw[red!10!blue, ->] (c)--(f) node[midway, above, sloped] {\tiny $10$};
  \draw[red!50!blue, ->] (f)--(e) node[midway, below, sloped] {\tiny $50$};
  \draw[red!10!blue, ->] (e)--(d) node[midway, below, sloped] {\tiny $10$};
  \draw[red!50!blue, ->] (d)--(a) node[midway, below, sloped] {\tiny $50$};
  \draw[red!10!blue, ->, rounded corners] (a)--(1.5, -0.5)node[near end, below, sloped] {\tiny $10$}  -- (e) ;
\end{tikzpicture}
      }
      \caption{}
    \end{subfigure}
    \begin{subfigure}[t]{0.3\textwidth}
      \centering
      \resizebox{\textwidth}{!}{%
        \begin{tikzpicture}[
  roundnode/.style={circle, fill=black, minimum size=4pt},
	squarenode/.style={fill=black, minimum size=4pt},
	inner sep=2pt,
	outer sep=1pt
  ]
  \node[] at (-0.5, 0) {};
  \node[] at (3.5, 0) {};

  \node (a) at (0, 0) [roundnode, label=left:$v_0$] {};
	\node (b) at (1, 1) [roundnode, label=above:$v_1$] {};
	\node (c) at (2, 1) [roundnode, label=above:$v_2$] {};
	\node (d) at (1, -1) [roundnode, label=below:$v_3$] {};
	\node (f) at (3, 0) [roundnode, label=right:$v_5$] {};

  \draw[red!10!blue, ->] (a)--(b) node[midway, above, sloped] {\tiny $10$};
  \draw[red!10!blue, ->] (b)--(c) node[midway, above, sloped] {\tiny $10$};
  \draw[red!10!blue, ->] (c)--(f) node[midway, above, sloped] {\tiny $10$};
  \draw[red!100!blue, ->] (a)--(f) node[midway, above, sloped] {\tiny $100$};
  \draw[red!50!blue, ->] (d)--(a) node[midway, below, sloped] {\tiny $50$};
\end{tikzpicture}
      }
      \caption{}
    \end{subfigure}
  \end{minipage}
  \begin{minipage}{.07\linewidth}
      \resizebox{\textwidth}{!}{%
        \begin{tikzpicture}[
  roundnode/.style={circle, fill=black, minimum size=4pt},
	squarenode/.style={fill=black, minimum size=4pt},
	inner sep=2pt,
	outer sep=1pt
  ]

  \node[shading=axis,
    rectangle,
    left color=red!100!blue, right color=red!0!blue, shading angle = 0,
    minimum width=5pt, minimum height=5cm, anchor=north] (my_gradient) at (-1, 2.5) {};
  \node[] at (-1.5, 2.5) {$100$};
  \node[] at (-1.5, -2.5) {$0$};

\end{tikzpicture}
      }
  \end{minipage}
	\caption{%
    Illustration of a number of operations for altering a weighted digraph.
    \textbf{(a)} The initial weighted digraph $G$.
    \textbf{(b)} Weight perturbation $\wdop{p}{w'}{G}$.
    \textbf{(c)} Edge subdivision $\wdop{s}{S}{G}$
    where the subdivision is $S:\{(v_0, v_5), (v_5, v_4)\}\to \Delta^2$ where $S((v_0, v_5))=(4/10, 3/10, 3/10)$ and $S((v_5, v_4))=(1/2, 1/2, 0)$.
    \textbf{(d)} Edge collapse $\wdop{c}{e}{G}$ where $e=(v_0, v_5)$.
    \textbf{(e)} Edge deletion $\wdop{d}{e}{G}$ where $e=(v_0, v_5)$.
    \textbf{(f)} Vertex deletion $\wdop{d}{v}{G}$ where $v=v_4$.
  }\label{fig:all_ops}
\end{figure}

It is important that \grpph\ is stable with respect to a reasonable noise model.
Typically this is shown by proving that $\zbpershommap$ is Lipschitz with respect to reasonable metrics on $\WDgr$ and the bottleneck distance on $\PersVec$.
A common choice of metric on graphs is the graph edit distance.
However, assigning costs to operations such as edge deletion or edge subdivision is somewhat arbitrary.

Therefore, in this section, we consider operations $\wdop{T}{\theta}{}:\Obj(\WDgr)\to\Obj(\WDgr)$ for editing weighted digraphs, where $T$ is the type of operation $\theta$ is the parameter of the operation.
For each type $T$, we derive bounds of the form
\begin{equation}
 d_B(\thedesc{G}, \thedesc{\wdop{T}{\theta}{G}}) \leq f(G, \theta)\label{eq:general_stab_bound}
\end{equation}
and then say \grpph\ is \mdf{stable} to operations of type $T$.

Often the operations only alter the graph at a subset of vertices or edges.
We say that \grpph\ is \mdf{locally stable} to the operation if we obtain a bound as in (\ref{eq:general_stab_bound}) and $f$ depends only on the neighbourhood graph around the altered vertices/edges and $\theta$.
If we can show that no such local $f$ exists (for general $G$ and $\theta$) then we say \grpph\ is \mdf{locally unstable}.
If, in general, $f$ depends on all of $G$ then we say \grpph\ is \mdf{non-locally stable} to the operation.
Occasionally, some operations do not change the descriptor and we can find an isomorphism $\zbpershom{G}\cong\zbpershom{\wdop{T}{\theta}{G}}$.

Figure~\ref{fig:all_ops} illustrates all of the operations we consider, the precise definitions of which of which are provided in the relevant subsection.
Table~\ref{tbl:stability_result_summary} summaries our findings.

\begin{table}[hbtp]
\begin{center}
\ars{1.2}
\tcs{0.7\tabcolsep}
\renewcommand\theadfont{\bfseries}
\newcommand*\partialthm{${}^\blacklozenge$}
\begin{tabular}{ l | c c c c }
  \thead{Operation} & \thead{Locally\\Stable} & \thead{Non-locally\\Stable} & \thead{Locally\\Unstable}  & \thead{Isomorphism} \\ \hline
  Weight perturbation & Theorem~\ref{thm:pertub_stability} & & & \\
  Edge subdivision & Theorem~\ref{thm:subdiv_stability} & & & \\
  Edge collapse & Theorem~\ref{thm:collapse_stability}\partialthm &
                & Theorem~\ref{thm:collapse_instability} & \\
  Edge deletion &
  Corollary~\ref{cor:local_edge_deletion}\partialthm &
  Theorem~\ref{thm:edge_deletion_stability} &
  Theorem~\ref{thm:edge_deletion_instability} &
    Theorem~\ref{cor:edge_deletion_disconnect}\partialthm
  \\
  Vertex deletion & & & Corollary~\ref{cor:vertex_deletion_instability} & Corollary~\ref{cor:isolated_vertex_deletion}\partialthm
\end{tabular}
\caption{
  Stability and instability theorems for $\zbpershommap$, under various digraph operations.
  $\blacklozenge$ Denotes a theorem which only applies to a subset of such operations.
}\label{tbl:stability_result_summary}
\end{center}
\end{table}

\subsection{Preliminaries}

In order to prove stability, we will have to build interleaving chain maps.
Often, these will be constructed via maps of the underlying vertex sets.

\begin{defin}
 For $\delta\geq 0$, a \mdf{$\delta$-shifting vertex map}, between two weighted digraphs $G$ and $H$, is a vertex map $f:V(G) \to V(H)$ such that $f$ induces digraph maps 
$G^t \to H^{t+\delta}$ and $G\cup G^t \to H \cup H^{t+\delta}$
 for all $t\geq 0$.
\end{defin}

\begin{rem}
By Lemma~\ref{lem:wdgrfd_is_contwdgr}, a $0$-shifting vertex map is precisely a contraction digraph map.
\end{rem}

\begin{lemma}\label{lem:shift_induce_comp}
Any $\delta$-shifting vertex map $f:V(G) \to V(H)$ induces a morphism
\begin{equation}
 \shiftinduced{ch}{f}{\delta}:\zb{C}(G) \to \shifted{\zb{C}(H)}{\delta}.
\end{equation}
Given another $\epsilon$-shifting vertex map $g:V(H) \to V(K)$,
\begin{equation}
 \shiftinduced{ch}{g\circ f}{\epsilon + \delta} = \shiftinduced{ch}{g}{\epsilon} \circ \shiftinduced{ch}{f}{\delta}.
\label{eq:shift_functoriality_comp}
\end{equation}
Moreover, if $f$ is $0$-shifting then $\shiftinduced{ch}{f}{0} = \zbdiginduc{ch}{f}$.
\end{lemma}
\begin{proof}
Fix $t\geq 0$.
Since $f$ is a $\delta$-shifting vertex map, the functor $C$ induces chain maps $\inducedch{f}: C_\bullet(G^t) \to C_\bullet(H^{t+\delta})$ and $\inducedch{f}: C_\bullet(G \cup G^t) \to C_\bullet(H \cup H^{t+\delta})$.
These chain maps fit together into the following commutative diagram.
\begin{figure}[H]
\centering
\begin{tikzcd}[row sep=tiny]
  C_3(G^t) \arrow[r] \arrow[ddd, "\inducedch{f}"] \arrow[rddd, phantom, description, "\color{blue}\boxed{A}"] &
  C_2(G^t) \arrow[rr] \arrow[rd] \arrow[ddd, "\inducedch{f}"] \arrow[rddd, phantom, description, "\color{blue}\boxed{B}"]
  & & C_1(G \cup G^t) \arrow[r] \arrow[ddd, "\inducedch{f}"] \arrow[lddd, phantom, description, "\color{blue}\boxed{C}"]
  \arrow[rddd, phantom, description, "\color{blue}\boxed{D}"]
  & C_0(G\cup G^t) \arrow[ddd, "\inducedch{f}"]  \\
  & & C_1(G^t) \arrow[hook, ru] \arrow[d, "\inducedch{f}"] & & \\
  & & C_1(H^{t+\delta}) \arrow[hook, rd] & & \\
  C_3(H^{t+\delta}) \arrow[r] &
  C_2(H^{t+\delta}) \arrow[ru] \arrow[rr] 
  & {\phantom{C_1(H^{t+\delta})}} & C_1(H \cup H^{t+\delta}) \arrow[r] 
  & C_0(H \cup H^{t+\delta})
\end{tikzcd}
\end{figure}\noindent
Squares $\color{blue}\boxed{A}$, $\color{blue}\boxed{B}$ and $\color{blue}\boxed{D}$ commute because they are parts of the same chain map.
The inclusion digraph maps $G^t\hookrightarrow G \cup G^t$ and $H^{t+\delta} \hookrightarrow H\cup H^{t+\delta}$ are induced by the identity vertex maps.
Therefore $f$ trivially commutes with these inclusions and hence square $\color{blue}\boxed{C}$ commutes.
Hence we get a chain map  which we denote $\shiftinduced{ch}{f}{\delta}:\zb{C}_\bullet(G, t)\to\zb{C}_\bullet(H, t+\delta)$.
Note, in the case $\delta=0$, this is precisely the same construction as $\zbdiginduc{ch}{f}$, given in the proof of Theorem~\ref{thm:grd_functoriality}.

It remains to show that these chain maps (for each $t\geq 0$) constitute a morphism $\zb{C}(G) \to \shifted{\zb{C}(H)}{\delta}$.
Given $s\leq t$, the chain map $\zbinclinduc{ch}{s}{t}:\zb{C}_\bullet(G, s) \to \zb{C}_\bullet(G, t)$ is induced by the identity vertex map.
The identity vertex map clearly commutes with $f$.
Hence, in each degree $\inducedch{f}$ commutes with the $\inducedch{\filtincl{s}{t}}$.
Therefore, $\shiftinduced{ch}{f}{\delta}$ commutes with $\zbinclinduc{ch}{s}{t}$.

Simple composition of digraph maps confirms that $g\circ f$ is a $\delta_g + \delta_f$-shifting vertex map.
Note that in each degree the chain map is constructed via the functor $C_\bullet$.
Therefore, in each degree, we have $\inducedch{(g\circ f)} = \inducedch{g}\circ\inducedch{f}$ which yields equation~(\ref{eq:shift_functoriality_comp}).
\end{proof}

Recall that $\transmorph{\zb{C}(G)}{\epsilon}$ at each $t\geq 0$ is the chain map $\zbinclinduc{ch}{t}{t+\epsilon}:\zb{C}(G, t) \to \zb{C}(G, t+\epsilon)$.
Shifting this by $\delta$, $\shifted{\transmorph{\zb{C}(G)}{\epsilon}}{\delta}$ is given at $t\geq 0$ by the chain map $\zbinclinduc{ch}{t+\delta}{t+\delta+\epsilon}$.

\begin{lemma}\label{lem:shift_induce_cont}
A $\delta$-shifting vertex map is a $\delta'$-shifting vertex map for any $\delta' \geq \delta$ and
\begin{equation}
 \shiftinduced{ch}{f}{\delta'} = \shifted{\transmorph{\zb{C}(H)}{\delta' - \delta}}{\delta} \circ \shiftinduced{ch}{f}{\delta}.
\end{equation}
\end{lemma}
\begin{proof}
Clearly a $\delta$-shifting vertex map is a $\delta'$-shifting vertex map since $G^\delta \subseteq G^{\delta '}$.
Since $\filtincl{t+\delta}{t+\delta'}$ is induced by the inclusion vertex map, we get the following commutative diagrams of digraph maps
\begin{figure}[H]
  \centering
  \begin{tikzcd}
    G^t \arrow[d, "f"'] \arrow[rd, "f"] &
                                       &
    G\cup G^t \arrow[d, "f"'] \arrow[rd, "f"] &
    \\
    H^{t+\delta} \arrow[r, "\filtincl{t+\delta}{t+\delta'}"'] &
    H^{t+\delta'} &
    H \cup H^{t+\delta} \arrow[r, "\filtincl{t+\delta}{t+\delta'}"'] &
    H \cup H^{t+\delta'}
  \end{tikzcd}
\end{figure}\noindent
Applying the functor $C$ to these diagrams yields two commutative diagrams in $\Ch$.
The diagonal maps are components of $\shiftinduced{ch}{f}{\delta'}$, the vertical maps are components of $\shiftinduced{ch}{f}{\delta}$ and the horizontal maps are components of $\zbinclinduc{ch}{t+\delta}{t+\delta'}$.
Hence we have a commutative diagram
\begin{figure}[H]
  \centering
  \begin{tikzcd}[column sep=huge]
    \zb{C}_\bullet(G, t) \arrow[d, "\shiftinduced{ch}{f}{\delta}"']
    \arrow[rd, "\shiftinduced{ch}{f}{\delta'}"] & \\
    \zb{C}_\bullet(H, t+\delta) \arrow[r, "\zbinclinduc{ch}{t+\delta}{t+\delta'}"'] &
    \zb{C}_\bullet(H, t+\delta')
  \end{tikzcd}
\end{figure}\noindent
as required.
\end{proof}

\subsection{Weight perturbation}

The classical stability theorem of persistent homology (first shown in~\cite{cohen2005stability}) is that for two continuous tame function $f, g: X \to \R$ of a trianguable topological space $X$, denoting the persistence barcode of their sub-level set filtration by $\Barcode{}_f$ and $\Barcode{}_g$ respectively,
\begin{equation}
 d_B(\Barcode{}_f, \Barcode{}_g) \leq \norm{f - g}_\infty.
\end{equation}
In our setting, the closet analogy to changing the function is changing the weighting, as well as the corresponding effect that has on the shortest-path quasimetric.
We find that \grpph\ is stable to perturbations of the edge weights.
Moreover, the stability is local since it depends only of the weights of the perturbed edges.

\begin{defin}
 Given a weighted digraph $G=(V, E, w)$ and a new weight function $w': E(G) \to \Rpos$, we define $\mdf{\wdop{p}{w'}{G}}\defeq (V, E, w')$.
\end{defin}

\begin{theorem}\label{thm:pertub_stability}
Given a weighted digraph $G=(V, E, w)\in\WDgr$ and a new weighting function $w':E(G) \to \Rpos$,
let $d$ and $d'$ denote the shortest-path quasimetric on $G$ and $\wdop{p}{w'}{G}$ respectively.
Then
\begin{equation}
 d_B\big(\thedesc{G}, \thedesc{\wdop{p}{w'}{G}}\big) \leq
 \max_{i, j \in V} \abs{d(i, j) - d'(i, j)} \leq \sum_{e \in E} \abs{w(e) - w'(e)}.
\end{equation}
\end{theorem}
\begin{proof}
 For brevity, denote $G'\defeq \wdop{p}{w'}{G}$ and $\delta\defeq \max_{i, j\in V}\abs{d(i, j) - d'(i, j)}$.
 First note that for any $(i, j)$ and any path $p\in\pathsfrom{i}{j}$
 \begin{equation}
     \abs{\sum_{e \in p}w(e) - \sum_{e\in p} w'(e)} \leq \sum_{e \in p}\abs{w(e) - w'(e)} \leq \sum_{e \in E}\abs{w(e) - w'(e)}\eqdef W_1.
 \end{equation}
 So the cost of $p$ differs by at most $W_1$.
 Minimising over $\pathsfrom{i}{j}$, we see $\abs{d(i, j) - d'(i, j)} \leq W_1$.

 Since $V(G) = V(G')$, there are identity vertex maps $i_1:V(G) \to V(G')$ and $i_2:V(G')\to V(G)$.
 Now $i_1$ defines a digraph map $G\to G'$ since $G=G'$ as digraphs.
 Moreover, given $(i, j)\in E(G^t)$, then $d(i, j) \leq t$ so  $d'(i, j) \leq t + \delta$ and hence $(i, j) \in E((G')^{t+\delta})$.
 This shows $i_1$ defines a digraph map $G^t \to (G')^{t+\delta}$ for all $t \geq 0$.
 Therefore $i_1$ (and likewise $i_2$) is a $\delta$-shifting vertex map.
 Hence, we obtain morphisms
 \begin{equation}
  \shiftinduced{ch}{i_1}{\delta} : \zb{C}(G) \to \shifted{\zb{C}(G')}{\delta}
  \quad\text{ and }\quad
  \shiftinduced{ch}{i_2}{\delta} : \zb{C}(G') \to \shifted{\zb{C}(G)}{\delta}.
 \end{equation}
 Moreover, since composing $i_1$ and $i_2$ in either ordered yields the identity vertex map, these morphisms constitute a $\delta$-interleaving.
 The first inequality then follows by the isometry theorem.
\end{proof}

\begin{rem}
Continuing the analogy to the classical stability theorem, note that the sharper bound obtained by Theorem~\ref{thm:pertub_stability} is $\norm{d-d'}_\infty$ while the weaker bound is $\norm{w-w'}_1$.
\end{rem}

\subsection{Edge subdivision}

Weighted digraphs arising in applications are subject not only to numerical noise (i.e. weight perturbation) but also \emph{structural noise}.
For the remainder of this section, we investigate the effects of various structural perturbations.

First, we consider edge subdivision, in which one or mare parent edge is split into multiple daughter edges with the weight distributed amongst them.
Since we are interpreting edge weights as corresponding to a length, it is natural to require that the sum of the weights of the daughter edges equals the weight of the parent edge.
In order to formalise how the weight of an edge is subdivided amongst its daughters, we use maps into the standard $d$-simplex, where $d$ is the number of daughter edges.

\begin{defin}
Given a weighted digraph $G=(V, E, w)$ 
\begin{enumerate}[label=(\alph*)]
  \item A \mdf{subdivision} $S$ of $G$ is a choice of edges $F \subseteq E$, along with a map $S:F\to\sqcup_{d\in\N} \Delta^d$ from edges in $F$ to the formal disjoint union of all standard $d$-simplices.
\end{enumerate}
\end{defin}

Intuitively, a subdivision gives us a recipe for subdividing the edges of $F$ where $S(e)_i$ describes the fraction of $w(e)$ which the $i^{th}$ daughter edge of $e$ should receive.

\begin{notation}
Given a subdivision $S:F\to\sqcup_{d\in\N}\Delta^d$,
\begin{enumerate}[label=(\alph*)]
  \item Let $\mdf{d(e)}$ denote the simplex dimension such that $S(e) \in \Delta^{d(e)}$.
  \item Let $\mdf{CS(e)}$ denote the $d(e)$-tuple of cumulative sums, i.e.\ $CS(e)_i \defeq \sum_{j=1}^i S(e)_j$.
\end{enumerate}
\end{notation}

\begin{defin}
Given a subdivision $S:F\to\sqcup_{d\in\N}\Delta^d$, define
\begin{align*}
  V_S &\defeq V_{old} \sqcup V_{new} \defeq V \sqcup \bigsqcup_{e \in F} \{ v_{e, 1}, \dots, v_{e, d(e) -1}\}  \\
  E_S &\defeq E_{old} \sqcup E_{new}\defeq (E \setminus F) \sqcup \bigsqcup_{e\in F} \{ \tau_{e, 1}, \dots, \tau_{e, d(e)}\} \\
  w_S(\tau) & \defeq
  \begin{cases}
    w(\tau) & \text{if }\tau \in E_{old} \\
    S(e)_i \cdot w(e) & \text{if }\tau = \tau_{e, i}
  \end{cases}
\end{align*}
where $\tau_{e, i} = (v_{e, i-1}, v_{e, i})$ and we denote $v_{e, 0} \defeq \st(e)$ and $v_{e, d(e)} \defeq \fn(e)$.
We then define $\mdf{\wdop{s}{S}{G}}\defeq (V_S, E_S, w_S)$.
\end{defin}

We show that the descriptor is stable to arbitrary subdivisions of arbitrary subsets of edges.
Moreover, this stability is local since the bound depends only on the weight of subdivided edges.

\begin{theorem}\label{thm:subdiv_stability}
Given a weighted digraph $G=(V, E, w)\in\WDgr$ and any subdivision $S:F\to\sqcup_{d\in\N}\Delta^d$,
\begin{equation}
d_B\big(\thedesc{G} , \thedesc{\wdop{s}{S}{G}} \big) \leq \max_{e\in F}w(e).
\end{equation}
\end{theorem}
\begin{proof}
First, we setup some notation.
Denote $G_S\defeq \wdop{s}{S}{G} = (V_{old} \sqcup V_{new}, E_{old}\sqcup E_{new}, w_S)$.
We let $d$ and $d_S$ denote the shortest-path quasimetric on $G$ and $G_S$ respectively.
Finally, let $W\defeq \max_{e \in F} w(G)(e)$.

Our strategy is to employ the isometry theorem (Theorem~\ref{thm:isometry})
Define the following vertex maps
\begin{align}
    f: V(G) \to V(G_S) \quad & v \mapsto v; \\
    g: V(G_S) \to V(G) \quad & v \mapsto
    \begin{cases}
        v &\text{if } v\in V_{old},\\
        \st(e) & \text{if }v_{e, i} \in V_{new} \text{ and }CS(e)_i < 1/2, \\
        \fn(e) & \text{if }v_{e, i} \in V_{new} \text{ and }CS(e)_i \geq 1/2,
    \end{cases}
\end{align}
which are visualized in Figure~\ref{fig:subdiv_vertex_map}.

\begin{figure}[htb]
  \centering
  \begin{tikzpicture}[
  roundnode/.style={circle, fill=black, minimum size=4pt},
	squarenode/.style={fill=black, minimum size=4pt},
	inner sep=2pt,
	outer sep=1pt
  ]

  \node(a1) at (0, 2) [roundnode, label=above:$a$] {};
  \node(b1) at (4, 2) [roundnode, label=above:$b$] {};

  \draw[->] (a1)--(b1);

  \node (a) at (0, 0) [roundnode, label=below:$a$] {};
	\node (b) at (1, 0) [roundnode, label=below:$v_{e, 1}$] {};
	\node (c) at (2, 0) [roundnode, label=below:$v_{e, 2}$] {};
	\node (d) at (3, 0) [roundnode, label=below:$v_{e, 3}$] {};
	\node (e) at (4, 0) [roundnode, label=below:$b$] {};

  \draw[->] (a)--(b);
  \draw[->] (b)--(c);
  \draw[->] (c)--(d);
  \draw[->] (d)--(e);

  \path[|->, blue] (a1) edge [bend right] (a);
  \path[|->, blue] (b1) edge [bend left] node[anchor=west, right=3pt] {f} (e);

  \path[|->, red] (a) edge (a1);
  \path[|->, red] (b) edge (a1);
  \path[|->, red] (c) edge node[anchor=east, left=7pt] {g} (b1);
  \path[|->, red] (d) edge (b1);
  \path[|->, red] (e) edge (b1);

  \node[] at (-1.5, 2) {$G:$};
  \node[] at (-1.5, 0) {$G_S:$};
\end{tikzpicture}
  \caption{Visualising the vertex maps $f$ and $g$ under the subdivision $S(e) = (1/4, 1/4, 1/4, 1/4)$ where $e=(a, b)$.}
  \label{fig:subdiv_vertex_map}
\end{figure}
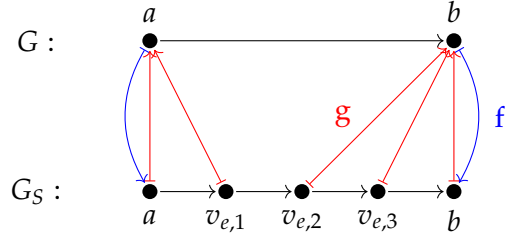

\begin{claim}\label{claim:paths_conserved}
For vertices $i, j \in V_{old}$, there is a path $i\leadsto j$ in $G$ of length $t$ if and only if there is one in $G_S$.
\end{claim}
\begin{poc}
This is clear to see, since the weight of an edge is shared amongst its daughter edges in the subdivision.
\end{poc}

\begin{claim}
    For any $t\geq 0$, $f$ defines a digraph map $G^t \to G_S^{t+W}$ and $G\cup G^t \to G_S \cup G_S^{t+W}$.
\end{claim}
\begin{poc}
Since $f$ is just the inclusion vertex map, Claim~\ref{claim:paths_conserved} shows that $f$ defines a digraph map $G^t \to G_S^t$ and so certainly $G^t \to G_S^{t+W}$.
For the second map, pick an edge $e \in E$ and note $f(e) = e$.
If $e\not\in F$ then it is undivided and $e \in E(G_s)$.
Otherwise $e\in F$ and $e\not\in E(G_s)$, however we note $d_S(\st(e), \fn(e)) \leq w(e) \leq W$.
Therefore, for any $t\geq 0$, $e \in G_s^t+W$ and hence $f$ defines a digraph map $G\cup G^t \to G_S \cup G_S^{t+W}$.
\end{poc}

\begin{claim}
  For any $t\geq 0$, $g$ defines a digraph map $G_S^t \to G^{t+W}$  and $G_S \cup G_S^t \to G \cup G^{t+W}$.
\end{claim}
\begin{poc}
Given an edge $\tau= (a, b)\in E(G_S^t)$ there is a path $p:a\leadsto b$ in $G_S$ of length at most $t$.
We may assume that $g(a)\neq g(b)$, else there is nothing to check for this edge.
If $a=v_{e, i}$ is a new vertex from subdividing an edge $e\in F$ then $g(a)$ is either $\st(e)$ or $\fn(e)$.
Either by adding or removing relevant daughter edges of $e$ to/from the start of $p$, we obtain a new path $g(a)\leadsto b$ in $G_S$.
By construction this, will add at most $w(e)/2 \leq W/2$ to the length of $p$.
Likewise we can alter the end of $p$ to obtain a path $g(a) \leadsto g(b)$ in $G_S$ of length at most $t+W$.
By Claim~\ref{claim:paths_conserved}, we see $g(\tau)\in E(G^{t+W})$.

Finally, given an edge $\tau\in G_s$ there are two cases.
If $\tau \in E_{old}$ then the edge is preserved under $g$.
Else $\tau = \tau_{e, i} \in E_{new}$ in which case either $\tau$ is collapsed to one of the endpoints of $e$, or it is mapped to $e$.
Hence $g$ is digraph map $G_S \to G$ and the final requirement follows.
\end{poc}

Therefore, $f$ and $g$ are $W$-shifting vertex maps and induce morphisms
\begin{equation}
\shiftinduced{ch}{f}{W}:\zb{C}(G)\to\shifted{\zb{C}(G_S)}{W}
\quad \text{ and }\quad
\shiftinduced{ch}{g}{W}:\zb{C}(G_S)\to\shifted{\zb{C}(G)}{W}.
\end{equation}
Now note that, as vertex maps $g\circ f = \id_{V(G)}$.
Therefore, at the level of homology we have
$
  \shiftinduced{hom}{g}{W} \circ \shiftinduced{hom}{f}{W} 
  = \transmorph{\zbpershom{G}}{2W}.
$

Composing vertex maps in the opposite order, we do not obtain the identity.
Moreover, $\shiftinduced{ch}{f}{W} \circ \shiftinduced{ch}{g}{W} \neq \transmorph{\zb{C}(G_s)}{2W}$.
However, we will show that, at every $t\geq 0$, the chain maps on either side of this inequality differ by a boundary.
First, choose a basis $\{ c_1, \dots, c_k \}$ of simple undirected circuits for $\zbcycles{G_S}{0}$.
By Lemma~\ref{prop:cycles_born_at_0},
it suffices to prove that for each $c_i$,
\begin{equation}\label{eq:interleaving_equality}
    \shiftinduced{ch}{f}{W} \shiftinduced{ch}{g}{W} \zbinclinduc{ch}{0}{t} c_i = \zbinclinduc{ch}{0}{t+2W} c_i \pmod{B} 
\end{equation}
where $B \defeq \zbbdrs{G_S}{t+2W}$.
Since simple undirected circuits are non-backtracking and the vertices $\{v_{e, 1}, \dots, v_{e, d(e)-1}\}$ have in-degree $1$ and out-degree $1$, we can write
\begin{equation}
  c_i = \sum_{e \in F}\alpha_e (\tau_{e, 1} + \dots + \tau_{e, d(e)}) +
  \sum_{e \in E\setminus F}\alpha_e e
\end{equation}
for some $\alpha_e\in\{0, \pm 1\}$.
Therefore, it suffices to prove the following two claims.

\begin{claim}
  For each $e \in F$, $\shiftinduced{ch}{f}{W} \shiftinduced{ch}{g}{W} (\tau_{e, 1} + \dots + \tau_{e, d(e)}) = e$ and for each $e \in E\setminus F$, $\shiftinduced{ch}{f}{W} \shiftinduced{ch}{g}{W} e = e$.
\end{claim}
\begin{poc}
First note that $f$ fixes the endpoints of every edge $e\in E$ and hence $\shiftinduced{ch}{f}{W} e = e$.
Now, choose arbitrary $e \in F$.
There exists some $M$ such that $g(v_{e, i}) = \st(e)$ for all $i< M$ and $g(v_{e, i}) = \fn(e)$ for all $i\geq M$.
Hence $\shiftinduced{ch}{g}{W} \tau_{e, i} = 0$ for all $i < M$ and all $i > M$ but $\shiftinduced{ch}{g}{W} \tau_{e, M} = e$.
Hence
\begin{equation}
  \shiftinduced{ch}{f}{W} \shiftinduced{ch}{g}{W} (\tau_{e, 1} + \dots + \tau_{e, d(e)}) = \shiftinduced{ch}{f}{W} e = e.
\end{equation}
Finally for $e \in E \setminus F$, $g$ fixes the endpoints of $e$ and hence $\shiftinduced{ch}{g}{W} e = e$.
\end{poc}

\begin{claim}
  For each $e \in F$, $\tau_{e,1} + \dots + \tau_{e, d(e)} = e \pmod{B}$.
\end{claim}
\begin{poc}
Choose any $e \in F$.
The path $(\tau_{e, 1}, \dots, \tau_{e, d(e)})$ is a path of length $W$ in $G_S$, between the endpoints of $e$.
The claim now follows by Lemma~\ref{lem:paths_homologous}.
\end{poc}

Hence 
$\shiftinduced{ch}{f}{W} \shiftinduced{ch}{g}{W} \zbinclinduc{ch}{0}{t} c_i = \zbinclinduc{ch}{0}{t+2W} c_i \pmod{B}$
for each $c_i$ and each $t\geq 0$,
so, at the level of homology,
$\shiftinduced{hom}{f}{W} \circ \shiftinduced{hom}{g}{W} = \transmorph{\zbpershom{G_S}}{2W}$.
Therefore $\shiftinduced{hom}{f}{W}$ and $\shiftinduced{hom}{g}{W}$ constitute a $W$-interleaving
and the bound on bottleneck distance follows by the isometry theorem.
\end{proof}

\begin{rem}
Since subdividing an edge does not effect circuit rank of $\underlying{G}$, the number of features does not change upon subdivision (by Corollary~\ref{cor:decreasing_curves}).
\end{rem}

\begin{defin}\label{defin:ims}
  Fix a weighted digraph $G=(V, E, w)\in\WDgr$.
  \begin{enumerate}[label=(\alph*)]
    \item The \mdf{medial subdivision}, $\mdf{S_{med}(G)}:E(G)\to\Delta^2$, is given by $S(e) = (1/2, 1/2)$ for every $e \in E(G)$.
    \item The \mdf{$n^{th}$ iterated medial subdivision of $G$}, \mdf{$\IMS{G}{n}$}, is defined iteratively as follows.
      Firstly, $\IMS{G}{0}\defeq G$ then for each $n$, we define $\IMS{G}{n} \defeq \wdop{s}{S}{\IMS{G}{n-1}}$ where $S=S_{med}(\IMS{G}{n-1})$.
  \end{enumerate}
\end{defin}

\begin{cor}\label{cor:gen_ims_convergence}
  Given a weighted digraph $G\in\WDgr$, the sequence of barcodes
  $(\thedesc{\IMS{G}{n}})_{n\in\N}$ converges under the bottleneck distance.
\end{cor}
\begin{proof}
We first note that
\begin{equation}
  \max_{e\in E(\IMS{G}{n})} w(e) = \frac{1}{2^n}\max_{e\in E(G)} w(e).
\end{equation}
Hence Theorem~\ref{thm:subdiv_stability} implies that the sequence of barcodes is Cauchy.
The space of persistence diagrams with the bottleneck distance is complete~\cite{Che2022} and hence the sequence of barcodes converges.
\end{proof}

While we have bottleneck stability, we do \emph{not} have $1$-Wasserstein stability.

\begin{prop}\label{prop:subdiv_wass_instability}
There exists no function $f: \WDgr \to \R$ such that for any weighted digraph $G=(V, E, w)\in\WDgr$ and any subdivision $S: F \to \Delta^d$ we have
\begin{equation}
  d_{W_1}\big(
    \thedesc{\wdop{s}{S}{G}} ,
   \thedesc{G} 
    \big) \leq f\big( \nbhdgraph(F; G) \big).
\end{equation}
\end{prop}
\begin{proof}
Suppose such $f$ exists and consider the following sequence of digraphs in which each edge has unit weight.
\begin{figure}[H]
  \centering
\begin{tikzpicture}[
  roundnode/.style={circle, fill=black, minimum size=4pt},
	squarenode/.style={fill=black, minimum size=4pt},
	inner sep=2pt,
	outer sep=1pt
  ]

  \node (a1) at (0, 0) [roundnode, label=below:$v_0$] {};
  \node (x1) at (0.66, 0) [roundnode, label=below:$v_1$] {};
  \node (y1) at (1.33, 0) [roundnode, label=below:$v_2$] {};
  \node (b1) at (2, 0) [roundnode, label=below:$v_3$] {};
  \node (c1) at (2, 2) [roundnode] {};
  \node (d1) at (0, 2) [roundnode] {};
  \draw[->] (a1)--(x1);
  \draw[->] (x1)--(y1);
  \draw[->] (y1)--(b1);
  \draw[->] (b1)--(c1);
  \draw[->] (d1)--(c1);
  \draw[->] (d1)--(a1);

  \node (a2) at (5, 0) [roundnode, label=below:$v_0$] {};
  \node (x2) at (5.66, 0) [roundnode, label=below:$v_1$] {};
  \node (y2) at (6.33, 0) [roundnode, label=below:$v_2$] {};
  \node (b2) at (7, 0) [roundnode, label=below:$v_3$] {};
  \node (c2) at (7, 2) [roundnode] {};
  \node (d2) at (5, 2) [roundnode] {};
  \node (c22) at (6.6, 1.6) [roundnode] {};
  \node (d22) at (4.6, 1.6) [roundnode] {};
  \draw[->] (a2)--(x2);
  \draw[->] (x2)--(y2);
  \draw[->] (y2)--(b2);
  \draw[->] (b2)--(c2);
  \draw[->] (d2)--(c2);
  \draw[->] (d2)--(a2);
  \draw[->] (b2)--(c22);
  \draw[->] (d22)--(c22);
  \draw[->] (d22)--(a2);

  \node (a3) at (10, 0) [roundnode, label=below:$v_0$] {};
  \node (x3) at (10.66, 0) [roundnode, label=below:$v_1$] {};
  \node (y3) at (11.33, 0) [roundnode, label=below:$v_2$] {};
  \node (b3) at (12, 0) [roundnode, label=below:$v_3$] {};
  \node (c3) at (12, 2) [roundnode] {};
  \node (d3) at (10, 2) [roundnode] {};
  \node (c32) at (11.6, 1.6) [roundnode] {};
  \node (d32) at (9.6, 1.6) [roundnode] {};
  \node (c33) at (11.2, 1.2) [roundnode] {};
  \node (d33) at (9.2, 1.2) [roundnode] {};
  \draw[->] (a3)--(x3);
  \draw[->] (x3)--(y3);
  \draw[->] (y3)--(b3);
  \draw[->] (b3)--(c3);
  \draw[->] (d3)--(c3);
  \draw[->] (d3)--(a3);
  \draw[->] (b3)--(c32);
  \draw[->] (d32)--(c32);
  \draw[->] (d32)--(a3);
  \draw[->] (b3)--(c33);
  \draw[->] (d33)--(c33);
  \draw[->] (d33)--(a3);

  \node[] at (1, -1) {$G_1$};
  \node[] at (6, -1) {$G_2$};
  \node[] at (11, -1) {$G_3$};
\end{tikzpicture}
    \caption{
      A sequence of weighted digraphs $G_n$, in which all edges have unit weight and $\card{\thedesc{G_n}}=n$.
    }
    \label{fig:bad_split}
\end{figure}
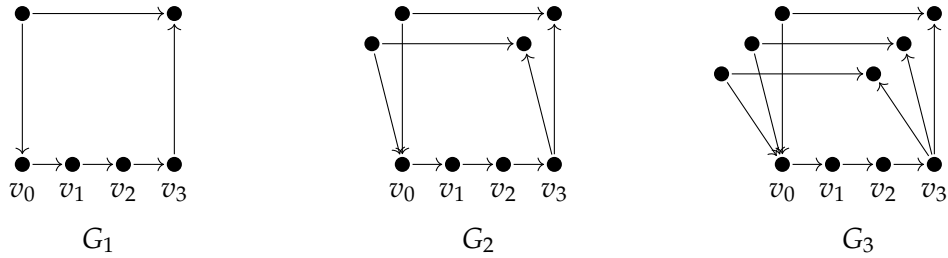\noindent
Intuitively, $G_n$ is constructed by gluing $n$ disjoint copies of $G_1$ along the path $(v_0, v_1, v_2, v_3)$.
Note that each copy of $G_1$ introduces a feature which dies at $t=3$ so
\begin{equation}
  \thedesc{G_n} = \ms{[0, 3)\text{ with multiplicity } n }.
\end{equation}
Upon subdividing the edge $e=(v_1, v_2)$ via $S(e)=(1/2, 1/2)$, each feature changes to $[0, 2.5)$.
Hence 
\begin{equation}
  d_{W_1}\big(
    \thedesc{\wdop{s}{S}{G}} ,
   \thedesc{G} 
    \big) = 0.5n
\end{equation}
which eventually exceeds the constant $f(\nbhdgraph(e; G))$.
\end{proof}

\subsection{Edge collapse}

Another potential structural perturbation is that of edge collapses, in which the two end points of an edge are identified and the edge deleted.
In applications, this may happen particularly to low-weight edges, which cannot be discerned by the imaging method and hence collapsed to a vertex instead.
Since we interpret edge-weights as corresponding to distance, we add half the weight of the collapsed edge to each of its neighbours so that the length of paths through the collapsed edge are not changed.

\begin{defin}
Given a weighted digraph $G=(V, E, w)$ 
and an edge $e=(a, b)\in E$, define
\begin{align}
  V_e &\defeq \faktor{V}{a\sim b}\\
  E_e &\defeq \faktor{(E\setminus\{e\})}{(i, j)\sim (i', j') \iff i \sim i', j\sim j'}\\
  w_e(\sigma) &\defeq \min_{\tau \in \sigma} \left( w(\tau) + \frac{w(e)}{2}\cdot \indic{\tau \in \nbhd(e)} \right)
\end{align}
where $\indic{\tau \in \nbhd(e)} = 1 \iff \tau \in \nbhd(e)$, else $\indic{\tau\in\nbhd(e)}=0$.
We then define $\mdf{\wdop{c}{e}{G}} \defeq (V_e, E_e, w_e)$.
\end{defin}

Some edge collapses do not drastically alter the topological structure of the graph or the shortest-path quasimetric.
Therefore, we can get a local stability bound on a subset of such operations.

\begin{theorem}\label{thm:collapse_stability}
Given a weighted digraph, $G=(V, E, w)\in\WDgr$ and an edge $e = (a, b) \in E$.
Suppose $e$ is the only outgoing edge from $a$ and the only incoming edge to $b$ (as in Figure~\ref{fig:collapse_stability}), then
\begin{equation}
 d_B\big(\thedesc{G}, \thedesc{\wdop{c}{e}{G}}\big)
 \leq w(e) +
 \min\left( \max_{v\in\nbhd_{in}(a)} w(v, a) , \max_{v\in\nbhd_{out}(b)} w(b, v) \right)
 \eqdef\delta.
\end{equation}
\end{theorem}
\begin{figure}[H]
  \centering
  \begin{tikzpicture}[
  roundnode/.style={circle, fill=black, minimum size=4pt},
	squarenode/.style={fill=black, minimum size=4pt},
	inner sep=2pt,
	outer sep=1pt
  ]

  \node (a) at (0, 0)[roundnode, label=above:$a$] {};
  \node (b) at (1, 0)[roundnode, label=above:$b$] {};
  \node (c) at (-2, -1)[roundnode] {};
  \node (d) at (-2, 0)[roundnode] {};
  \node (e) at (-2, 1)[roundnode] {};
  \node (f) at (3, -1)[roundnode] {};
  \node (g) at (3, 0)[roundnode] {};
  \node (h) at (3, 1)[roundnode] {};

  \draw[->] (a)--(b);
  \draw[->] (c)--(a);
  \draw[->] (d)--(a);
  \draw[->] (e)--(a);
  \draw[->] (b)--(f);
  \draw[->] (b)--(g);
  \draw[->] (b)--(h);

  \node (a2) at (7.5, 0)[roundnode, label=above:$v_e$] {};
  \node (c2) at (5, -1)[roundnode] {};
  \node (d2) at (5, 0)[roundnode] {};
  \node (e2) at (5, 1)[roundnode] {};
  \node (f2) at (10, -1)[roundnode] {};
  \node (g2) at (10, 0)[roundnode] {};
  \node (h2) at (10, 1)[roundnode] {};

  \draw[->, rounded corners] (c2)--(6.8, -0.1)--(a2);
  \draw[->, rounded corners] (d2)--(6.8, 0)--(a2);
  \draw[->, rounded corners] (e2)--(6.8, 0.1)--(a2);
  \draw[->, rounded corners] (a2)--(8.2, -0.1)--(f2);
  \draw[->, rounded corners] (a2)--(8.2, 0)--(g2);
  \draw[->, rounded corners] (a2)--(8.2, 0.1)--(h2);

  \node[] at (0.5, -1.5) {$G$};
  \node[] at (7.5, -1.5) {$\wdop{c}{e}{G}$};

\end{tikzpicture}
  \caption{Schematic of $\nbhdgraph(e)\subseteq G$ and $\nbhdgraph(v_e)\subseteq \wdop{c}{e}{G}$, under the assumptions of Theorem~\ref{thm:collapse_stability}.}\label{fig:collapse_stability}
\end{figure}
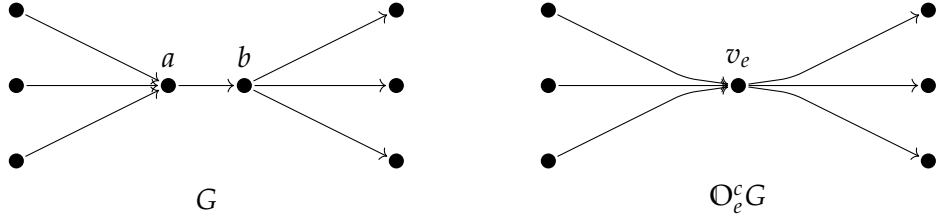
\begin{proof}
Denote $G_e \defeq \wdop{c}{e}{G} = (V_e, E_e, w_e)$.
First note that the condition ensures that $\{a, b\}$ are the only vertices which get identified.
Moreover, no edges are identified, although $e$ is collapsed and the weights on neighbours of $e$ may change.
To ease notation we drop all equivalence class notation for the vertices and edges of $G_e$ and refer to the new vertex as $v_e\defeq \{ a, b\}$.

We define two vertex maps.
Firstly $f: V \to V_e$ is given by $v\mapsto v$ for $v\neq a, b$ and $a, b\mapsto v_e$.
Secondly, we define $g:V_e \to V$ as follows.
For most elements of $V_e$ we choose the only representative of the equivalence class $v\mapsto v$.
The only class containing more than one element is $v_e=\{a, b \}$.
For this class, we choose $g(v_e)=a$ if
\begin{equation}
  \max_{v\in\nbhd_{in}(a)} w(v, a) \geq \max_{v\in\nbhd_{out}(b)} w(b, v),
\end{equation}
else we choose $g(v_e) = b$.
Let us assume that $g(v_e) = a$; the other case admits a similar proof.
We show that $f$ and $g$ define $\delta$-shifting vertex maps.

\begin{claim}
  $f$ defines a digraph map $G \to G_e$.
\end{claim}
\begin{poc}
All edges of $G$ are mapped to edges of $G_e$, with the exception of $e=(a, b)$.
The two endpoints of $e$ are mapped to the same point, $v_e$.
Therefore, $f$ define a digraph map as required.
\end{poc}

\begin{claim}
  $f$ defines a digraph map $G^t\to G_e^{t+\delta}$ for all $t\geq 0$.
\end{claim}
\begin{poc}
If $(i, j)\in E(G^t)$, then there is a path $p$ joining $i\leadsto j$ of length at most $t$ in $G$.
Thanks to the previous claim, $f(p)$ is a path $f(i) \leadsto f(j)$ in $G_e$.
Thanks to the assumption on $e$, $p$ passes through at most one incoming edge to $a$ and at most one outgoing edge from $b$.
Since these are the only edges whose weights are increased and each of these is increased by $w(e)/2$, we see $f(p)$ is of length at most $t + w(e)$.
Therefore $(i, j) \in E(G_e^{t+\delta})$.
\end{poc}

\begin{claim}\label{claim:collapse_s_dig_map}
  $g$ defines a digraph map $G_e^t \to G^{t+\delta}$.
\end{claim}
\begin{poc}
Suppose $(i, j) \in E(G_e^t)$, then there is a path $p:i\leadsto j$ in $G_e$ of length at most $t$.
Suppose $p$ does not traverse $v_e$, then it also doesn't traverse any edge incident to $v_e$.
Therefore $p$ also exists in $G$ and is of length at most $t$.
Hence, $(i, j) \in G^{t+\delta}$.

Conversely, suppose $p$ does traverse $v_e$ and write
\begin{equation}
  p = (v_0=i, v_1, \dots, v_{k-1}, v_e, v_{k+1}, \dots, v_l=j).
\end{equation}
Note that $v_{k-1}\in\nbhd_{in}(a)$ and $v_{k+1}\in\nbhd_{out}(a)$.
Replacing $v_e$ with the sequence $(a, b)$ we obtain a new path
\begin{equation}
  p' = (v_0=i, v_1, \dots, v_{k-1}, a, b, v_{k+1}, \dots, v_l=j).
\end{equation}
which exists in $G$.
We now split into cases.\\
\textbf{Case 1:}
If $i\neq v_e$ and $j\neq v_e$ then $g(i)=i$ and $g(j)=j$ and $p'$ is a path $g(i)=i\leadsto j=g(j)$. \\
\textbf{Case 2:}
If $i=v_e$ then $j\neq v_e$ so $g(i)=a$ and $g(j)=j$ and $p'$ is a path $g(i)=a\leadsto j= g(j)$. \\
\textbf{Case 3:}
If $j=v_e$ then $i\neq v_e$ so $g(i)=i$ and $g(j)=a$ but $p'$ is a path $g(i)=i\leadsto b$ which ends $(a, e, b)$; removing the last edge yields a path $g(i)=i\leadsto a= g(j)$.

Note that edge-weights may decrease moving from $G_e$ to $G$ however we may add one additional edge, namely $e$.
Therefore, the length of $p'$ is at most $t+w(e)$ and hence $(i, j) \in G_e^{t+\delta}$.
\end{poc}

\begin{claim}
  $g$ defines a digraph map $G_e \cup G_e^t \to G \cup G^{t+\delta}$.
\end{claim}
\begin{poc}
Thanks to the previous claim, we only need to check edges $(i, j)\in E(G_e)$.
Any edge $(i, j)\in E(G_e)$ which is not incident to $v_e$ is mapped by $g$ to itself $(i, j)\in E(G)$.
By the assumption on $e$, if $j = v_e$, then $i\in\nbhd_{in}(a)$ and hence $(g(i), g(v_e)) = (i, a) \in E(G)$.
Else, suppose $i = v_e$, then $j \in \nbhd_{out}(b)$.
Note that the path $(a, b, j)$ in $G$ is of length $w(e)+w(b, j) \leq \delta$.
Therefore $(g(v_e), g(j))=(a, j)\in E(G^{t+\delta})$ for all $t\geq 0$.
\end{poc}

As vertex maps $f\circ g = \id_{V_e}$ and hence $\shiftinduced{hom}{f}{\delta} \shiftinduced{hom}{g}{\delta} = \transmorph{\zb{C}(G_e)}{2\delta}$.
However, it is not the case that $g\circ f = \id_V$ nor $\shiftinduced{ch}{g}{\delta} \shiftinduced{ch}{f}{\delta} = \transmorph{\zb{C}(G)}{2\delta}$.
However, these two chain maps do agree at the level of homology, as we now show.
We follow a similar approach to the proof of Theorem~\ref{thm:subdiv_stability}.

\begin{claim}
  Given a simple undirected circuit $c\in\zbcycles{G}{0}$, we can write
  \begin{equation}
    c = \alpha_e\big( v_1 a + ab + b v_2 \big) + \sum_{\substack{\tau \in E(G)\\ \tau\neq e, \tau \not\in\nbhd(e)}}\alpha_\tau \tau
  \end{equation}
  for some $v_1 \in \nbhd_{in}(a)$, $v_2\in\nbhd_{out}(b)$ and $\alpha_e, \alpha_\tau \in \{ 0, \pm 1 \}$.
\end{claim}
\begin{poc}
This is a direct consequence of the requirement that $e$ is the only outgoing edge from $a$ and the only incoming edge to $b$.
\end{poc}

\begin{claim}
 For $v_1 \in \nbhd_{in}(a)$ and $v_2 \in \nbhd_{out}(b)$, we have $\shiftinduced{ch}{g}{\delta} \shiftinduced{ch}{f}{\delta} (v_1 a + ab + bv_2) = v_1 a + av_2$ and for $\tau\in E(G)\setminus\nbhd(e), \tau \neq e$ we have $\shiftinduced{ch}{g}{\delta} \shiftinduced{ch}{f}{\delta} \tau = \tau$.
\end{claim}
\begin{poc}
For the first equality, note that $a$, $v_1$ and $v_2$ are fixed by $g\circ f$ whereas $(g\circ f)(b) = a$.
For the second, any edge $\tau \in E(G) \setminus\nbhd(e)$, $\tau \neq e$ does not have $b$ as one of its endpoints and hence both endpoints of $\tau$ are fixed by $g \circ f$.
\end{poc}

Finally it remains to show that for any $v_1\in\nbhd_{in}(a)$ and $v_2\in\nbhd_{out}(b)$,
\begin{equation}
  v_1 a + ab + bv_2 = v_1 a + a v_2 \pmod{\zbbdrs{G}{t+\delta}}
\end{equation}
for all $t\geq 0$.
Note that $(a, b, v_2)$ is a path $a\leadsto v_2$ of length at most $\delta$ and hence the directed triangle $(a, b, v_2)$ is in the digraph $G^{t+\delta}$ for all $t\geq 0$.
The claim then follows since $\zb{\bd}_2(a b v_2) = ab + bv_2 - a v_2$.
\end{proof}

\begin{rem}
  If $G$ is a DAG and an edge $e=(a,b)$ satisfies the condition of Theorem~\ref{thm:collapse_stability}, then there is a topological ordering of $G$ with $a$ and $b$ adjacent.
  However, note that this is \emph{not} a sufficient condition for local stability to edge collapse.
\end{rem}

\begin{rem}
  Collapses of the sort described in Theorem~\ref{thm:collapse_stability} remove exactly one vertex and one edge and do not change the number of weakly connected components.
  Hence, by Corollary~\ref{cor:decreasing_curves} we have
  $\card{\thedesc{G}}= \card{\thedesc{\wdop{c}{e}{G}}}.$
\end{rem}

While some edge collapses are relatively minor, in general they can drastically alter the topology of the digraph.
As such, we cannot expect local stability to arbitrary edge collapses.

\begin{theorem}\label{thm:collapse_instability}
There exists no function $f: \WDgr \to \R$ such that for any weighted digraph $G=(V, E, w)\in\WDgr$ and any edge $e\in E$ therein we have
\begin{equation}
 d_B\big(\thedesc{G}, \thedesc{\wdop{c}{e}{G}}\big) \leq f\big( \nbhdgraph(e; G) \big).
\end{equation}
\end{theorem}
\begin{proof}
Suppose such $f$ exists  then 
consider the following weighted digraph $G$,
where $e\defeq(v_0, v_5)$ and $W\defeq 2f(\nbhdgraph(e; G))+4$.
\begin{figure}[H]
  \centering
  \begin{tikzpicture}[
  roundnode/.style={circle, fill=black, minimum size=4pt},
	squarenode/.style={fill=black, minimum size=4pt},
	inner sep=2pt,
	outer sep=1pt
  ]
  %

  \node (a) at (0, 0)[roundnode, label=left:$v_0$] {};
  \node (b) at (1, -1)[roundnode, label=left:$v_1$] {};
  \node (c) at (2, -1)[roundnode, label=right:$v_2$] {};
  \node (d) at (1, 1)[roundnode, label=left:$v_3$] {};
  \node (e) at (2, 1)[roundnode, label=right:$v_4$] {};
  \node (f) at (3, 0)[roundnode, label=right:$v_5$] {};

  \node (a2) at (6, 0)[roundnode, label=left:$v_e$] {};
  \node (b2) at (5.5, -1)[roundnode, label=left:$v_1$] {};
  \node (c2) at (6.5, -1)[roundnode, label=right:$v_2$] {};
  \node (d2) at (5.5, 1)[roundnode, label=left:$v_3$] {};
  \node (e2) at (6.5, 1)[roundnode, label=right:$v_4$] {};

  \draw[->, rounded corners, red] (c)--(2, -1.5)--(-1, -1.5)--(-1, 1.5) node[midway, sloped, above] {\tiny $W$} --(1, 1.5) --(d);
  \draw[->, blue] (a)--(b);
  \draw[->, blue] (b)--(c);
  \draw[->, blue] (a)--(d);
  \draw[->, blue] (d)--(e);
  \draw[->, blue] (c)--(f);
  \draw[->, blue] (e)--(f);
  \draw[->, blue] (a)--(f);

  \draw[->, rounded corners, red] (c2)--(6.5, -1.5)--(4.5, -1.5)--(4.5, 1.5) node[midway, sloped, above] {\tiny $W$} --(5.5, 1.5)--(d2);
  \draw[->, blue!50!red] (a2)--(b2) node[midway, sloped, above] {\tiny $1.5$};
  \draw[->, blue] (b2)--(c2);
  \draw[->, blue!50!red] (a2)--(d2) node[midway, sloped, below] {\tiny $1.5$};
  \draw[->, blue] (d2)--(e2);
  \draw[->, blue!50!red] (c2)--(a2) node[midway, sloped, above] {\tiny $1.5$};
  \draw[->, blue!50!red] (e2)--(a2) node[midway, sloped, below] {\tiny $1.5$};


  \node[] at (1.5, -2) {$G$};
  \node[] at (6, -2) {$\wdop{c}{e}{G}$};

\end{tikzpicture}
  \caption{An example weighted digraph which illustrates that $\zbpershommap$ is locally unstable to arbitrary edge collapses.
  Unlabelled edges have weight $1$.}\label{fig:collapse_instability}
\end{figure}
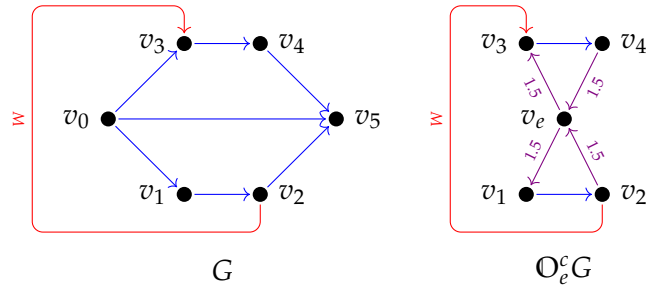\noindent
Note, $G$ has 3 features, which die at $2$, $2$ and $W$,
while $\wdop{d}{e}{G}$ has 3 features, which die at $2.5$, $2.5$ and $3$.
The longer feature, supported on the red edge $(v_2, v_3)$, has a reduced death-time in $\wdop{d}{e}{G}$ because there is a shortcut $(v_2, v_e, v_3)$, of length $3$.
Any bijection between these features (and the diagonals) must have bottleneck cost at least $\min(W/2, W-3) > f(\nbhdgraph(e; G))$.
\end{proof}

While this seems like a serious problem for our descriptor, note that the collapse in Figure~\ref{fig:collapse_instability} makes significant changes to the topology of the underlying digraph.
Originally, $G$ was a DAG with source $v_0$ and sink $v_5$; the edge collapse identified these two nodes and introduced directed cycles.
Moreover, in $G$ the only path $v_2\leadsto v_3$ was via the costly red edge but in $\wdop{d}{e}{G}$ there is a shortcut via $v_e$.
Therefore, since the profile of paths has changed drastically, it is arguably desirable that our descriptor changes too.

\subsection{Edge deletion}\label{sec:stability_edge_deletion}
\subsubsection{General case}

Another important class of structural perturbations is edge deletion.
Intuitively, as with edge collapse, some edge deletions can have drastic impact on the descriptor whereas some deletions are minor events.

\begin{defin}
Given a weighted digraph $G=(V, E, w)$ and an edge $e\in E$, we define
$\mdf{\wdop{d}{e}{G}}\defeq (V, E\setminus \{e\}, w_e)$ where $w_e$ is obtained by restricting $w$ to $E\setminus \{e\}$.
\end{defin}

We find that our descriptor is stable to deletions but the bound depends on the minimum length of a possible diversion.
In general, this diversion cost may be infinite.

\begin{theorem}\label{thm:edge_deletion_stability}
Given a weighted digraph $G=(V, E, w)\in\WDgr$ and an edge $e=(a, b)\in E$, 
let $d$ and $d_e$ denote the shortest-path quasimetric for $G$ and $\wdop{d}{e}{G}$ respectively.
Assume that $d_e(a, b)$ is finite, then
\begin{equation}
  d_B\big( \thedesc{G} , \thedesc{\wdop{d}{e}{G}}\big)
  \leq d_e(a, b).
\end{equation}
\end{theorem}
\begin{proof}
Denote $G_e\defeq\wdop{d}{e}{G}=(V_e, E_e, w_e)$ and $\delta\defeq d_e(a, b)$.
We first note that
\begin{equation}\label{eq:metric_difference}
 \max_{i, j \in V} \abs{d_e(i, j) - d(i, j)} \leq \abs{d_e(a, b) - d(a, b)} \leq d_e(a, b)=\delta.
\end{equation}
To see this, note that for arbitrary $i, j \in V$ we have $d_e(i, j) \geq d(i, j)$ since there are strictly fewer paths $i\leadsto j$ in $G_e$ than in $G$.
Moreover, there is path $p_e:a\leadsto b$ in $G_e$ of length at most $\delta$.
Then, given a path $p:i\leadsto j$ in $G$ of length $t$, the path contains $e$ at most once.
We can replace $e$ with $p_e$ to obtain a new path $i\leadsto j$ in $G_e$ of length at most $t+(\delta - w(e))\leq t+\delta$.
Therefore, $d_e(i, j) \leq d(i, j) + \delta$.

We claim that $\id_{V}$ constitutes a $d_e(a, b)$-shifting vertex map $G\to G_e$ and $G_e \to G$.
Then, by a similar argument to that of Theorem~\ref{thm:pertub_stability}, we obtain the result via the isometry theorem.
The inequalities of (\ref{eq:metric_difference}) automatically imply that
$\id_V$ defines a digraph map $G^t \to G_e^{t+\delta}$ and $G_e^t \to G^{t+\delta}$ for all $t\geq 0$.

Then $E_e \subseteq E$, so $\id_V$ certainly defines a digraph map $G_e \to G$ and thus $G_e \cup G_e^t \to G \cup G^{t+\delta}$ for all $t\geq 0$.
In the other direction, $E\setminus E_e = \{ e\}$ so $\id_V$ does \emph{not} define a digraph map $G \to G_e$
However $e\in E(G_e^{t+\delta})$ for all $t\geq 0$ and hence $\id_V$ \emph{does} define a digraph map $G\cup G^t \to G_e \cup G_e^{t+\delta}$ for all $t\geq 0$.
\end{proof}

In general, Theorem~\ref{thm:edge_deletion_stability} is a non-local bound, but if an edge has an alternative route in its local neighbourhood then the bound becomes local.

\begin{cor}\label{cor:local_edge_deletion}
Given a weighted digraph $G=(V, E, w)$ and an edge $e=(a, b)\in E$ such that there exists a vertex $v\in V$ such that $a \to v \to b$ then
\begin{equation}
  d_B\big( \thedesc{G} , \thedesc{\wdop{d}{e}{G}}\big)
  \leq w(a, v) + w(v, b).
\end{equation}
\end{cor}

Next, we consider the scenario where an edge $e$ is subdivided and then one of the daughter edges $\tau\defeq\tau_{e, i}$ is deleted (as shown in Figure~\ref{fig:edge_subdivision_deletion}).
In the resulting weighted digraph $\wdop{d}{\tau}{G}$, there is no alternative path between the endpoints of $\tau$ so the bound from Theorem~\ref{thm:edge_deletion_stability} would be infinite.
However, we can bound the effect of such an operation. 

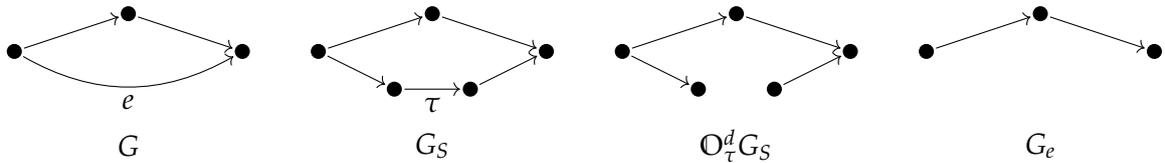
\begin{figure}[ht]
  \centering
  \begin{tikzpicture}[
  roundnode/.style={circle, fill=black, minimum size=4pt},
	squarenode/.style={fill=black, minimum size=4pt},
	inner sep=2pt,
	outer sep=1pt
  ]
  \node (a) at (0, 0) [roundnode] {};
  \node (b) at (1.5, 0.5) [roundnode] {};
  \node (c) at (3, 0) [roundnode] {};
  \node (d) at (1, -0.5) [roundnode] {};
  \node (e) at (2, -0.5) [roundnode] {};

  \node (a2) at (4, 0) [roundnode] {};
  \node (b2) at (5.5, 0.5) [roundnode] {};
  \node (c2) at (7, 0) [roundnode] {};
  \node (d2) at (5, -0.5) [roundnode] {};
  \node (e2) at (6, -0.5) [roundnode] {};

  \node (a3) at (-4, 0) [roundnode] {};
  \node (b3) at (-2.5, 0.5) [roundnode] {};
  \node (c3) at (-1, 0) [roundnode] {};

  \node (a4) at (8, 0) [roundnode] {};
  \node (b4) at (9.5, 0.5) [roundnode] {};
  \node (c4) at (11, 0) [roundnode] {};

  \draw[->] (a)--(b);
  \draw[->] (b)--(c);
  \draw[->] (a)--(d);
  \draw[->] (d)-- node[anchor=north, black] {$\tau$} (e);
  \draw[->] (e)--(c);

  \draw[->] (a2)--(b2);
  \draw[->] (b2)--(c2);
  \draw[->] (a2)--(d2);
  \draw[->] (e2)--(c2);

  \draw[->] (a3)--(b3);
  \draw[->] (b3)--(c3);
  \draw[->] (a3) to[bend right] node[anchor=north, black] {$e$} (c3);

  \draw[->] (a4)--(b4);
  \draw[->] (b4)--(c4);

  \node[] at (1.5, -1.25) {$G_S$};
  \node[] at (5.5, -1.25) {$\wdop{d}{\tau}{G_S}$};
  \node[] at (-2.5, -1.25) {$G$};
  \node[] at (9.5, -1.25) {$G_e$};
\end{tikzpicture}
  \caption{Schematic of the four weighted digraphs in Theorem~\ref{thm:subdiv_then_delete}, where $e$ is the curved bottom edge and $\tau$ is as labelled.}
  \label{fig:edge_subdivision_deletion}
\end{figure}

\begin{theorem}\label{thm:subdiv_then_delete}
  Given a weighted digraph $G=(V, E, w)\in\WDgr$, an edge $e=(a, b)\in E$ and a subdivision $S: \{e \} \to \Delta^d$,
  denote $G_S \defeq \wdop{s}{S}{G}$ and $G_e\defeq \wdop{d}{e}{G}$. 
  Let $d_e$ denote the shortest-path quasimetric in $G_e$.
  Choose any of the daughter edges $\tau\defeq\tau_{e, i} \in E(G_S)$, then
  \begin{equation}
  d_B\big( \thedesc{G_S} , \thedesc{\wdop{d}{\tau}{G_S}}\big)
  \leq
  \max\left(d_e(a, b) , \frac{w(e)}{2}\right) \eqdef \delta .
  \end{equation}
\end{theorem}
\begin{proof}
We first note that upon deleting $\tau$ from $G_S$, the remaining daughter edges and daughter vertices from the subdivision can be deleted, using Corollary~\ref{cor:edge_deletion_disconnect} and Corollary~\ref{cor:isolated_vertex_deletion}.
The remaining graph is precisely $G_e$ and so
$\zbpershom{\wdop{d}{\tau}{G_S}} \cong \zbpershom{G_e}.$
Hence, it remains to prove
\begin{equation}
 d_B\big(
 \thedesc{G_S} , \thedesc{G_e}
 \big)
  \leq \delta .
\end{equation}
The proof continues as an amalgamation of the proofs of Theorem~\ref{thm:subdiv_stability} and Theorem~\ref{thm:edge_deletion_stability}.
Define the following vertex maps
\begin{align}
    f: V(G_e) \to V(G_S) \quad & v \mapsto v; \\
    g: V(G_S) \to V(G_e) \quad & v \mapsto
    \begin{cases}
        v &\text{if } v\in V_{old},\\
        \st(e) & \text{if }v_{e, i} \in V_{new} \text{ and }CS(e)_i < 1/2, \\
        \fn(e) & \text{if }v_{e, i} \in V_{new} \text{ and }CS(e)_i \geq 1/2.
    \end{cases}
\end{align}
We note that $f$ induces a contraction digraph map $G_e \to G_S$
Hence, by Lemma~\ref{lem:shift_induce_cont}, it is a $\delta$-shifting vertex map and moreover,
$\shiftinduced{ch}{f}{\delta} = \transmorph{\zb{C}(G_s)}{\delta}\circ\zbdiginduc{ch}{f}$.

\begin{claim}
  $g$ induces a digraph map $G_S^{t} \to G_e^{t+\delta}$, for every $t\geq 0$.
\end{claim}
\begin{poc}
Given $(i, j)\in E(G_S^t)$, there is a path $p:i\leadsto j$ in $G_s$ of length at most $t$.
Since $d_e(a, b)\leq \delta$ there is a path $p_e: a\leadsto b$ in $G_e$ of length at most $\delta$.
We construct a new trail $p'$ in $G_e$ as follows.

If the entire sequence of daughter edges $(\tau_{e, 1}, \dots, \tau_{e, d(e)})$ appears in $p$ then we replace that sequence with $p_e$.
If $i \in V_{new}$ and $g(i)=a$ then we replace the initial sequence of daughter edges with $p_e$.
If $i \in V_{new}$ and $g(i) = b$ then we simply remove the initial sequence of daughter edges.
Likewise, if $j\in V_{new}$ and $g(j) = b$ then we replace the final sequence of daughter edges with $p_e$.
If $j \in V_{new}$ and $g(j) = a$ then we simply remove the final sequence of daughter edges.
This yields a trail $p': g(i) \leadsto g(j)$
Since $p$ cannot repeat edges, this construction inserts $p_e$ at most once and hence the length of $p'$ is at most $t+\delta$.
Therefore $(i, j) \in E(G_e^{t+\delta})$.
\end{poc}

\begin{claim}
  $g$ induces a digraph map $G_S\cup G_S^{t} \to G_e \cup G_e^{t+\delta}$, for every $t\geq 0$.
\end{claim}
\begin{poc}
  It remains to check the image of edge $e\in (G_S)$
Any un-subdivided edge $e\in E_{old}$ is preserved under $g$.
Given an edge $\tau_{e, i}=(x, y)\in E_{new}$ then $\tau=(v_{e, i-1}, v_{e, i})$ for some $i$ and there are three cases
\begin{equation}
  (g(v_{e, i-1}), g(v_{e, i})) = (\st(e), \st(e)) \text{ or } (\st(e), \fn(e)) \text{ or } (\fn(e), \fn(e)).
\end{equation}
Hence either $g(x) = g(y)$ or $(g(x), g(y)) = e$.
The edge $e$ does not appear in $G_e$ but it does appear in $G_e^{t+\delta}$ for all $t\geq 0$.
Therefore $g$ defines a digraph map as required.
\end{poc}

Again, we see $g \circ f = \id_{V(G)}$ and hence $\shiftinduced{hom}{g}{\delta} \circ \shiftinduced{hom}{f}{\delta} = \transmorph{\zbpershom{G_e}}{2\delta}$.
The proof that $\shiftinduced{hom}{f}{\delta}\circ\shiftinduced{hom}{g}{\delta}=\transmorph{\zbpershom{G_S}}{2\delta}$ is identical to the corresponding section in the proof of Theorem~\ref{thm:subdiv_stability}.
Note that we require $2\delta \geq w(e)$ so that we can apply Lemma~\ref{lem:paths_homologous} to show 
\begin{equation}
  \tau_{e, 1} + \dots + \tau_{e, d(e)} = e
  \pmod{\zbbdrs{G}{t+2\delta}},
\end{equation}
for all $t\geq 0$.
\end{proof}

In general, the shortest-path distance between the endpoints of an edge, upon its deletion, can depend on \emph{all} remaining edges in the graph.
Therefore, the bound of Theorem~\ref{thm:edge_deletion_stability} is non-local and indeed no generic, local stability theorem is possible.

\begin{theorem}\label{thm:edge_deletion_instability}
There exists no function $f:\WDgr \to \R$ such that for any digraph $G=(V, E, w)\in\WDgr$ and any edge $e\in E$ therein we have
\begin{equation}
  d_B\big( \thedesc{G} , \thedesc{\wdop{d}{e}{G}}\big)
  \leq f\big( \nbhdgraph(e; G) \big).
\end{equation}
\end{theorem}
\begin{proof}
Suppose such $f$ exists  then 
consider the following weighted digraph $G$,
where $e\defeq(v_1, v_3)$ and $W\defeq 2f(\nbhdgraph(e; G))+2$.
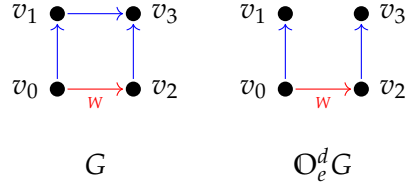
\begin{figure}[H]
  \centering
  \begin{tikzpicture}[
  roundnode/.style={circle, fill=black, minimum size=4pt},
	squarenode/.style={fill=black, minimum size=4pt},
	inner sep=2pt,
	outer sep=1pt
  ]
  \node (a) at (0, 0) [roundnode, label=left:$v_0$] {};
  \node (b) at (0, 1) [roundnode, label=left:$v_1$] {};
  \node (c) at (1, 1) [roundnode, label=right:$v_3$] {};
  \node (d) at (1, 0) [roundnode, label=right:$v_2$] {};

  \node (a2) at (3, 0) [roundnode, label=left:$v_0$] {};
  \node (b2) at (3, 1) [roundnode, label=left:$v_1$] {};
  \node (c2) at (4, 1) [roundnode, label=right:$v_3$] {};
  \node (d2) at (4, 0) [roundnode, label=right:$v_2$] {};

  \draw[->, blue] (a)--(b);
  \draw[->, blue] (b)--(c);
  \draw[->, blue] (d)--(c);
  \draw[->, red] (a)--(d) node[midway, sloped, below] {\tiny $W$};

  \draw[->, blue] (a2)--(b2);
  \draw[->, blue] (d2)--(c2);
  \draw[->, red] (a2)--(d2) node[midway, sloped, below] {\tiny $W$};



  \node[] at (0.5, -1) {$G$};
  \node[] at (3.5, -1) {$\wdop{d}{e}{G}$};
\end{tikzpicture}
  \caption{An example weighted digraph which illustrates that $\zbpershommap$ is locally unstable to arbitrary edge deletions.
  Unlabelled edges have weight $1$.}%
  \label{fig:edge_deletion_instability}
\end{figure}\noindent
Note, $G$ has a single feature which dies at time $W$, whereas $\wdop{d}{e}{G}$ has no features.
Therefore, the bottleneck distance is $W/2 > f(\nbhdgraph(e; G))$.
\end{proof}

\subsubsection{Separating edges}

Since all features are born at $t=0$ and $\zbcycles{G}{0}$ has a basis of simple undirected circuits, one might expect that edges never involved in such circuits can be safely deleted without changing the descriptor.
Indeed, this is the case and is a direct consequence of the wedge decomposition theorem.

\begin{defin}
  In a weighted digraph $G=(V, E, w)$, an edge $e=(a, b)\in E$ is called a \mdf{separating edge} if $a$ and $b$ are weakly disconnected in $\wdop{d}{e}{G}$.
\end{defin}

\begin{rem}
An edge is separating if and only if there are no simple undirected circuits containing it.
\end{rem}

\begin{cor}\label{cor:edge_deletion_disconnect}
Given a weighted digraph $G=(V, E, w)$ and a separating edge $e=(a, b)\in E$,
\begin{equation}
  \zbpershom{G} \cong \zbpershom{\wdop{d}{e}{G}}.
\end{equation}
\end{cor}
\begin{proof}
First note that $a$ and $b$ are both wedge vertices.
Let $V_1$ denote the vertices in the weak connected component of $a$ in $\wdop{d}{e}{G}$.
Define $V_2 \defeq \{ a, b \}$.
Finally, define $V_3 \defeq (V(G)\setminus V_1)$.
Let $G_1$, $G_2$ and $G_3$ denote the induced subgraph of $G$ on $V_1$, $V_2$ and $V_3$ respectively.

Then a wedge decomposition of $G$ is
$G = (G_1 \vee_a G_2) \vee_b G_3$
and a disjoint union decomposition of $\wdop{d}{e}{G}$ is
$\wdop{d}{e}{G} = G_1 \sqcup G_3$.
Note that $G_2$ is just a single edge connecting two vertices so $\zbpershom{G_2}$ is the trivial persistent vector space.
Using Theorem~\ref{thm:union_decomp} and Theorem~\ref{thm:wedge_decomp}, we see
\begin{equation}
 \zbpershom{G}
 \cong \zbpershom{G_1} \oplus \zbpershom{G_2} \oplus \zbpershom{G_3}
 \cong \zbpershom{G_1} \oplus \zbpershom{G_3}
 \cong \zbpershom{\wdop{d}{e}{G}}
\end{equation}
as required.
\end{proof}

\subsubsection{Interpretation}

\begin{figure}[ht]
  \centering
  \begin{tikzpicture}[
  roundnode/.style={circle, fill=black, minimum size=4pt},
	squarenode/.style={fill=black, minimum size=4pt},
	inner sep=2pt,
	outer sep=1pt
  ]

  \node (a) at (0, 0) [roundnode, label=left:$v_0$] {};
	\node (b) at (0.75, 1) [roundnode, label=above:$v_2$] {};
	\node (c) at (0.75, -1) [roundnode] {};
	\node (d) at (2.25, 1) [roundnode, label=above:$v_3$] {};
	\node (e) at (2.25, -1) [roundnode] {};
	\node (f) at (1.5, 0) [roundnode] {};
	\node (g) at (3, 0) [roundnode, label=right:$v_1$] {};

  \draw[->] (a)--(b);
  \draw[->] (a)--(c);
  \draw[->] (c)--(e);
  \draw[->] (b)--(d);
  \draw[->] (f)--(d);
  \draw[->] (f)--(e);
  \draw[->] (a)--(f);
  \draw[->] (f)--(g);
  \draw[->] (d)--(g);
  \draw[->] (e)--(g);
  \draw[->] (b)--(f);
  \draw[->] (c)--(f);
  \draw[->, red, rounded corners] (a)--(0.25, 1.75)--(2.75, 1.75) node[midway, above, sloped] {\tiny $5$} -- (g) ;

  \node(a2) at (5, 0) [roundnode, label=left:$w_0$] {};
  \node(b2) at (8, 0) [roundnode, label=right:$w_1$] {};
  \node(c2) at (5.5, -0.4) [roundnode] {};
  \node(d2) at (6.15, -0.6) [roundnode, label=below:$w_2$] {};
  \node(e2) at (6.85, -0.6) [roundnode, label=below:$w_3$] {};
  \node(f2) at (7.5, -0.4) [roundnode] {};

  \draw[->, red] (5.1, 0.1) arc (130:50:2.18) node[midway, above, sloped] {\tiny $10$};
  \draw[->, red!20!blue] (a2)--(c2) node[midway, above, sloped] {\tiny $2$};
  \draw[->, red!20!blue] (c2)--(d2) node[midway, above, sloped] {\tiny $2$};
  \draw[->, red!20!blue] (d2)--(e2) node[midway, above, sloped] {\tiny $2$};
  \draw[->, red!20!blue] (e2)--(f2) node[midway, above, sloped] {\tiny $2$};
  \draw[->, red!20!blue] (f2)--(b2) node[midway, above, sloped] {\tiny $2$};

  \node[] at (1.5, -1.5) {$G_1$};
  \node[] at (6.5, -1.5) {$G_2$};
\end{tikzpicture}
  \caption{Some example weighted digraphs, used to interpret the consequences of the stability theorems obtained in Section~\ref{sec:stability_edge_deletion}. Unlabelled edges have weight $1$.}
  \label{fig:edge_deletion_remark}
\end{figure}
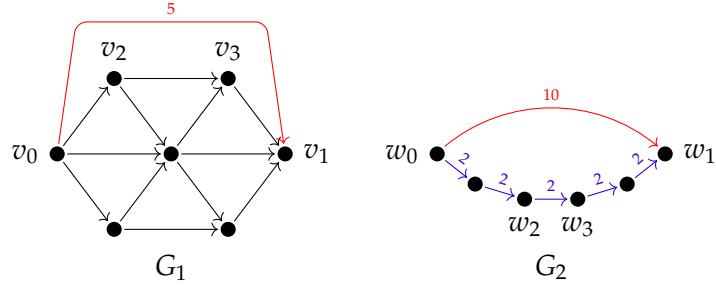

Theorem~\ref{thm:edge_deletion_stability} tells us that we are stable to deleting edges which have fast diversions.
That is, if there is a path $p:i\leadsto j$, not involving the edge $e\defeq(i, j)$, of length $\delta$, then removing $e$ changes the barcode by at most $\delta$ in bottleneck distance.
Note, this bound is independent of the weight of the deleted edge $w(e)$.

To illustrate this point, consider $G_1$ in Figure~\ref{fig:edge_deletion_remark}.
Removing $(v_2, v_3)$ incurs a bottleneck cost of at most $2$, since there is a diversion of length $2$.
Likewise, despite being a highly-weighted edge, we can also remove $(v_0, v_1)$ for a bottleneck cost of at most $2$.

On the other hand, consider now $G_2$ in Figure~\ref{fig:edge_deletion_remark}. 
The edge $(w_0, w_1)$ has a high weight and the only diversion is via the black edges, of length $10$.
Deleting the edge $(w_0, w_1)$ incurs a bottleneck cost of $5$ since it removes the sole feature.
Moreover, deleting one of the smaller edges (for example $(w_2, w_3)$) also incurs a bottleneck cost of $5$ since it deletes the same feature.

\subsection{Vertex deletion}
\begin{defin}
Given a weighted digraph $G=(V, E, w)$ and a vertex $v\in V$, we define
$\mdf{\wdop{d}{v}{G}}\defeq (V\setminus \{v\}, E_v, w_v)$ where $E_v \defeq E \cap (V\setminus \{v\})\times (V\setminus \{v\})$ and $w_v$ is obtained by restricting $w$ to $E_v$.
\end{defin}

Since a single vertex graph has trivial \grpph\, the disjoint union decomposition theorem (Theorem~\ref{thm:union_decomp}) allows us to delete isolated vertices.

\begin{cor}\label{cor:isolated_vertex_deletion}
Given a weighted digraph $G=(V, E, w)$ and an isolated vertex $v_i\in V$ (i.e.~$\nbhd(v)=\emptyset$), then
\begin{equation}
  \thedesc{G} \cong \thedesc{\wdop{d}{v_i}{G}}.
\end{equation}
\end{cor}

However, in general, deleting a vertex from a digraph can drastically change its topology.
This follows immediately from Theorem~\ref{thm:subdiv_stability} and Theorem~\ref{thm:edge_deletion_instability} since a local vertex deletion stability theorem would imply a local edge deletion stability theorem.

\begin{cor}\label{cor:vertex_deletion_instability}
There exists no function $f:\WDgr \to \R$ such that for any digraph $G=(V, E, w)$ and any vertex $v\in V$ therein we have
\begin{equation}
  d_B\big( \thedesc{G} , \thedesc{\wdop{d}{v}{G}}\big)
  \leq f\big( \nbhdgraph(v; G) \big).
\end{equation}
\end{cor}

\section{Examples}\label{sec:examples}

\subsection{Iterated medial subdivision}\label{sec:asymptotic}

In order to develop intuition for how the descriptor behaves under iterative subdivision, we explicitly derive the limiting diagram for a DAG with exactly one loop (shown in Figure~\ref{fig:two_paths}).
Certainly, the diagram contains exactly one feature which is born at $t=0$.
Intuitively, the death time corresponds to the earliest time that a long square can appear between the source and sink nodes, filling in the central hole.

\begin{figure}[htbp]
  \centering
  \begin{tikzpicture}[
	roundnode/.style={circle, fill=black, minimum size=4pt},
	roundnodered/.style={circle, fill=red, minimum size=1pt},
	inner sep=2pt,
	outer sep=1pt
	]
	\node (a) at (0, 0) [roundnode, label=left:$a$] {};
	\node (b) at (4, 0) [roundnode, label=right:$b$] {};
  \node (mid1) at (2, 0.5) [roundnodered] {};
  \node (mid2) at (2, -0.5) [roundnodered] {};

  \draw[->, rounded corners] (a) -- (0.5, 1.2) -- (1.5, 0.9) -- (mid1) -- (2.5, 1.1) -- (3.5, 0.8) -- (b);
  \draw[->, rounded corners] (a) -- (0.3, -0.9) -- (0.7, -0.4) -- (1.5, -0.9) -- (mid2) -- (2.9, -1.0) -- (3.2, -0.9) -- (b);

  \draw[->, dashed, red] (a) -- (mid1);
  \draw[->, dashed, red] (a) -- (mid2);
  \draw[->, dashed, red] (mid1) -- (b);
  \draw[->, dashed, red] (mid2) -- (b);

  \node[] at (2, 1) {$l_1$};
  \node[] at (2, -1) {$l_2$};
\end{tikzpicture}
  \caption{Illustration of the weighted digraph $G$, considered in Proposition~\ref{prop:asymptotic_example}.
    The top path $p_1$ has length $l_1$ and the bottom path $p_2$ has length $l_2$.
    The limiting death time of the sole feature corresponds to the limiting value of the earliest time that a long square appears of the form $av_1 b - av_2 b$ where $v_1$ is along $p_1$ and $v_2$ is along $p_2$ (as drawn in red) dashed lines).
  }
  \label{fig:two_paths}
\end{figure}
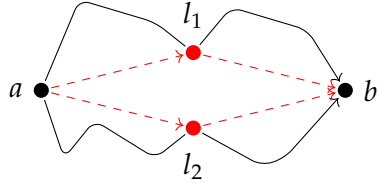

\begin{prop}\label{prop:asymptotic_example}
  Suppose $G\in\WDag$ is the union two directed paths $p_1, p_2$ from a source to a sink, with lengths $l_1 \geq l_2$ respectively.
  Recall the definition of iterated medial subdivision (Definition~\ref{defin:ims}).
  Then
  \begin{equation}
    \lim_{n\to\infty}\thedesc{\IMS{G}{n}}
    =
    \ms{
      \left[0, \frac{1}{2}l_1 \right)
    }
  \end{equation}
\end{prop}
\begin{proof}
For brevity we denote $G_n \defeq \IMS{G}{n}$.
By Corollary~\ref{cor:decreasing_curves}, the barcode $\thedesc{G_n}$ has exactly one feature.
Let $p_i^{(n)}$ denote the path $a\leadsto b$ in $G_n$ arising from subdividing the edges of $p_i$.
For each $n$, let $c^{(n)}$ denote the simple undirected circuit in $G_n$ which follows $p_1^{(n)}$ and then $p_2^{(n)}$ in reverse.
Clearly $\{\repres{c^{(n)}}\}$ is a persistence basis for $\zbpershom{G_n}$.
Therefore, it suffices to show $\death(\repres{c^{(n)}}) \to \frac{1}{2}l_1$ as $n\to\infty$.
Fix some natural $n$.

Using Lemma~\ref{lem:death_bound} we see $\death(\repres{c^{(n)}}) \leq \max(h_{1}^{(n)}, h_{2}^{(n)})$ where
\begin{equation}
 h_i^{(n)} \defeq \min
 \left\{ 
   t \geq 0 \rmv
   \exists v_i \in V\text{ along } p_i^{(n)} \text{ such that } d(a, v_i) \leq t \text{ and } d(v_i, b) \leq t
 \right\}.
\end{equation}
Note that $h_1^{(n)} \to \frac{1}{2}l_1$ and $h_2^{(n)} \to \frac{1}{2}l_2$ and hence $\max(h_1^{(n)}, h_2^{(n)}) \to \frac{1}{2}l_1$ as $n\to\infty$.

Next, we wish to show $\death(\repres{c^{(n)}}) \geq \frac{1}{2}l_1$.
Choose arbitrary $t_2 < \frac{1}{2} l_1$, then it suffices to show that $\dim\zbhom{G_n}{t_2}\neq 0$.
In order to do so, we claim the inclusion chain map
\begin{figure}[H]
  \centering
  \begin{tikzcd}
    0 \arrow[r] \arrow[d, "j_2"]
    & C_1(G_n) \arrow[r, "\zb{\bd}_1"] \arrow[d, "j_1"]
    & C_0(G_n) \arrow[d, "j_0"] \\
    C_2(G_n^{t_2}) \arrow[r, "\zb{\bd}_2"]
    & C_1(G_n \cup G_n^{t_2}) \arrow[r, "\zb{\bd}_1"]
    & C_0(G_n \cup G_n^{t_2}) \\
  \end{tikzcd}
\end{figure}\noindent
induces an isomorphism on homology in degree $1$.
It then follows that $\dim\zbhom{G_n}{t_2}=1$ because the first homology of top row is the real cycle space of $G_n$.

We define a chain map $\inducedch{q}$ in the opposite direction to $q$.
In degree $2$, $q_2$ is the zero map and in degree $0$, $q_0$ is the identity map.
Finally in degree $1$, given $(i, j)\in C_1(G_n \cup G_n^{t_2})$, if $(i, j) = (a, b)$ then let $p_{i, j}\defeq p_2$, otherwise let $p_{i, j}$ denote the unique path $i\leadsto j$ in $G_n$
Then $q_1$ is given by $q_1(ij) \defeq \repres{p_{i,j}}$.
This is a chain map because there is no $2$-path $avb$ where $v$ is somewhere along $p_1^{(n)}$.

It is certainly the case that, at the level of chain maps, $q_1 j_1 = \id$.
Choose arbitrary $(i, j) \in C_1(G_n \cup G_n^{t_2})$ and note that
$j_1 q_1 (ij) - (ij) = \repres{p_{i, j}} - ij$.
By Lemma~\ref{lem:paths_homologous}, there is some $u_{i, j} \in C_2(G_n^{t_2})$ such that $\zb{\bd}_2 u_{i, j} =  \repres{p_{i, j}} - ij$.
Define $P:C_1(G_n \cup G_n^{t_2}) \to C_2(G_n^{t_2})$ by $ij \mapsto u_{i, j}$.
Then, by construction, we see $j_1 q_1 - \id = \zb{\bd}_2 P$.
Hence, $j$ and $q$ are mutually inverse on homology in degree $1$.

To conclude, we have shown for each $n$,
\begin{equation}
  \frac{1}{2}l_1 \leq \death(\repres{c^{(n)}}) \leq \max( h_1^{(n)}, h_2^{(n)} ).
\end{equation}
Taking the limit $n\to\infty$ finishes the proof.
\end{proof}

Note that description of Proposition~\ref{prop:asymptotic_example} is not unique to this descriptor, indeed the same result holds for the standard pipeline.
As discussed in Section~\ref{sec:intro_subdivision}, for the standard pipeline, as the weighted digraph is subdivided, the birth times of all features tend to $0$.
When the digraph is sufficiently subdivided, all edges enter the filtration very early on and the effect of adding the edges from $G$ at $t=0$ has negligible effect.
Hence, in the subdivision limit, the diagrams obtained from the two pipelines coincide.

\begin{theorem}
  Given $G\in\WDgr$, let $\mathcal{H}_1$ denote the `standard pipeline' with $C=\Omega$ and $F=F_d$, as used in Example~\ref{ex:std_pipe}. Then
  \begin{equation}
    \lim_{n\to\infty}\thedesc{\IMS{G}{n}}=
    \lim_{n\to\infty}\Barcode(
    \mathcal{H}_1(\IMS{G}{n})
    ).
  \end{equation}
\end{theorem}
\begin{proof}
For brevity, we denote $G_n \defeq \IMS{G}{n}$.
Fix $\epsilon >0$ and choose $N$ sufficiently large that for any $n\geq N$ we have $w(e) < \epsilon$ for all $e\in E(G_n)$.
For any $t\geq 0$, define
\begin{align*}
  i_1 &: C_1(G_n^t)\to C_1(G_n \cup G_n^{t+\epsilon}) \\
  i_2 &: C_2(G_n^t)\to C_2(G_n^{t+\epsilon})
\end{align*}
where each $i_k$ is taken from the chain map induced by the relevant inclusion of digraphs.
It can be easily checked that $i_1 \bd_2 = \zb{\bd}_2 i_2$ and hence $i_1$ induces a map on homology $i_\ast : H_1(G_n^t) \to \zb{H}_1(G_n,{t+\epsilon})$.
Similarly, for any $t\geq 0$ define
\begin{align*}
  j_1 &: C_1(G_n \cup G_n^t)\to C_1(G_n^{t+\epsilon}) \\
  j_2 &: C_2(G_n^t)\to C_2(G_n^{t+\epsilon})
\end{align*}
where each $j_k$ is likewise taken from the chain map induced by the relevant inclusion of digraphs.
Note, in particular, given an edge $e\in E(G_n)$, we know $d(\st(e), \fn(e)) < \epsilon$ and hence $e \in E(G_n^{t+\epsilon})$.
Again $j_1 \zb{\bd}_2 = \bd_2 j_2$ and hence $j_1$ induces a map on homology $j_\ast : \zb{H}_1(G,t) \to H_1(G^{t+\epsilon})$.

Clearly $i_\ast \circ j_\ast = \zbinclinduc{hom}{t}{t+2\epsilon}$ and $j_\ast \circ i_\ast = \inducedhom{\filtincl{t}{t+2\epsilon}}$.
Therefore, by the algebraic stability theorem, we see
\begin{equation}
    d_B \left(
    \thedesc{G_n},
    \Dgm(\mathcal{H}_1(G_n))
  \right) \leq \epsilon
\end{equation}
for all $n\geq N$.
\end{proof}

\subsection{Square motifs}\label{sec:square_motifs}

\begin{figure}
  \centering
  \begin{tikzpicture}[
  roundnode/.style={circle, fill=black, minimum size=4pt},
	squarenode/.style={fill=black, minimum size=4pt},
	inner sep=2pt,
	outer sep=1pt
  ]

  \node (a) at (0, 0) [roundnode, label=left:$a$] {};
	\node (b) at (1, 1) [roundnode, label=above:$b$] {};
	\node (c) at (1, -1) [roundnode, label=below:$c$] {};
	\node (d) at (2, 0) [roundnode, label=right:$d$] {};
  \draw[->] (a)--(b);
  \draw[->] (a)--(c);
  \draw[->] (b)--(d);
  \draw[->] (c)--(d);

  \node (a2) at (3, 0) [roundnode, label=left:$a$] {};
	\node (b2) at (4, 1) [roundnode, label=above:$b$] {};
	\node (c2) at (4, -1) [roundnode, label=below:$c$] {};
	\node (d2) at (5, 0) [roundnode, label=right:$d$] {};
  \draw[->] (a2)--(b2);
  \draw[->] (a2)--(c2);
  \draw[->] (d2)--(b2);
  \draw[->] (d2)--(c2);

  \node (a3) at (6, 0) [roundnode, label=left:$a$] {};
	\node (b3) at (7, 1) [roundnode, label=above:$b$] {};
	\node (c3) at (7, -1) [roundnode, label=below:$c$] {};
	\node (d3) at (8, 0) [roundnode, label=right:$d$] {};
  \node (e3) at (9,0) [roundnode, label=right:$e$] {};
  \draw[->] (a3)--(b3);
  \draw[->] (a3)--(c3);
  \draw[->] (d3)--(b3);
  \draw[->] (d3)--(c3);
  \draw[->] (b3) edge [bend left] (e3);
  \draw[->] (c3) edge [bend right] (e3);

  \node (a4) at (10, 0) [roundnode, label=left:$a$] {};
	\node (b4) at (11, 1) [roundnode, label=above:$b$] {};
	\node (c4) at (11, -1) [roundnode, label=below:$c$] {};
	\node (d4) at (12, 0) [roundnode, label=right:$d$] {};
  \node (e4) at (13,0) [roundnode, label=right:$e$] {};
  \draw[->] (a4)--(b4);
  \draw[->] (a4)--(c4);
  \draw[->] (d4)--(b4);
  \draw[->] (d4)--(c4);
  \draw[->] (e4) edge [bend right] (b4);
  \draw[->] (e4) edge [bend left] (c4);

  \node[] at (1, 2) {$G_1$};
  \node[] at (4, 2) {$G_2$};
  \node[] at (7.5, 2) {$G_3$};
  \node[] at (11.5, 2) {$G_4$};

  \draw[dotted] (2.5, 2.5) -- (2.5, -2.5);
  \draw[dotted] (5.5, 2.5) -- (5.5, -2.5);
  \draw[dotted] (9.5, 2.5) -- (9.5, -2.5);

  \draw[dotted] (-0.5, 1.7) -- (13.5, 1.7);
  \draw[dotted] (-0.5, -1.6) -- (13.5, -1.6);

  \node[] at (1, -2) {$\ms{[0, 1)}$};
  \node[] at (4, -2) {$\ms{[0, \infty)}$};
  \node[] at (7.5, -2) {$\ms{[0, 1), [0, 1)}$};
  \node[] at (11.5, -2) {$\ms{[0,\infty), [0, \infty)}$};
\end{tikzpicture}
  \caption{Interpreting $\zbpershom{G}$ via differences between directed paths in square motifs.
  Top row: name of weighted digraph; middle row: diagram where all weights are $1$; bottom row: barcode of \grpph. }
  \label{fig:two_squares}
\end{figure}
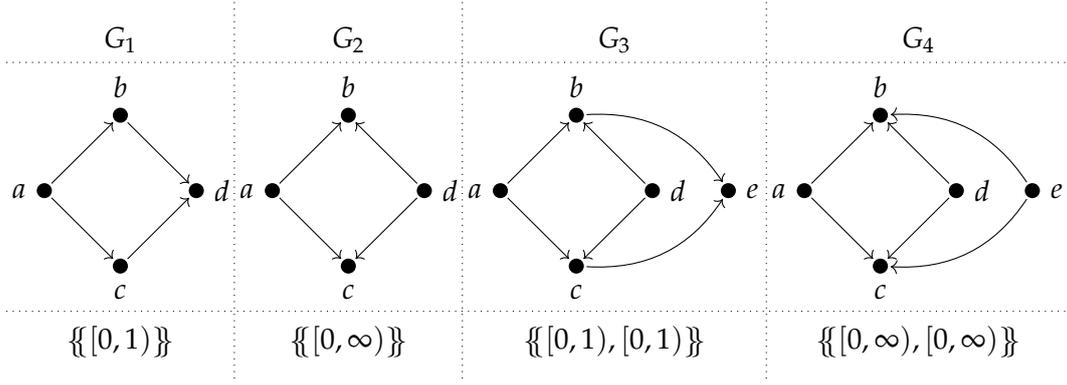

\begin{example}\label{ex:two_squares}
  Further to the interpretation developed in Proposition~\ref{prop:asymptotic_example}, consider the four weighted digraphs in Figure~\ref{fig:two_squares}.
All edges are given unit weight and the barcodes are indicated under each digraph.
Homology representatives for each of the features are given by
\begin{align*}
  &ab + bd - cd - ac; \\
  &ab - db + dc - ac; \\
  &ab + be - ce - ac \quad , \quad db + be - ce - dc; \\
  &ab - eb + ec - ac \quad , \quad db - eb + ec - dc. 
\end{align*}

In $G_1$, note that the flow starting at $a$ recombines at $d$ after flowing for $t=2$ seconds.
In contrast, the flow in $G_2$ splits from the sources and then never recombines.
This is reflected in the lifetime of the feature changing from $[0, 1)$ to $[0, \infty)$.

If we add additional edges to $G_2$ to recombine the flow at a new vertex (as in $G_3$), we add an additional feature but all features now have finite lifetime.
Finally, reversing these additional edges (as in $G_4$) prevents the flow from recombining again and the features return to lifetime $[0, \infty)$.

This further emphasises the interpretation that features arise when flow is split between two paths and the lifetime of the feature is related to the time it takes for the flow to recombine.
\end{example}

\subsection{Multiple paths}

\begin{figure}
  \centering
%
\begin{tikzpicture}[
	roundnode/.style={circle, fill=black, minimum size=4pt},
	roundnodered/.style={circle, fill=red, minimum size=1pt},
	inner sep=2pt,
	outer sep=1pt
	]
	\node (a) at (-3, 0) [roundnode, label=left:$a$] {};
	\node (b) at (3, 0) [roundnode, label=right:$b$] {};
  \node (mid1) at (0, 1.0) [roundnode, label=above:$v_1$] {};
  \node (mid2) at (0, 0.3) [roundnode, label=above:$v_2$] {};
  \node (mid3) at (0, -0.3) [roundnode, label=below:$v_{n-1}$] {};
  \node (mid4) at (0, -1.0) [roundnode, label=below:$v_n$] {};

  \draw[->] (a) to[bend left] node[anchor=south, black] {$a_1$} (mid1);
  \draw[->] (a) to node[anchor=south, black] {$a_2$} (mid2);
  \draw[->] (a) to node[anchor=north, black] {$a_{n-1}$} (mid3);
  \draw[->] (a) to[bend right] node[anchor=north, black] {$a_n$} (mid4);
  \draw[->] (mid1) to[bend left] node[anchor=south, black] {$b_1$} (b);
  \draw[->] (mid2) to node[anchor=south, black] {$b_2$} (b);
  \draw[->] (mid3) to node[anchor=north, black] {$b_{n-1}$} (b);
  \draw[->] (mid4) to[bend right] node[anchor=north, black] {$b_n$} (b);

  \node[] at (0, 0.1) {\small$\vdots$};
\end{tikzpicture}
    \caption{A weighted digraph with many paths (of 2 edges each) from source to sink for which we can compute $\thedesc{G}$.}
    \label{fig:redundancy}
\end{figure}
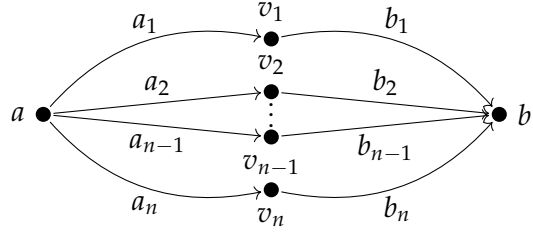

\begin{example}\label{ex:redundancy}
  Consider Figure~\ref{fig:redundancy}, in which there is a single source and a single sink but multiple paths between.
  Define
    $\alpha_i \defeq \max (a_i, b_i)$
  and assume that 
  $
    \alpha_1 \leq \alpha_2 \leq \dots \alpha_{n-1} \leq \alpha_n.
  $
  Then, the barcode is
  \begin{equation}
    \thedesc{G} = \ms{ [0, \alpha_2), [0, \alpha_3), \dots, [0, \alpha_{n-1}), [0, \alpha_{n}) }.
  \end{equation}
  A persistence basis for $\zbpershom{G}$ is $\{ c_2, \dots, c_n \}$ where
  $
    c_i \defeq {a v_1} + {v_1 b} - {v_i b} - {a v_i}
  $
  and $\death(c_i) = \alpha_i$.
\end{example}

\subsection{Identical quasimetric}

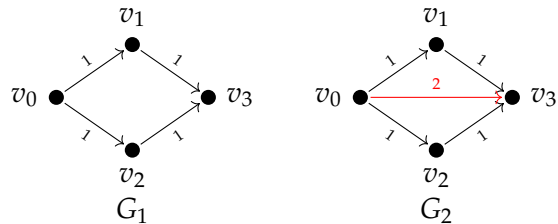
\begin{figure}[hptb]
  \centering
  \begin{tikzpicture}[
  roundnode/.style={circle, fill=black, minimum size=4pt},
	squarenode/.style={fill=black, minimum size=4pt},
	inner sep=2pt,
	outer sep=1pt
  ]
  \node (a) at (0, 0) [roundnode, label=left:$v_0$] {};
	\node (b) at (1, 0.7) [roundnode, label=above:$v_1$] {};
	\node (c) at (1, -0.7) [roundnode, label=below:$v_2$] {};
	\node (d) at (2, 0) [roundnode, label=right:$v_3$] {};

  \draw[->] (a)--(b) node[midway, above, sloped] {\tiny $1$};
  \draw[->] (a)--(c) node[midway, below, sloped] {\tiny $1$};
  \draw[->] (b)--(d) node[midway, above, sloped] {\tiny $1$};
  \draw[->] (c)--(d) node[midway, below, sloped] {\tiny $1$};

  \node (a2) at (4, 0) [roundnode, label=left:$v_0$] {};
	\node (b2) at (5, 0.7) [roundnode, label=above:$v_1$] {};
	\node (c2) at (5, -0.7) [roundnode, label=below:$v_2$] {};
	\node (d2) at (6, 0) [roundnode, label=right:$v_3$] {};

  \draw[->] (a2)--(b2) node[midway, above, sloped] {\tiny $1$};
  \draw[->] (a2)--(c2) node[midway, below, sloped] {\tiny $1$};
  \draw[->] (b2)--(d2) node[midway, above, sloped] {\tiny $1$};
  \draw[->] (c2)--(d2) node[midway, below, sloped] {\tiny $1$};
  \draw[->, red] (a2)--(d2) node[midway, above, sloped] {\tiny $2$};

  \node (g1) at (1, -1.5) {$G_1$};
  \node (g2) at (5, -1.5) {$G_2$};

\end{tikzpicture}
  \caption{Two weighted digraphs with identical shortest-path quasimetric network but differing \grpph\ which can be explained by the difference in circuit rank.}\label{fig:standard_pipeline}
\end{figure}

\begin{example}
Finally, consider the two weighted digraphs illustrated in Figure~\ref{fig:standard_pipeline}.
Since they both have the same shortest-path quasimetric, they yield the same barcode under the standard pipeline.
More formally, $F_d(G_1) = F_d(G_2)$ and hence 
$\zbpershom{G_1} = \zbpershom{G_2}$.
Moreover, $\thedesc{G_1}$ is empty because the circuit $(v_0, v_1, v_3, v_2)$ is filled-in with a long square as soon as it appears in the filtration.
In contrast,
\begin{equation}
  \thedesc{G_1}  = \ms{ [0, 1)} \quad\text{ and }\quad
  \thedesc{G_2}  = \ms{ [0, 1), [0, 2) }.
\end{equation}
A persistence basis for $\zbpershom{G_1}$ is
$
  \{v_0 v_1 + v_1 v_3 - ( v_0 v_2 + v_2 v_3 ) \}
$
while a persistence basis for $\zbpershom{G_2}$ is
$\{
  v_0 v_1 + v_1 v_3 - (v_0 v_2 + v_2 v_3)
  \quad , \quad
  v_0 v_1 + v_1 v_3 - v_0 v_3 .
  \}$
Note that at $t=1$ the two triangular cycles becomes homologous in $\zbpershom{G_2}$ but are still non-trivial, until they die at $t=2$.
\end{example}

\appendix
\section{Grounded pipeline with the directed flag complex}%
\label{appdx:complex}
We will now consider the pipeline developed in Section~\ref{sec:descriptor_defin}, changing the choice of chain complex, $C$, to the directed flag complex.

\begin{defin} 
  Given a digraph $G=(V, E)$,
  a \mdf{directed $n$-clique} is a $(n+1)$-tuple of distinct vertices $v_0 \dots v_n$ such that $i < j \implies v_i \to v_j$.
\end{defin}

\begin{defin}
  The \mdf{directed flag complex, $\dFl(G)$}, of a digraph $G\in\Dgr$ is the chain complex
  \begin{figure}[H]
    \centering
    \begin{tikzcd}
      \cdots \arrow[r, "\bd_3"]
      & \dFl_2(G) \arrow[r, "\bd_2"]
      & \dFl_1(G) \arrow[r, "\bd_1"]
      & \dFl_0(G) \arrow[r, "\bd_0"]
      & \Ring \arrow[r] & 0
    \end{tikzcd}
  \end{figure}\noindent
  where $\dFl_k(G)$ is the $\R$-vector space freely generated by the $(k+1)$-cliques in $G$.
  The boundary map $\bd_k$ is defined on the basis of cliques by
  \begin{equation}
    \bd_k (v_0 \dots v_k) \defeq \sum_{i=0}^{k} (-1)^i v_0 \dots \hat{v}_i \dots v_k
  \end{equation}
  where $v_0 \dots \hat{v}_i \dots v_k$ is the $k$-clique obtained from $v_0 \dots v_k$ by removing the vertex $v_i$.
\end{defin}

\begin{figure}[htbp]
  \centering
  \begin{tikzpicture}[
  roundnode/.style={circle, fill=black, minimum size=4pt},
	squarenode/.style={fill=black, minimum size=4pt},
	inner sep=2pt,
	outer sep=1pt
  ]

  \node (a) at (0, 0) [roundnode, label=left:$v_0$] {};
  \node (b) at (0.6, 1) [roundnode, label=left:$v_1$] {};
  \node (c) at (2.4, 1) [roundnode, label=right:$v_2$] {};
  \node (d) at (3, 0) [roundnode, label=right:$v_3$] {};

  \draw[->] (a) -- (b);
  \draw[->] (a) -- (c);
  \draw[->] (a) -- (d);
  \draw[->] (b) -- (c);
  \draw[->] (b) -- (d);
  \draw[->] (c) -- (d);
\end{tikzpicture}
  \caption{A directed $4$-clique, giving rise to a generator $v_0 v_1 v_2 v_3 \in F_3$.}
  \label{fig:directed_clique}
\end{figure}
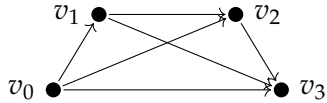

The directed flag complex is an alternative chain complex for digraphs, $\dFl:\Obj(\Dgr)\to\Obj(\Ch)$, which has seen more use in applications than path homology (see e.g.~\cite{reimann2017cliques, Luetgehetmann2020}).
We now repeat the investigation conducted in the main body of the paper, replacing path homology with the directed flag complex.
\textbf{Henceforth, for the rest of the paper, we fix $F$ to be the shortest-path filtration (see Definition~\ref{defin:shortest_path}) and $C$ to be the directed flag complex},
\begin{equation}
 F = F_d
 \quad\text{ and }\quad
 C=\dFl.
\end{equation}

Recalling that morphisms in $\Dgr$ can collapse edges, we show that $\dFl$ cannot constitute a functor $\Dgr\to\Ch$.
However, we do find a smaller category $\Tricol\Dgr\subseteq\Dgr$, containing all digraphs, upon which $\dFl$ is functorial.
This smaller category contains inclusions and thus 
we can apply the machinery developed in Section~\ref{sec:descriptor_defin} to obtain grounded persistent directed flag complex homology (\grpdflh), $\zbpershomdflagmap:\Obj(\WDgr)\to\Obj(\PersVec)$.

For the remainder of the appendix, we review the results obtained for \grpph, in the main text and check which results apply to $\zbpershomdflagmap$.
Reassuringly, the results on undirected circuit representatives and weight perturbation stability hold in this setting.
However, some stability results fail, most notably edge subdivision, since they rely on morphisms outside of $\Tricol\Dgr$.
We summarise the known stability results for the directed flag complex in Table~\ref{tbl:dflag_stability_result_summary}.

Moreover, we find that the failure in functoriality $\Dgr\to\Ch$ leads to a failure in the wedge decomposition theorem (Theorem~\ref{thm:wedge_decomp}).
In Example~\ref{ex:dflag_appendage_instability}, we exhibit an explicit example wherein deleting a small appendage edge can dramatically alter the barcode, $\thedescdfl{G}$.
This behaviour is not present in \grpph; we argue this is an instability which complicates the interpretation of $\thedescdfl{G}$.

\subsection{Functoriality of the directed flag complex}

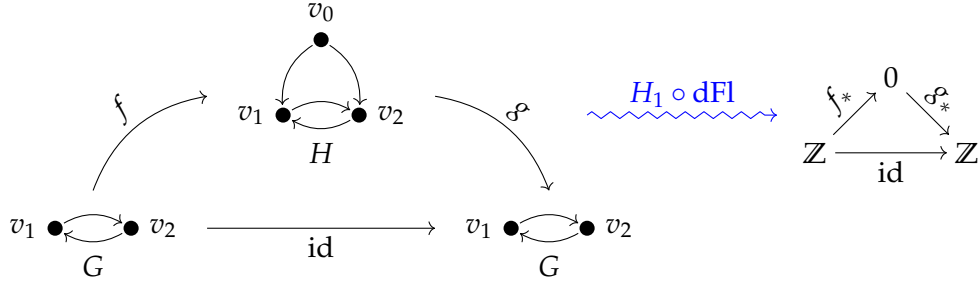
\begin{figure}[htbp]
  \centering
  \begin{tikzpicture}[
  roundnode/.style={circle, fill=black, minimum size=4pt},
	squarenode/.style={fill=black, minimum size=4pt},
	inner sep=2pt,
	outer sep=1pt
  ]

  \node (a) at (0, 0) [roundnode, label=left:$v_1$] {};
  \node (b) at (1, 0) [roundnode, label=right:$v_2$] {};
  \node at (0.5, -0.5) {$G$};

  \node (c) at (3, 1.5) [roundnode, label=left:$v_1$] {};
  \node (d) at (4, 1.5) [roundnode, label=right:$v_2$] {};
  \node (e) at (3.5, 2.5) [roundnode, label=above:$v_0$] {};
  \node at (3.5, 1) {$H$};

  \node (f) at (6, 0) [roundnode, label=left:$v_1$] {};
  \node (g) at (7, 0) [roundnode, label=right:$v_2$] {};
  \node at (6.5, -0.5) {$G$};

  \draw[->, bend left] (a) to (b);
  \draw[->, bend left] (b) to (a);

  \draw[->, bend left] (c) to (d);
  \draw[->, bend left] (d) to (c);
  \draw[->, bend right] (e) to (c);
  \draw[->, bend left] (e) to (d);

  \draw[->, bend left] (f) to (g);
  \draw[->, bend left] (g) to (f);

  \node (g1) at (10, 1) {$\Z$};
  \node (g2) at (11, 2) {$0$};
  \node (g3) at (12, 1) {$\Z$};

  \draw[->] (g1) to node [midway, above, sloped] {$\inducedhom{f}$} (g2);
  \draw[->] (g2) to node [midway, above, sloped] {$\inducedhom{g}$} (g3);
  \draw[->] (g1) to node [midway, below, sloped] {$\id$} (g3);

  \draw[->, bend left] (0.5, 0.5) to node [midway, above, sloped] {$f$} (2, 1.75);
  \draw[->, bend left] (5, 1.75) to node [midway, above, sloped] {$g$} (6.5, 0.5);
  \draw[->] (2, 0) to node [midway, below, sloped] {$\id$} (5, 0);

  \draw[->, blue,
    line join=round,
    decorate, decoration={
        zigzag,
        segment length=7,
        amplitude=1.3,post=lineto,
        post length=2pt
    }] (7, 1.5) to node [midway, above] {$H_1 \circ \dFl$} (9.5, 1.5);

\end{tikzpicture}
  \caption{A commuting diagram in $\Dgr$ which illustrates why $\dFl$ cannot be made into a functor $\Dgr\to\Ch$.
  For definitions of the digraph maps, see the proof of Theorem~\ref{thm:dflag_func_failure}}%
  \label{fig:dflag_functoriality}
\end{figure}

\begin{theorem}\label{thm:dflag_func_failure}
The directed flag complex cannot be made into a functor $\dFl:\Dgr \to \Ch$.
\end{theorem}
\begin{proof}
Suppose such a functor exists.
Consider the three digraphs illustrated in Figure~\ref{fig:dflag_functoriality}.
The digraph map $f$ is given by the obvious inclusion, whilst $g$ maps $v_1$ and $v_2$ to themselves and $v_0\mapsto v_1$.
Note that $g\circ f$ composes to the identity and hence the triangle of digraph maps commutes.
Applying $H_1 \circ \dFl$ to this diagram we find that $\inducedhom{g}\circ\inducedhom{f}$ is the identity on $\Z$.
However, $H_1(\dFl(H))$ is the trivial vector space because $H$ has two ordered $2$-simplices which fit together to form a hemi-sphere.
The identity cannot factor through $0$ and hence we have a contradiction.
\end{proof}

In order to see why functoriality fails, consider the following obvious guess at the induced map.

\begin{defin}
Given a digraph map $f: G \to H$, the \emph{induced map} $\inducedch{f}:\dFl_k(G) \to \dFl_k(H)$ 
is given on each directed $k$-clique $v_0 \dots v_k$ in $G$ by
\begin{equation}\label{eq:induced_map_dfl}
  \inducedch{f} (v_0 \dots v_k) \defeq
  \begin{cases}
    f(v_0) \dots f(v_k) & \text{if all the }f(v_i)\text{ are distinct},\\
    0 & \text{otherwise}.
  \end{cases}
\end{equation}
We extend linearly to obtain a linear map $\inducedch{f}:\dFl_k(G) \to \dFl_k(H)$.
\end{defin}

This does not necessarily constitute a chain map $\dFl(G) \to \dFl(H)$.
To see why, consider the map $g$ defined in the proof of Theorem~\ref{thm:dflag_func_failure} and illustrated in Figure~\ref{fig:dflag_functoriality}.
The directed clique $(v_0, v_2, v_1)$ is mapped to $(v_1, v_2, v_1)$ which is a double edge and not a clique; this leads to a boundary being mapped to a non-trivial homological cycle.
However, such mappings (sending directed triangles to double edges) are the only obstruction to functoriality.

\begin{defin}
  \begin{enumerate}[label=(\alph*)]
    \item A digraph map $f:G \to H$ is called \mdf{triangle-collapsing} if whenever $ijk$ is a directed $2$-clique and $f(i) = f(k)$ then $f(j) = f(i) = f(k)$.
    \item We denote the category of all digraphs with triangle-collapsing digraph maps as $\mdf{\Tricol\Dgr}$.
  \end{enumerate}
\end{defin}

\begin{prop}\label{prop:dfl_functor}
 The directed flag complex with induced maps as in formula (\ref{eq:induced_map_dfl}) is a functor $\dFl:\Tricol\Dgr \to \Ch$.
\end{prop}
\begin{proof}
We first note that formula~(\ref{eq:induced_map_dfl}) certainly respects the composition and identity axioms required for functoriality; it remains only to confirm that $\inducedch{f}$ is a chain map $\dFl(G) \to \dFl(H)$.
That is, we need to check $\inducedch{f}$ commutes with the boundary map.
We verify this on the basis of $\dFl_k(G)$.

Choose some directed $(k+1)$-clique $(v_0, \dots, v_k)$ in $G$.
If all the $f(v_i)$ are distinct then clearly $\inducedch{f}\bd_n (v_0 \dots v_k) = \bd_n \inducedch{f} (v_0 \dots v_k)$.
Otherwise $f(v_{i_1}) = f(v_{i_2})$ for some $i_1 < i_2$.
Then
\begin{align}
  \bd_k \inducedch{f} (v_0 \dots v_k) &= \bd_k(0) = 0 \\
  \inducedch{f} \bd_k (v_0 \dots v_k) &=
  (-1)^{i_1} \inducedch{f} \left[ v_0 \dots \hat{v}_{i_1} \dots v_k \right] + 
  (-1)^{i_2} \inducedch{f} \left[ v_0 \dots \hat{v}_{i_2} \dots v_k \right] \label{eq:fbd_dfl}
\end{align}
If $i_2 = i_1 + 1$ then the two terms of (\ref{eq:fbd_dfl}) are equal but opposite sign so $\inducedch{f}\bd_k(v_0 \dots v_k) = 0$.
Else there is some $j$ such that $i_1 < j < i_2$.
Since $v_0 \dots v_k$ is a clique, $v_{i_1} v_j v_{i_2}$ is also a clique.
Then the triangle-collapsing condition requires that $f(v_j) = f(v_{i_1}) = f(v_{i_2})$ and hence both summands in (\ref{eq:fbd_dfl}) are $0$.
\end{proof}

\begin{rem}
    If $f:G \to H$ is an inclusion map, injective as a vertex map or $H$ is an oriented graph then $f$ is triangle-collapsing.
\end{rem}

\subsection{Grounded pipeline}

An immediate corollary of Proposition~\ref{prop:dfl_functor} is that $\dFl$ is a functor $\Incl\Dgr\to\Ch$.
This was the minimum condition that we needed to define the grounded pipeline (see Lemma~\ref{lem:zb_pers_complex}), and so we get a map on objects
\begin{equation}
 \zb{C}_F : \Obj(\WDgr) \to \Obj(\Funct{\Rposet}{\Ch}).
\end{equation}
Recall that, since $F=F_d$ and $C=\dFl$, at $t\in\R$ the chain complex $\zb{C}_F(G, t)$ is
\begin{figure}[H]
    \centering
    \begin{tikzcd}[row sep=small, column sep=small]
        \cdots \dFl_3(G^t) \arrow[r, "\partial_3"] &
        \dFl_2(G^t) \arrow[rr, "\inducedch{\iota}\circ\bd_2"] \arrow[rd, "\bd_2"', dotted] & &
        \dFl_1(G\cup G^t) \arrow[rr, "\bd_1"] & &
        \dFl_0(G\cup G^t) \cdots \\
       & & \dFl_1(G^t)\arrow[ru, "\inducedch{\iota}"', dotted, hook]
       \arrow[rr, dotted, "\bd_1"']
       & & \dFl_0(G^t)\arrow[ru, dotted, "\inducedch{\iota}"', hook] &
    \end{tikzcd}
\end{figure}

However since $\dFl$ is not a functor $\Dgr \to \Ch$, we cannot apply Theorem~\ref{thm:grd_functoriality} to obtain a functor $\zb{C}_F:\Cont\WDgr \to \Funct{\Rposet}{\Ch}$.
Instead, we must restrict the morphisms on $\WDgr$ so that they induce triangle-collapsing morphisms between the relevant digraphs at every step of the persistent chain complex.

\begin{defin}
  \begin{enumerate}[label=(\alph*)]
    \item Given two weighted digraphs $G, H$, 
      a vertex map $f: V(G) \to V(H)$ is called \mdf{path-collapsing}
      if whenever $f(i) = f(k)$ and there is some vertex $j$ with paths $i \leadsto j \leadsto k$ then $f(i) = f(j) = f(k)$.
    \item We let $\mdf{\Pathcol\Cont\WDgr}$ denote the category of weighted digraphs where morphisms are path-collapsing, contracting digraph maps.
  \end{enumerate}
\end{defin}

\begin{rem}
    Both $\Incl\Cont\WDgr$ and $\Cont\WDag$ are subcategories of $\Pathcol\Cont\WDgr$.
\end{rem}

One can check that if $f:G\to H$ is path-collapsing and contracting, then the underlying vertex map induces triangle-collapsing digraph maps $G \to H$, $G^t \to H^t$ and $G\cup G^t \to H \cup H^t$ for every $t$.
Then, using the functoriality established in Proposition~\ref{prop:dfl_functor}, the proof of Theorem~\ref{thm:grd_functoriality} goes through to show the following.

\begin{theorem}\label{thm:grd_functoriality_dflag}
Fixing $F=F_d$ and $C=\dFl$,
$\zb{C}_F$ is a functor $\zb{C}_F:\Pathcol\Cont\WDgr \to \Funct{\Rposet}{\Ch}$.
\end{theorem}

\begin{definthm}\label{defin:grpdflh}
  \mdf{Grounded persistent directed flag homology (\grpdflh)} is the functor
  \begin{equation}
 \mdf{\zbpershomdflagmap} \defeq \Funct{\Rposet}{H_1} \circ \zb{C}_F : \Pathcol\Cont\WDgr \to \PersVec.
  \end{equation}
\end{definthm}

We reuse the notation of Definition~\ref{notation:grd_notation}, to denote the space of grounded cycles, boundaries and homology in this directed flag complex setting.

\subsection{Interpretation}
Section~\ref{sec:decreasing_curves} established some basic properties of the descriptor.
In particular, all features are born at $t=0$ and the number of features in the barcode, $\card{\thedescdfl{G}}$, is the circuit rank of the underlying, undirected graph, $\underlying{G}$.
These results do not rely on the functoriality of $\zb{C}_F$.
Indeed, they only require that there are generators corresponding to directed triangles in $C_2(G)$, which is certainly the case for $C=\dFl$.
Therefore, all results of Section~\ref{sec:decreasing_curves} apply to the directed flag complex.

In contrast, Lemma~\ref{lem:death_bound}, bounding the death-time of a given circuit, requires the existence of generators corresponding to long squares.
Therefore the proof does not work in the directed flag setting.
However, we can prove the following, slightly weaker bound.

\begin{lemma}\label{lem:death_bound_dflag}
Given $G=(V, E, w)\in\WDgr$ and two directed paths $p_1, p_2: a \leadsto b$ between distinct vertices $a, b \in V$, let $p_c$ denote the undirected circuit which traverses $p_1$ forwards and then $p_2$ in reverse.
For $i= 1, 2$, define
\begin{equation}
 h_i \defeq \min
 \left\{ 
   t \geq 0 \rmv
   \exists v_i \in V\text{ along } p_i \text{ such that } d(a, v_i) \leq t \text{ and } d(v_i, b) \leq t
 \right\}.
\end{equation}

Then $\death(p_c) \leq \max( h_1, h_2, d(a, b) )$.
\end{lemma}
\begin{proof}
The proof is similar to that of Lemma~\ref{lem:death_bound}.
However, since there are no generators corresponding to long squares, we need $T$ to be sufficiently large that the edge $(a, b)$ appears in $G^T$.
Therefore $av_1b$ and $av_2 b$ are both generators of $\zb{C}_2(G, T)$ and so certainly $av_1 b - av_2 b \in \zb{C}_2(G, T)$.
\end{proof}

Theorem~\ref{thm:rep_of_cycles}, which guarantees the existence of a persistence basis of undirected circuit, also applies to the directed flag complex.
The proof is unchanged except that each of the generators $u_1, \dots, u_n$ must be directed triangles since these are the generators of $C_2(G^t)$.

Finally, the disjoint union decomposition theorem (Theorem~\ref{thm:union_decomp}) applies unchanged to the directed flag complex.
Unfortunately, the proof of wedge the wedge decomposition theorem (Theorem~\ref{thm:wedge_decomp}) fails in the directed flag case, because the contractions $f_i$ are not necessarily path-collapsing.
However, if we restrict to $G\in\WDag$ then the proof goes through unchanged because $\Cont\WDag \subseteq \Pathcol\Cont\WDgr$.

Moreover, it is not just the proof of Theorem~\ref{thm:wedge_decomp} which fails but indeed the statement itself.
Note that, if Theorem~\ref{thm:wedge_decomp} were to hold in the directed flag case then Corollary~\ref{thm:edge_deletion_stability} (on stability to separating edge deletion) would hold automatically too.

\begin{figure}[hptb]
  \centering
  \begin{tikzpicture}[
  roundnode/.style={circle, fill=black, minimum size=4pt},
	squarenode/.style={fill=black, minimum size=4pt},
	inner sep=2pt,
	outer sep=1pt
  ]

  \node (a) at (0, 0) [roundnode, label=above:$v_1$] {};
  \node (b) at (1, 1) [roundnode, label=above:$v_2$] {};
  \node (c) at (2, 0) [roundnode, label=right:$v_3$] {};
  \node (d) at (1, -1) [roundnode, label=below:$v_4$] {};
  \node (e) at (-1, 0) [roundnode, label=above:$v_0$] {};

  \draw[red!100!blue, ->] (a) -- (b)node[midway, above, sloped] {\tiny $10$};
  \draw[red!100!blue, ->] (b) -- (c)node[midway, above, sloped] {\tiny $10$};
  \draw[red!100!blue, ->] (c) -- (d)node[midway, below, sloped] {\tiny $10$};
  \draw[red!100!blue, ->] (d) -- (a)node[midway, below, sloped] {\tiny $10$};
  \draw[red!10!blue, ->] (e) -- (a) node[midway, above, sloped] {\tiny $1$};
  
  \node[] at (3, 0) {};
  \node[] at (1, -2) {$G$};

\end{tikzpicture}
  \caption{
    An example weighted digraph for which the wedge decomposition result fails, when using the directed flag complex.
  }\label{fig:decomposition_dfl_instability}
\end{figure}
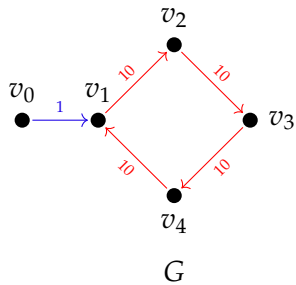

\begin{example}\label{ex:dflag_appendage_instability}
  Consider Figure~\ref{fig:decomposition_dfl_instability} and denote $e\defeq (v_0, v_1)$.
  Note that $e$ is a separating edge and hence, by Corollary~\ref{thm:edge_deletion_stability},
  $\zbpershom{G}\cong\zbpershom{\wdop{d}{e}{G}}$ and we can compute $\thedesc{G} = \ms{ [ 0, 20) }$.
  However, using the directed flag complex we observe
  \begin{equation}
    \thedescdfl{G} = \ms{ [0, 21) }
    \quad\text{ but }\quad
    \thedescdfl{\wdop{d}{e}{G}} = \ms{ [ 0, 30) }.
  \end{equation}
  The sole cycle in $G$ dies at $t=21$ because
  \begin{equation}
    \zb{\bd}_2(v_1 v_2 v_3 + v_3 v_4 v_1 + v_0 v_1 v_3 + v_0 v_3 v_1) = v_1 v_2 + v_2 v_3 + v_3 v_4 + v_4 v_1
  \end{equation}
  and the corresponding directed triangles appear in $G^{21}$.
\end{example}

This example illustrates an instability of $\zbpershomdflagmap$ which is not present in $\zbpershommap$.
Moreover, since $e$ is not involved in any simple undirected circuits, we would not expect its presence to affect $\zbpershomdflag{G}$; this instability complicates the interpretation of $\zbpershomdflagmap$.

\subsection{Stability analysis}

The key results used in the proof of most stability theorems were Lemmas~\ref{lem:shift_induce_comp} and~\ref{lem:shift_induce_cont}, in which a $\delta$-shifting vertex map $f:V(G) \to V(H)$ was used to construct part of an interleaving.
In order to repeat this construction in the directed flag setting, we need to ensure that $f$ is path-collapsing.

\begin{lemma}
Any $\delta$-shifting, path-collapsing vertex map $f:V(G) \to V(H)$ induces a morphism
\begin{equation}
 \shiftinduced{ch}{f}{\delta}:\zb{C}(G) \to \shifted{\zb{C}(H)}{\delta}.
\end{equation}
Given another $\epsilon$-shifting, path-collapsing vertex map $g:V(H) \to V(K)$,
\begin{equation}
 \shiftinduced{ch}{g\circ f}{\epsilon + \delta} = \shiftinduced{ch}{g}{\epsilon} \circ \shiftinduced{ch}{f}{\delta}.
\end{equation}
Moreover, if $f$ is $0$-shifting then $\shiftinduced{ch}{f}{0} = \zbdiginduc{ch}{f}$.
\end{lemma}

\begin{lemma}
A $\delta$-shifting, path-collapsing vertex map is a $\delta'$-shifting, path-collapsing vertex map for any $\delta' \geq \delta$ and
\begin{equation}
 \shiftinduced{ch}{f}{\delta'} = \shifted{\transmorph{\zb{C}(H)}{\delta' - \delta}}{\delta} \circ \shiftinduced{ch}{f}{\delta}.
\end{equation}
\end{lemma}

With these lemmas, many theorems from Section~\ref{sec:stability} go through unchanged
ether because they only involve $\delta$-shifting maps which are inclusions or because they use explicit counter-examples which have the same barcodes in this new setting.
For the other results, since $\Cont\WDag\subseteq \Pathcol\Cont\WDgr$, the proofs work unchanged so long as we restrict to $G\in\WDag$.
To summarise this, we present Table~\ref{tbl:dflag_stability_result_summary}, which is a reproduction of Table~\ref{tbl:stability_result_summary} with additional annotations.
Of particular note, we emphasise that weight perturbation stability holds unrestricted but edge subdivision stability (and hence convergence under iterated medial subdivision) only holds if $G\in\WDag$.

\begin{table}[hbtp]
\begin{center}
\ars{1.2}
\tcs{0.7\tabcolsep}
\renewcommand\theadfont{\bfseries}
\newcommand*\partialthm{${}^\blacklozenge$}
\newcommand*\unrestrictedthm{${}^\checkmark$}
\newcommand*\restrictedthm{${}^\restriction$}
\begin{tabular}{ l | c c c c }
  \thead{Operation} & \thead{Locally\\Stable} & \thead{Non-locally\\Stable} & \thead{Locally\\Unstable}  & \thead{Isomorphism} \\ \hline
  Weight perturbation & Theorem~\ref{thm:pertub_stability}\unrestrictedthm & & & \\
  Edge subdivision & Theorem~\ref{thm:subdiv_stability}\restrictedthm & & & \\
  Edge collapse & Theorem~\ref{thm:collapse_stability}\partialthm\restrictedthm &
                & Theorem~\ref{thm:collapse_instability}\unrestrictedthm & \\
  Edge deletion &
  Corollary~\ref{cor:local_edge_deletion}\partialthm\unrestrictedthm &
  Theorem~\ref{thm:edge_deletion_stability}\unrestrictedthm &
  Theorem~\ref{thm:edge_deletion_instability}\unrestrictedthm &
    Theorem~\ref{cor:edge_deletion_disconnect}\partialthm\restrictedthm
  \\
  Vertex deletion & & & Corollary~\ref{cor:vertex_deletion_instability}\unrestrictedthm & Corollary~\ref{cor:isolated_vertex_deletion}\partialthm\unrestrictedthm
\end{tabular}
\caption{
  Stability and instability theorems for $\zbpershomdflagmap$, under various digraph operations.
  {$\blacklozenge$} Denotes a theorem which only applies to a subset of such operations.
  {$\checkmark$} Denotes a theorem which applies to the directed flag complex unrestricted.
  {$\restriction$} Denotes a theorem which applies to the directed flag complex after restricting to $G\in\WDag$.
}\label{tbl:dflag_stability_result_summary}
\end{center}
\end{table}

\subsection{Iterated medial subdivision}
Recall Proposition~\ref{prop:asymptotic_example}, in which we found the limiting barcode of a simple cycle graph, under iterated medial subdivision.
This limiting value does not hold in the directed flag setting.

\begin{figure}[htbp]
  \centering
  \begin{tikzpicture}[
  roundnode/.style={circle, fill=black, minimum size=4pt},
	squarenode/.style={fill=black, minimum size=4pt},
	inner sep=2pt,
	outer sep=1pt
  ]
  \node (a) at (0, 0) [roundnode, label=left:$v_0$] {};
  \node (b) at (1, 1) [roundnode, label=above:$v_1$] {};
  \node (c) at (2, 0) [roundnode, label=right:$v_2$] {};

  \node (a2) at (4, 0) [roundnode, label=left:$v_0$] {};
  \node (b2) at (5, 2) [roundnode, label=above:$v_1$] {};
  \node (c2) at (6, 0) [roundnode, label=right:$v_2$] {};

  \draw[->] (a)--(b) node[midway, above, sloped] {\tiny $2$};
  \draw[->] (b)--(c) node[midway, above, sloped] {\tiny $2$};
  \draw[->, red!50!blue] (a)--(c) node[midway, below, sloped] {\tiny $3$};

  \draw[->,red] (a2)--(b2) node[midway, above, sloped] {\tiny $5$};
  \draw[->,red] (b2)--(c2) node[midway, above, sloped] {\tiny $5$};
  \draw[->] (a2)--(c2) node[midway, below, sloped] {\tiny $2$};

  \node[] at (1, -0.5) {$G$};
  \node[] at (5, -0.5) {$H$};
\end{tikzpicture}
  \caption{Two weighted digraphs, of the form considered by Proposition~\ref{prop:asymptotic_example}, illustrating that the behaviour of $\thedescdfl{G}$ under iterated subdivision is not as symmetric.}
  \label{fig:subdivision_dflag}
\end{figure}
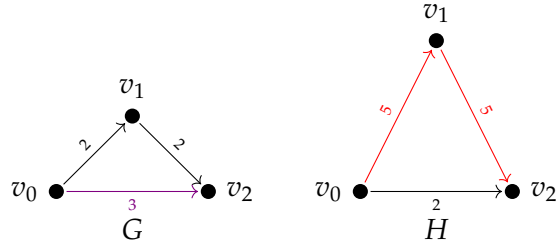

\begin{example}
Consider the two digraphs pictured in Figure~\ref{fig:subdivision_dflag}.
In $G$, the two paths have lengths $l_{G,1} = 4 > 3 = l_{G,2}$, while in $H$ the two paths have lengths $l_{H,1} = 10 > 2 = l_{H,2}$.
When using directed flag complex,
\begin{align}
  \lim_{n\to\infty}\thedescdfl{G_n} &= \ms{[0, 3)} = \ms{[0, l_{G, 2})}, \\
  \lim_{n\to\infty}\thedescdfl{H_n} &= \ms{[0, 5)} = \ms{[0, 1/2 \cdot l_{H,1})}.
\end{align}
The reason this limiting value differs from path homology is because edge $(v_0, v_2)$ must appear in $G^t$ before the sole feature can die.
\end{example}

By a similar method to Proposition~\ref{prop:asymptotic_example}, one can show the following.

\begin{prop}\label{prop:dflag_asymptotic_example}
  Suppose $G\in\WDag$ is the union two directed paths $p_1, p_2$ from a source to a sink, with lengths $l_1 \geq l_2$ respectively.
  Recall the definition of iterated medial subdivision (Definition~\ref{defin:ims}).
  Then
  \begin{equation}
    \lim_{n\to\infty}\thedescdfl{\IMS{G}{n}}
    =
    \ms{
      \left[0, \max\left(\frac{1}{2}l_1, l_2 \right) \right)
    }
  \end{equation}
\end{prop}

The asymmetry in this limiting descriptor arises from asymmetry in the path lengths of the directed triangle motif, in contract to the long square motif.
In contrast, path homology yields a simpler interpretation; the size of the limiting feature is directly proportional to the length of the longer path.

\printbibliography%
\end{document}